\documentclass[10pt]{article}

\usepackage[utf8x]{inputenc}
\usepackage{amsfonts}
\usepackage{amsmath}\usepackage{amssymb}
\usepackage{amsthm}
\usepackage{mathrsfs}
\usepackage{latexsym}
\usepackage{graphicx}
\usepackage{bm}
\usepackage{inputenc}
\usepackage[shortlabels]{enumitem}

\usepackage{hyperref}
\usepackage{graphicx}
\usepackage[a4paper,bindingoffset=0.2in,%
left=1.0in,right=1.0in,top=1.2in,bottom=1.2in,%
footskip=.35in]{geometry}
\usepackage{imakeidx}
\makeindex[title=Notation Index]

\allowdisplaybreaks

\newtheorem{Lemma}{Lemma}[section]
\newtheorem{Corollary}[Lemma]{Corollary}
\newtheorem{Theorem}[Lemma]{Theorem}

\newtheorem{theorem}{Theorem}

\theoremstyle{definition}

\newtheorem{Definition}[Lemma]{Definition}

\theoremstyle{remark}

\newtheorem{Remark}[Lemma]{Remark}
\newtheoremstyle{citing}
{3pt}
{3pt}
{\bfshape}
{}
{\itseries}
{.}
{.5em}
{\thmnote{#3}}
\theoremstyle{citing}

\DeclareMathOperator{\restrict}{\llcorner}

\DeclareMathOperator{\loc}{loc}
\DeclareMathOperator{\Tan}{Tan}
\DeclareMathOperator{\diam}{diam}
\DeclareMathOperator{\im}{im}
\DeclareMathOperator{\Lip}{Lip}
\DeclareMathOperator{\Unp}{Unp}
\DeclareMathOperator{\dmn}{dmn}
\DeclareMathOperator{\Nor}{Nor}
\DeclareMathOperator{\Hom}{Hom}
\DeclareMathOperator{\Der}{D}

 \DeclareMathOperator{\dist}{dist}
\DeclareMathOperator{\trace}{trace}
\DeclareMathOperator{\Clos}{clos}
\DeclareMathOperator{\ap}{ap}
\DeclareMathOperator{\lin}{lin}
\DeclareMathOperator{\Cut}{Cut}
\DeclareMathOperator{\reach}{reach}
\DeclareMathOperator{\Id}{Id}
\DeclareRobustCommand{\rchi}{{\mathpalette\irchi\relax}}
\newcommand{\irchi}[2]{\raisebox{\depth}{$#1\chi$}}
\DeclareMathOperator{\Int}{int}
\DeclareMathOperator{\graph}{graph}
\DeclareMathOperator{\epi}{epi}

\DeclareMathOperator{\conv}{conv}
\DeclareMathOperator{\Dis}{Dis}
\DeclareMathOperator{\ind}{\bm{1}}

\newcommand{\ar}[2]{\boldsymbol{r}_{#1}^{#2}}
\newcommand{\arl}[2]{\underline{\ar{#1}{#2}}}
\newcommand{\R}{\mathbf{R}}
\newcommand{\bS}{\mathbf{S}}
\newcommand{\N}{\mathbb{N}}
\usepackage{color}

\title{Curvature measures and soap bubbles beyond convexity}
\author{Daniel Hug \and Mario Santilli}
\date{}

\begin{document}
	\maketitle
	
\begin{abstract}
Extending the celebrated results of Alexandrov (1958) and Korevaar--Ros (1988) for smooth sets, as well as the results of Schneider (1979) and the first author (1999) for arbitrary convex bodies, we obtain for the first time the characterization of the isoperimetric sets of a uniformly convex  smooth finite-dimensional normed space (i.e.~Wulff shapes) in the non-smooth and non-convex setting, based on a natural geometric condition involving the curvature measures. More specifically we show, under a weak mean convexity assumption,  that finite unions of disjoint Wulff shapes are the only sets of positive reach $ A \subseteq \mathbf{R}^{n+1} $ with finite and positive volume such that, for some $ k \in \{0, \ldots , n-1\}$, the  $ k $-th generalized curvature measure $ \Theta^\phi_{k}(A, \cdot) $, which is defined on the unit normal bundle of $ A $ with respect to the relative geometry induced by $ \phi $, is proportional to the top order generalized curvature measure $ \Theta^\phi_{n}(A, \cdot)$. If $ k = n-1 $ the conclusion  holds for all sets of positive reach with finite and positive volume. We also prove a related sharp result about the removability of the singularities. This result is based on the extension of the notion of a normal boundary point, originally introduced by Busemann and Feller (1936) for arbitrary convex bodies, to  sets of positive reach.

These findings are new even in the Euclidean space.

Several auxiliary and related results are proved, which are of independent interest. They include the extension of the classical Steiner--Weyl tube formula to arbitrary closed sets in a finite dimensional uniformly convex normed vector space, a general formula for the derivative of the localized volume function, which extends and complements recent results of Chambolle--Lussardi--Villa (2021), and general versions of the Heintze--Karcher inequality.
\end{abstract}
	
	\paragraph{\small MSC-classes 2020.}{\small 49Q10, 52A30, 53C45, 53C60, 53C65, 28A75, 26B25.}
\paragraph{\small Keywords.}{\small Tube formula, curvature measures, sets of positive reach, Wulff shapes, removable singularities}
	\tableofcontents
	
	\section{Introduction}
	
The following result is fundamental in the geometry of submanifolds: if a
smooth hypersurface in a Euclidean space  encloses a bounded domain and one of its mean curvature functions is constant, then it is a Euclidean sphere. We refer to this statement as the \emph{soap bubble theorem}. We remark that we are considering \emph{only} hypersurfaces that \emph{enclose} a domain (i.e.~they are embedded);
otherwise (i.e.~for immersed hypersurfaces) it is well known that such a uniqueness result is in general true only in very special situations; see \cite{MR707850} and \cite{MR815044}. The aforementioned fundamental result was proved by Alexandrov for the mean curvature function in \cite{MR0102114} and by Korevaar and Ros for the higher-order mean curvature functions in \cite{MR925120} and \cite{MR996826}. Several other proofs were found  earlier under various additional hypotheses (e.g.\ convexity, star-shapedness, mean convexity type assumptions) starting from the pioneering work of Jellett in \cite{Jellett1853} in the nineteenth century; see \cite{MR68236} (and the references therein). An analogous result is true for hypersurfaces embedded in finite dimensional uniformly convex normed spaces,  provided that  the Euclidean mean curvature functions and the Euclidean sphere are replaced by their anisotropic counterparts (in particular the role of the sphere is played by the Wulff shape); see \cite{MR2514391}.

A key feature of all the results mentioned so far is that they hold for \emph{smooth} hypersurfaces. In fact, since these results are about hypersurfaces with constant mean curvature functions, the smoothness hypothesis may appear to be natural and somehow unavoidable. However, considering different but equivalent points of view, it turns out that the soap bubble theorem is only part of a more general and more natural problem that does not require any a-priori smoothness assumption. There are at least two standard ways to adjust the soap bubble theorem: via the variational approach based on the notion of a critical point of the isoperimetric functional, and via the integral-geometric approach based on the notion of curvature measures.

Let us first briefly describe the variational approach. A standard computation shows that if a smooth hypersurface with constant mean curvature encloses a bounded domain $ \Omega $, then $ \Omega $ is a critical point of the Euclidean isoperimetric functional among all sets of finite perimeter. Therefore, the Alexandrov theorem can be equivalently stated saying that a critical point of the isoperimetric functional is a sphere, provided it has a smooth boundary. The same is true for the anisotropic counterpart studied in \cite{MR2514391}, if a suitable anisotropic isoperimetric functional is considered. From this point of view the assumption of smoothness appears to be a possibly convenient condition rather than a necessary  restriction. In fact, one can ask whether it is true that all critical points of the Euclidean  isoperimetric functional are Euclidean balls. It should be remarked that the regularity theory in geometric measure theory does not imply that a critical point is automatically smooth. Hence, in the non-smooth framework the problem genuinely involves hypersurfaces which a priori may have severe singularities. A positive resolution of this type of problem is given in \cite{MR3921314} for the Euclidean case and in \cite{MR4160798} in an anisotropic setting, under some additional hypotheses.

We now focus on the integral-geometric approach, which is the one adopted in the present work. The Weyl tube formula asserts that for all sufficiently small radii $ \rho > 0 $ the volume of a tubular neighbourhood $ C_\rho $ around  a domain $ C \subseteq \mathbf{R}^{n+1} $ with $ \mathcal{C}^2 $-boundary $\partial C$ is a polynomial function in $ \rho $; in other words,
\begin{equation*}
\mathcal{L}^{n+1}(C_\rho \setminus
C) = \sum_{j=0}^{n} \frac{\rho^{j+1}}{j+1}c_{n-j}   \qquad \textrm{for all sufficiently small $ \rho > 0 $.}
\end{equation*}
 The coefficients $ c_k $ can be obtained integrating over $ \partial C $ (or, equivalently, over the unit normal bundle of $ C $) the $ (n-k) $-th mean curvature function of $ \partial C $ (with respect to the exterior normal map) for $ k \in \{0, \ldots , n\} $. In the special case of smooth convex domains or of convex polytopes in $\R^d$ with $d\in\{2,3\}$, this formula has already been found in the nineteenth century by Steiner. Now if $ k \in \{1, \ldots , n\} $, then the $ k $-th curvature measure  of $ C $ is defined as the Borel measure obtained by integrating over a given Borel subset of $ \partial C $ the $ (n-k) $-th mean curvature function of $ \partial C$. It has been Federer's fundamental discovery in \cite{MR0110078}  that the existence of the curvature measures and the validity of the polynomial Weyl tube formula are independent of the smoothness hypothesis. In the same seminal paper, Federer laid the foundation of a theory of \emph{sets of positive reach}, a class that includes all convex bodies, all embedded $ \mathcal{C}^2 $-submanifolds and much more; indeed the boundary of a domain with positive reach need not even be a topological manifold (see the example described at the end of this section).  The soap bubble theorem can be equivalently stated saying that if $ C $ is a bounded domain with smooth boundary  such that one of the curvature measures of $C$ is a multiple of the area measure associated to the  boundary of $C$, then $ C $ is a Euclidean ball. At this point the most compelling problem is to establish a corresponding uniqueness result without any smoothness hypothesis. In fact, this is a classical task in convex geometry. For an arbitrary convex body and for the $ 0 $-th mean curvature measure, Diskant accomplished this task and even established a sharp stability result in \cite{MR0239541}. A decade later, Schneider  \cite{MR522031}  resolved the problem for all curvature measures associated with a general convex body  (see also \cite[Theorem 8.5.7]{MR3155183}). Different approaches to prove and generalize Schneider's theorem were found by Kohlmann in \cite{MR1621971} and  \cite{MR1604003}. The extension to the anisotropic setting of the results of Schneider and Kohlmann can be found in \cite{Hug99}. On the other hand, as far as we are aware of, up to now no results are available for arbitrary sets of positive reach, and thus the problem has remained unexplored in the non-convex and non-smooth setting. The main goal of this paper is to address this problem in full generality.  Our main theorems, Theorem \ref{theo: Alexandrov}, Theorem \ref{theo: Alexandrovmodifiedb}  and Corollary \ref{Cor: Alexandrov} extend Alexandrov's theorem to all sets of positive reach  and extend its higher-order version, considered by Korevaar and Ros in the smooth setting, to sets of positive reach under a natural mean convexity hypothesis. Actually we treat this problem directly in the more general setting of uniformly convex finite-dimensional normed spaces, hence generalizing also the main result in \cite{MR2514391} to arbitrary sets of positive reach.

A central notion of this paper are the generalized curvature measures of a set of positive reach in the relative (Minkowski) geometry induced by a uniformly convex $ \mathcal{C}^2$-norm $ \phi $ in $\R^{n+1}$.     If $ \phi $ is the Euclidean norm, then  Federer's tube formula for sets of positive reach allows to introduce the Euclidean curvature measures (see \cite{MR0110078} and \cite{MR0849863}). On the other hand, for non-Euclidean norms a general tube formula for non-smooth and non-convex sets was missing so far. In the recent paper \cite{MR3458179}, an attempt to obtain an anisotropic tube formula for arbitrary sets of positive reach has been impeded by difficulties to obtain Lipschitz estimates for the nearest point projection (which in the Euclidean setting were established by Federer in \cite{MR0110078}); see the  comments after Theorem 1.1 in  \cite[p.~472]{MR3458179}.  Given this premise, the first task in this paper is to lay the foundation of the theory of curvature measures in the anisotropic setting for sets of positive reach and, more generally, for arbitrary closed sets. For this purpose, let $ \bm{\delta}^\phi_A $ and $ \bm{\nu}^\phi_A $ be the distance function and the Cahn--Hoffman map of $ A $ with respect to the metric induced by the conjugate norm $ \phi^\ast $ of $ \phi $ (see equations \eqref{eq: distance function} and  \eqref{eq: nu function} below). Then we define the $ \phi $-unit normal bundle $ N^\phi(A) $ of $ A $ by
\begin{equation*}
    N^\phi(A) = \{ (a, \eta) : a \in A,\, \eta \in \partial \mathcal{W}^\phi, \, \bm{\delta}^\phi_A(a + r \eta) = r \; \textrm{for some $ r > 0 $}\},
\end{equation*}
where $ \mathcal{W}^\phi = \{\eta \in \mathbf{R}^{n+1}: \phi^\ast(\eta) \le  1\}$ is called the Wulff shape of $ \phi $. In general, $ \bm{\nu}^\phi_A $ is a multivalued map and $ N^\phi(A) $ is a countably $ (\mathcal{H}^n, n) $ rectifiable subset of $ A \times \partial \mathcal{W}^\phi$. Employing recent results on the Lipschitz and differentiability properties of $ \bm{\nu}^\phi_A $ provided in \cite{kolasinski2021regularity}, we introduce the $ \phi $-principal curvatures
\begin{equation*}
    - \infty < \kappa^\phi_{A,1}(a, \eta) \leq \ldots \leq \kappa^\phi_{A,n}(a, \eta) \leq + \infty
\end{equation*}
of $ A $  at $ \mathcal{H}^n $ a.e.\ $ (a, \eta) \in N^\phi(A) $, similarly as in the smooth or convex case, as follows: if $ \chi^\phi_{A,1} (a+r\eta)\leq \ldots \leq \chi^\phi_{A,n}(a+r\eta) $ denote the eigenvalues of $ \Der \bm{\nu}^\phi_{A}(a+r\eta) $, then we define
 \begin{equation*}
     \kappa^\phi_{A,i}(a, \eta) = \frac{\chi^\phi_{A,i}(a+ r\eta)}{1 - r \chi^\phi_{A,i}(a + r \eta)} \in (-\infty, +\infty],
 \end{equation*}
where the right-hand side is independent of $ r $, if $ r>0 $ is chosen sufficiently small, depending on $(a, \eta)$. We denote by $ \widetilde{N}^\phi(A) $  the set of points $(a, \eta) \in N^\phi(A) $ where the principal curvatures exist (hence $ \mathcal{H}^n(N^\phi(A) \setminus \widetilde{N}^\phi(A)) =0$) and define $ \widetilde{N}^\phi_d(A) $ to be the set of all $ (a, \eta) \in \widetilde{N}^\phi(A) $ such that $ \kappa^\phi_{A,d}(a, \eta) < \infty $ and $ \kappa^\phi_{A, d+1}(a, \eta) = + \infty $. In particular, $ \widetilde{N}^\phi_n(A)$ is the set of all $(a, \eta) \in N^\phi(A)$ such that the $\phi $-principal curvatures of $ A $ at $(a,\eta)$ are finite. The $ i $-th $ \phi $-mean curvature function $ \bm{H}^\phi_{A,i}(a, \eta) $ of $ A $ at $ \mathcal{H}^n $ a.e.\ $ (a, \eta) \in N^\phi(A) $ is defined by taking certain combinations of the elementary symmetric functions of the $ \phi $-principal curvatures (if $ A $ is a smooth submanifold, then we recover classical definitions); see Definition \ref{Def:finitecurv}. As a consequence of Theorem \ref{theo: Steiner closed},   the volume of the tubular neighbourhood $ B^\phi(A, \rho)\setminus A= \{x \in \mathbf{R}^{n+1}: 0<\bm{\delta}^\phi_A(x) \leq \rho \}$ of an arbitrary compact set $ A $ can be expressed by the formula
 \begin{equation*}
     \mathcal{L}^{n+1}(B^\phi(A, \rho)\setminus A) = \sum_{j=0}^n \frac{1}{j+1} \int_{N^\phi(A)}\phi(\bm{n}^\phi(\eta))\, J^\phi_A(a,\eta)\,  \inf\{ \rho, \bm{r}^\phi_A(a,\eta)  \}^{j+1}\bm{H}^\phi_{A,j}(a,\eta)\,  d\mathcal{H}^n(a,\eta).
 \end{equation*}
Here $ \bm{r}^\phi_A $ is the reach function of $ A $ (see \eqref{eq: reach function}),  $ J^\phi_A $ is a Jacobian-type function encoding the tangential properties of the normal bundle (see Definition \ref{def: jacobian}) and $ \bm{n}^\phi$ is the Euclidean exterior unit-normal of $ \mathcal{W}^\phi$ (see \eqref{eq: normal of Wulff shape}). For an arbitrary compact set $A$, the right side of the tube formula is  in general not a polynomial function in $ \rho $ (the volume growth is sub-polynomial)  and the mean curvature functions $ \bm{H}^\phi_{A,i}$ are  in general not integrable on $ N^\phi(A) $ (in spite of the fact that the integral is well defined due to the power of the reach function under the integral). On the other hand, if the $\phi $-reach of $ A $ (see Definition \ref{def: reach}) is greater than or equal to some positive threshold $ \rho_0>0 $, then $ \bm{r}^\phi_A(a, \eta) \geq \rho_0  $ for every $(a, \eta) \in N^\phi(A)$ and we obtain a polynomial-type formula
 \begin{equation*}
     \mathcal{L}^{n+1}(B^\phi(A, \rho)\setminus A) = \sum_{j=0}^n \frac{\rho^{j+1}}{j+1} \int_{N^\phi(A)}\phi(\bm{n}^\phi(\eta))\, J^\phi_A(a,\eta)\,  \bm{H}^\phi_{A,j}(a,\eta)\,  d\mathcal{H}^n(a,\eta)
 \end{equation*}
for   $ \rho\in (0, \rho_0 )$. Noting that the functions $ J^\phi_A \cdot \bm{H}^\phi_{A,i} $ are integrable on $ N^\phi(A) $ if $ A $ has positive reach,  we define the  $ m $-th generalized $ \phi $-curvature measure of $ A $ as the signed Radon measure supported on $ N^\phi(A) $ given by
\begin{equation*}
    \Theta^\phi_{m}(A, B) = \frac{1}{n-m+1}\int_{N^\phi(A) \cap B}\phi(\bm{n}^\phi(\eta))\, J^\phi_A(a,\eta)\,  \bm{H}^\phi_{A,n-m}(a,\eta)\,  d\mathcal{H}^n(a,\eta)
\end{equation*}
for every bounded Borel set $ B \subseteq \mathbf{R}^{n+1} \times \mathbf{R}^{n+1}$ and $m\in\{0,\ldots,n\}$. If $ \phi $ is the Euclidean norm, these measures coincide with the classical generalized curvature measures for sets of positive reach (up to the normalization); see \cite{MR0849863}. Moreover, if $ \phi $ is the Euclidean norm, then the curvature functions and the tube formula from Section \ref{sec: steiner formula}  agree with those in \cite{MR2031455}. However, our approach here is substantially different from the one in \cite{MR2031455}. In fact, in \cite{MR2031455} the Euclidean principal curvatures of an arbitrary closed set are introduced by means of an approximation with sets of positive reach (see Stach{\'o}'s approximation Lemma in \cite[Lemma 2.3]{MR2031455} and \cite{MR0534512}). Observe that such an approximation argument is not available in the anisotropic case, since there is no fully fledged theory of sets with positive reach ready to be used  in the current more general framework (see the discussion above on \cite{MR3458179}).

An important consequence of the tube formula for arbitrary closed sets (see also Corollary \ref{theo: change of variable}) is the sharp integral-geometric inequality in Theorem \ref{theo: heintze karcher}, for which equality is attained  only by disjoint unions of finitely many rescaled and translated Wulff shapes (assuming an a priori bound for the mean curvature). Theorem \ref{theo: heintze karcher} generalizes the geometric inequality known as the  Heintze--Karcher inequality for sufficiently smooth sets (see \cite{MR996826} and \cite{MR1173047}) to arbitrary closed sets in Euclidean space, under a natural (weak) mean convexity assumption. For sufficiently smooth sets $C$, the Heintze--Karcher inequality provides an upper bound for the volume of $C$ by the integral average over the boundary  of $C$ of the reciprocal of the mean curvature function of $C$. For convex bodies (compact convex sets with non-empty interiors) this inequality was proved in \cite{MR1604003} (in the Euclidean case) and in \cite[Lemma 2.45]{Hug99} (in the anisotropic framework). The general anisotropic version for arbitrary closed sets given in Theorem \ref{theo: heintze karcher} will be specialized to sets of positive reach in Theorem \ref{theo: heintze karcher positive reach}, and this result is one of the pillars for the subsequent soap bubble theorems.

In Section \ref{sec: steiner formula}, we also provide a detailed analysis of the curvature functions of an arbitrary closed set $ A $ in relation to  the dimension of the fibers of the $\phi$-unit normal bundle $ N^\phi(A) $ of $A$.  This analysis allows us to obtain the general disintegration formula stated in Theorem \ref{thm:disint}, which in the present generality is new even in the special Euclidean case (see \cite[Theorem 5.5]{MR1742247} for sets with positive reach  and \cite[Theorems 1.56 and 1.57]{Hug99} for convex bodies).

As another consequence of the tube formula, we obtain differentiability  properties of the parallel volume function in Section \ref{sec:three}.  In Theorem \ref{theo: derivative of volume} we determine the left and right derivatives of the localized volume function of the tubular neighbourhood around a closed set $ A $ with respect to $\phi$ and provide a novel necessary and sufficient condition for the existence of the two-sided derivative. Although this result is not needed in the proof of the soap bubble theorems, we  have decided to include it here since it is of interest in itself. Indeed the derivative of the volume function has been the subject of several investigations; see  \cite{MR442202}, \cite{MR2031455}, \cite{MR2218870}, \cite{MR2865426} and \cite{MR4329249}.   The most recent contribution \cite[Theorem 5.2]{MR4329249} treats \emph{arbitrary and possibly asymmetric} norms and establishes formulae for the left and the right derivative of the volume function in terms of \emph{area-integrals} on the boundary of the tubular neighbourhood of a compact set. Theorem 5.2 in \cite{MR4329249} can be compared with the third equality in \eqref{theo: derivative of volume:1} and in \eqref{theo: derivative of volume:2} (notice, however, that the results in \cite{MR4329249} are not localized). On the other hand, the main novelty of our results are the formulae for the left and the right derivatives of the localized volume function in terms of \emph{curvature integrals} on the $ \phi $-unit normal bundle of $ A $.  As a consequence, our result establishes the relation between the (localized) area integrals on the boundary of the tubular neighbourhood of $ A $ with the corresponding (localized)  curvature integrals on the $ \phi $-unit normal bundle of $ A $.

In Section \ref{sec: Alexandrov points}, we generalize the classical notion of a  normal boundary point of a convex body (see \cite{MR3155183} and \cite{MR3272763} and the references therein)  to arbitrary sets of positive reach. It is well known that the boundary of a convex body $ C $ is locally the epigraph of a convex function around each of its boundary points. Employing the classical theorem of Alexandrov on the twice differentiability of  convex functions, one can see that the normal boundary points are precisely those boundary points where the locally representing function is twice differentiable. Consequently, the notion of a pointwise second fundamental form and pointwise defined mean curvature functions can be introduced at each such boundary point.  Some authors refer to the normal boundary points as \emph{Alexandrov points}, and for this reason we denote the set of these points by $\mathcal{A}(C)$. If $ \bm{p}: \mathbf{R}^{n+1} \times \mathbf{R}^{n+1} \rightarrow \mathbf{R}^{n+1}$, $ \bm{p}(a, \eta) = a $ for $(a, \eta) \in \mathbf{R}^{n+1} \times \mathbf{R}^{n+1}$, is the projection onto the first coordinate, it is known that for an arbitrary convex body $ C\subset\R^{n+1} $  it holds  that
\begin{equation*}
    \mathcal{A}(C) = \bm{p}(\widetilde{N}^\phi_n(C)).
\end{equation*}
Moreover, the pointwise defined mean curvature functions of $ C $, associated with the pointwise second fundamental form, coincide with the mean curvature functions defined on $ \widetilde{N}^\phi_n(C)$; see \cite[Lemma 3.1]{MR1654685} and \cite{Hug99}.
We extend these results to arbitrary sets of positive reach in an arbitrary uniformly convex normed space. A key difference to the case of convex bodies is that the boundary of a set $C$ of positive reach is in general not graphical around each of its points. Therefore we define $ \partial^v C$ as the set of points $ a \in \partial C $ where the fibre $ N(C,a) $ of the Euclidean unit normal bundle of $ C $ contains only one vector (the same set  is obtained if $ \partial^v C$ is defined with respect to the fibres of $N^\phi(C)$ and a general norm $\phi$). In Theorem \ref{prop:p1} we show that a set of positive reach $ C $ is locally the epigraph of a semiconvex function around each point  $ a \in \partial^v C$ and we prove that this function is twice differentiable at $ a $ if and only if $ \kappa^\phi_{C,i}(a, \eta) < \infty $ for  $ i = 1, \ldots , n $, where $ N^\phi(C,a) = \{\eta \}$. This result opens the way to introduce the notion of an Alexandrov point for a set of positive reach: these are all points of $ \partial^v A $ where the semiconvex function locally representing $ A $ is twice differentiable. We denote the set of all Alexandrov points of $ C $ by $ \mathcal{A}(C)$ and, as in the convex case, each Alexandrov point entails pointwise curvature information that we express by the mean curvature functions $ \bm{h}^\phi_{C,k} $, for $ k \in\{0, \ldots , n\}$; see Definition \ref{def: pointwise mean curvature}. We prove that
\begin{equation*}
    \mathcal{A}(C) = \bm{p}(\widetilde{N}^\phi_n(C)) \cap \partial^v C
\end{equation*}
and
\begin{equation*}
    \bm{h}^\phi_{C,k}(a) = \bm{H}^\phi_{C,k}(a, \eta) \quad \textrm{for every $ a \in \mathcal{A}(C) $ and $ N^\phi(C,a) = \{\eta\}$};
\end{equation*}
see Corollary \ref{cor: postive reach and viscosity boundary}. In the remaining part of Section \ref{sec: Alexandrov points}, employing the geometric inequality for closed sets from Theorem \ref{theo: heintze karcher}, we derive  a version of the Heintze--Karcher inequality for sets of positive reach in Theorem \ref{theo: heintze karcher positive reach}. As a consequence, we obtain Corollary \ref{theo: lower bound} which states that the only sets $C$ of positive reach  with finite and positive volume such that
\begin{equation}\label{eq:introlb}
    \bm{h}^\phi_{C,1}(a) \geq\frac{n \mathcal{P}^\phi(C)}{(n+1) \mathcal{L}^{n+1}(C)} \qquad \textrm{for $ \mathcal{H}^n $ a.e.\ $ a \in \mathcal{A}(C) $},
\end{equation}
are finite unions of rescaled and translated Wulff shapes of radius $ \frac{(n+1) \mathcal{L}^{n+1}(C)}{\mathcal{P}^\phi(C)}$. Here $ \mathcal{P}^\phi(C) $ is the $ \phi $-perimeter of $ C $, namely
\begin{equation*}
    \mathcal{P}^\phi(C) = \int_{\partial^\ast C}\phi(\bm{n}(C,a))\, d\mathcal{H}^n (a),
\end{equation*}
where $ \partial^\ast C $ is the reduced boundary of $ C $ and $ \bm{n}(C,\cdot) $ is the measure-theoretic Euclidean unit normal of $ C $ (notice that a set of positive reach has always locally finite perimeter, see Lemma \ref{lem:finite perimeter and positive reach}). One can easily see that the lower bound in  \eqref{eq:introlb} is sharp by considering convex bodies obtained as  unions of congruent spherical caps; see  Remark \ref{spherical caps}.  A key feature of these results is that they provide information on the global geometry of a set $C$ of positive reach  requiring only assumptions on points in $ \partial^v C $. This is quite surprising in view of the fact that there exist sets $ C $  of positive reach with finite volume and non-empty interior such that $ \mathcal{H}^n(\partial C \setminus \partial^v C) > 0 $ (an explicit example is obtained by taking the function $ f $ in the example described at the end of this introduction such that $ \{ f =0\} $ has positive $  \mathcal{L}^1$ measure).  Corollary \ref{theo: lower bound} plays a key role for the soap bubble theorems in Section \ref{Section: positive reach}, but is also of independent interest.

\paragraph{} A special case of our first soap bubble theorem (Theorem \ref{theo: Alexandrov}) can be stated as follows. In view of condition \eqref{intro Al eq 1}, we point out that while the Radon measures $\Theta^\phi_{j}(C, \cdot)$, for $j=0,\ldots,n-1$ and a set $C\subset\R^{n+1}$ of positive reach, are signed in general, the measure  $\Theta_n^\phi(C,\cdot)$ is always  non-negative. The hypothesis in \eqref{intro Al eq 1}, as well as the hypothesis in  \eqref{intro Alb eq} and \eqref{intro Al 2 eq}, respectively, of the subsequent theorems, is the natural generalization of the hypothesis of ``$ k $-convexity" for smooth domains (see \cite{MR3107515} and references therein) to the singular setting of the present paper. If $C$ is a convex body in $\R^{n+1}$, then all generalized curvature measures $\Theta_j^\phi(C,\cdot)$ are non-negative. In fact, a set $C\subset \R^{n+1}$ of positive reach is convex  if and only if  $\Theta_j^\phi(C,\cdot)\ge 0$ for all $j=0,\ldots,n-1$.

\begin{theorem}[\protect{cf.\ Theorem \ref{theo: Alexandrov}}]\label{intro Al}
Let $ k \in \{1, \ldots, n \}$,  and let $ C \subset \mathbf{R}^{n+1}$ be a set of positive reach with positive and finite volume. Assume  that
\begin{equation}\label{intro Al eq 1}
  \Theta^\phi_{n-i}(C, \cdot) \quad \textrm{is a non-negative measure for  $  i=1,\ldots, k-1 $}
\end{equation}
and
\begin{equation*}
    \Theta^\phi_{n-k}(C, \cdot) = \lambda \, \Theta^\phi_n(C, \cdot)\quad \textrm{for some $ \lambda \in \mathbf{R}\setminus \{0\}$}.
\end{equation*}
Then  $ C $ is a finite disjoint union of rescaled and translated Wulff shapes of radius $ \frac{(n+1) \mathcal{L}^{n+1}(C)}{\mathcal{P}^\phi(C)}$ and $\lambda>0$.

If $ k = 1 $, then the conclusion holds for every set of positive reach with finite and positive volume and for every $ \lambda \in \mathbf{R}$.
\end{theorem}

 In Theorem \ref{theo: Alexandrov} we point out how the common radius of the translated and rescaled Wulff shapes can be expressed in terms of $\lambda$, $n$ and $k$. Conversely, whenever $C$ is a finite disjoint union of rescaled and translated Wulff shapes, each of the generalized curvature measures $\Theta^\phi_{n-k}(C, \cdot)$ is proportional to the top order curvature measure $\Theta^\phi_n(C, \cdot)$.  For the proof of Theorem \ref{theo: Alexandrov}, we first establish an anisotropic extension of the Minkowski--Hsiung formulae for arbitrary sets of positive reach, which is  provided in Theorem \ref{theo: minkowski formula} and adds to several previous versions available in the literature (see \cite[Theorem 3.4]{MR1273334}, \cite[Corollary 3.4]{MR1651402},  \cite[Theorem 2.42]{Hug99}). It is interesting to notice that the hypothesis in \eqref{intro Al eq 1} is preserved for limits of sequences of smooth sets satisfying a positive lower bound on the reach; see Lemma \ref{lem:reachvague} and Corollary \ref{cor:reachvague}. Henceforth Theorem \ref{intro Al} might be helpful to study global geometric properties of limits of smooth almost $ k $-th mean convex sets. We remark that if $ C $ is a smooth set, then the hypothesis in \eqref{intro Al eq 1} is redundant, because the existence of an elliptic point in combination with the continuity of the principal curvatures guarantees the non-negativity hypothesis in \eqref{intro Al eq 1}. This is a classical argument outlined in \cite{MR996826}. In the general situation of Theorem \ref{intro Al}, we do not have any continuity for the curvature functions and it is unclear whether   the hypothesis in \eqref{intro Al eq 1} can be relaxed further (beyond the mean convexity assumption of Theorem \ref{theo: Alexandrov}).

 If the assumption \eqref{intro Al eq 1} is strengthened to include also the condition that $\Theta^\phi_{n-k}(C,\cdot)$ is a non-negative measure, then we obtain the following version of Theorem \ref{intro Al}, which is a special case of Theorem \ref{theo: Alexandrovmodifiedb}. In the statement of the result, we use the $\phi$-curvature measures $  \mathcal{C}^\phi_{j}(C, \cdot)$, $j\in\{0,\ldots,n\}$, of $C$, which are the Radon measures on $\R^{n+1}$ that are obtained as the image measures of the generalized curvature measures $  \Theta^\phi_{j}(C, \cdot)$ under the projection map $\bm{p}$, that is, $  \mathcal{C}^\phi_{j}(C, \cdot)= \Theta^\phi_{j}(C, \cdot\times\partial\mathcal{W}^\phi)$.

\begin{theorem}[\protect{cf.\ Theorem \ref{theo: Alexandrovmodifiedb}}]\label{intro Alb}
Let $ k \in \{1, \ldots, n \}$,  and let $ C \subset \mathbf{R}^{n+1}$ be a set of positive reach with positive and finite volume. Assume  that
\begin{equation}\label{intro Alb eq}
  \Theta^\phi_{n-i}(C, \cdot) \quad \textrm{is a non-negative measure for  $  i=1,\ldots, k $}
\end{equation}
and
\begin{equation*}
    \mathcal{C}^\phi_{n-k}(C, \cdot) = \lambda \, \mathcal{C}^\phi_n(C, \cdot)\quad\textrm{for some $ \lambda >0$}.
\end{equation*}
Then $ C $ is a finite union of rescaled and translated Wulff shapes of radius $ \frac{(n+1) \mathcal{L}^{n+1}(C)}{\mathcal{P}^\phi(C)}$.
\end{theorem}

In the particular case where $C$ is a convex body (and condition \eqref{intro Alb eq} is automatically satisfied as pointed out above), Theorem \ref{intro Alb} has already been established in \cite[Theorem 2.43]{Hug99}. In the framework of convex bodies the single measure on the left side of the hypothesis $\mathcal{C}^\phi_{n-k}(C, \cdot) = \lambda \, \mathcal{C}^\phi_n(C, \cdot)$  can even be replaced by a non-negative linear combination of curvature measures $\mathcal{C}^\phi_{n-k}(C, \cdot)$ with $k\in \{1,\ldots,n\}$. Related stability results have been proved in \cite[Theorems 2.47 and 2.48]{Hug99}. In fact, the results in \cite[Section 2.7]{Hug99} completely establish (in generalized form) Conjecture 8.2 stated in \cite{MR4214340}.

Theorems \ref{intro Al} and \ref{intro Alb} can be seen as  measure-theoretic versions of the soap bubble theorem. Employing the notion of pointwise curvature in Alexandrov points, we  obtain the following differential-geometric version.

\begin{theorem}[\protect{cf.\ Corollary  \ref{Cor: Alexandrov}}]\label{intro Al 2}
Suppose that $ k \in \{1, \ldots , n\} $, $ \lambda \in \mathbf{R} \setminus \{0\}$ and $\varnothing\neq C \subset   \mathbf{R}^{n+1}$ is a set of positive reach with positive and finite volume such that $ \mathcal{H}^{n}(\partial C \setminus \partial^v C) =0$. If $ k = 1 $ we allow $ \lambda \in \mathbf{R}$. Assume that
\begin{equation}\label{intro Al 2 eq}
    \bm{h}^\phi_{C,i}(a) \geq 0 \quad \textrm{for $ i = 1, \ldots , k-1 $ and for $ \mathcal{H}^n $ a.e.\ $ a \in \mathcal{A}(C) $ and}
\end{equation}
\begin{equation*}
    \bm{h}^\phi_{C,k}(a) = \lambda \quad \textrm{for $ \mathcal{H}^n $ a.e.\ $ a \in\mathcal{A}(C)$.}
\end{equation*}
Then $ C $ is a finite union of rescaled and translated Wulff shapes of radius $ \frac{(n+1) \mathcal{L}^{n+1}(C)}{\mathcal{P}^\phi(C)}$, provided that $ \mathcal{H}^{n-k}\big[ \bm{p}(\widetilde{N}^\phi(C) \setminus \widetilde{N}^\phi_n(C))\big] =0  $.
\end{theorem}

\noindent The hypothesis $ \mathcal{H}^n(\partial C \setminus \partial^v C) = 0 $ is equivalent to require that $ \mathcal{P}(C) = \mathcal{H}^n(\partial C)$, where $\mathcal{P}(C)=\mathcal{H}^n(\partial^* C) $ is the Euclidean perimeter of $C$; see Corollary \ref{cor: postive reach and viscosity boundary}. Hence the condition $ \mathcal{P}(C) = \mathcal{H}^n(\partial C)$ is an alternative way to say in a geometric-measure theoretic sense that $ \partial C $ encloses $ C $. As mentioned in the first paragraph of this introduction, this is a fundamental prerequisite to obtain soap bubbles in arbitrary dimension and without any topological assumptions, from the hypothesis that one of the mean curvature functions is constant. The hypothesis $ \mathcal{H}^n(\partial C \setminus \partial^v C) =0$ implies that $ \mathcal{H}^n $ a.e.\ $ a \in \bm{p}(\widetilde{N}^\phi_n(C)) $ is an Alexandrov point (see again Corollary  \ref{cor: postive reach and viscosity boundary}). Based on  the disjoint  decomposition
\begin{equation*}
    \bm{p}(\widetilde{N}^\phi(C) ) = \bm{p}(\widetilde{N}^\phi_n(C)) \cup \bm{p}(\widetilde{N}^\phi(C) \setminus \widetilde{N}^\phi_n(C)) \quad \textrm{and} \quad  \bm{p}(\widetilde{N}^\phi_n(C)) \cap \bm{p}(\widetilde{N}^\phi(C) \setminus \widetilde{N}^\phi_n(C)) = \varnothing,
\end{equation*}
provided in Theorem \ref{lem: graphical representability},
the set $  \bm{p}(\widetilde{N}^\phi(C) \setminus \widetilde{N}^\phi_n(C))  $ can be seen as the set of singular points of $ \partial C $. Consequently, the hypothesis  $ \mathcal{H}^{n-k}\big[ \bm{p}(\widetilde{N}^\phi(C) \setminus \widetilde{N}^\phi_n(C))\big] =0  $ is (in general) a sharp assumption on the smallness of the singular set. For $k=1$ the example of the union of two congruent spherical caps shows that the condition on the Hausdorff measure cannot be relaxed.  On the other hand, for $k=n$ the examples of peaked spheres (see \cite{MR3053706} and \cite[Section 2 and Theorem 15]{MR3951441}), including as a  special case a set resembling an American football \cite[Section 4]{MR3951441}, demonstrate   that the condition on the Hausdorff measure is sharp. (The construction of corresponding examples for $1<k<n$ seems to be an interesting problem.)   In this sense, Theorem \ref{intro Al 2} can also be seen as a result on removable singularities.

 We remark that the class of sets of positive reach (with finite and positive volume) satisfying the assumption $ \mathcal{H}^n(\partial C \setminus \partial^v C) =0$ not only includes all sets that can be locally represented as epigraphs of semiconvex functions (in particular all convex bodies), see Lemma \ref{lem: semiconvex sets}, but it includes many other examples of sets whose boundaries are not  topological manifolds. One such example can be constructed by  reflecting around the $ x $-axis in $ \mathbf{R}^2 $ the set $\{ (x, f(x)) : x \in \mathbf{R} \} $, where $ f : [a, b] \rightarrow \mathbf{R}$ is a non-negative smooth function such that $ f(a) =f(b) =0 $,  $ \lim_{x \to a+}f'(x) = +\infty $, $ \lim_{x \to b-}f'(x) = - \infty $ and $\{f =0\}  $ is a Cantor set of $ \mathcal{L}^1 $-measure zero (in this case $ \partial C \setminus \partial^v C = \{f=0\} $).

\paragraph{} We conclude this introduction with a few comments on some lines of research, which are naturally related with the results of the present paper.

It is well known that the soap bubble theorem also holds for smooth  hypersurfaces in Riemannian manifolds of constant sectional curvature, where the bubbles are realized by geodesic spheres. This was  proved long ago by Alexandrov \cite{MR0143162} for the mean curvature case by means of his celebrated method of moving planes, and extended by Montiel--Ros in \cite{MR1173047} to the case of higher-order mean curvature functions employing the integral-geometric approach. Looking at more general Riemannian spaces, soap bubble theorems in warped product spaces have been the  subject of intensive research in the last two decades. In this direction a fundamental contribution is made by the work of Brendle in \cite{MR3090261}, where it is proved that a compact and embedded hypersurface with constant mean curvature in a suitable class of warped product spaces is a slice;  in a subsequent work  \cite{MR3080487} Brendle and Eichmair treat the case of constant higher order mean curvature hypersurfaces under special convexity hypothesis. In the smooth setting stability results for the aforementioned soap bubble theorems are currently subject of intensive research; see \cite{MR4285109} and \cite{ScheuerXia22} and the references therein.

While in these investigations the smoothness assumption is crucial, it is natural to aim at extensions of the results of the present paper to non-Euclidean ambient spaces in a non-smooth framework as well. In this respect, several important foundational investigations should be mentioned. Walter (see his survey of related work up to 1981 in \cite{MR0600626}), Kleinjohann \cite{MR0604257,MR0610214} and Bangert  \cite{MR0509593,MR0508842,MR0534228,MR0646321}  studied intensively notions of convexity, sets with the unique footpoint property and regularity properties of the associated normal bundles in general Riemannian spaces. In his PhD-thesis (1988)
Kohlmann (see also \cite{MR1621971}) considered Alexandrov's soap bubble problem for general convex sets in constant curvature spaces via curvature measures, but in the non-Euclidean case his methods did not allow to resolve the
important mean curvature case. Sets with positive reach and curvature measures have been intensively studied in Euclidean space (see \cite{MR3932153} and the works cited there).  In the Riemannian setting, the theory of curvature measures and its connection to valuation theory is currently developed; see, e.g., the recent contribution by Fu and Wannerer \cite{MR3945834}. It remains to be explored whether some of these curvature measures can be used for obtaining uniqueness results as considered in the present work. Important structural information about sets with positive reach in Riemannian spaces, such as upper curvature bounds and characterization results, has been derived by Lytchak \cite{MR2098470,MR2180048}, but a complete structural description of general sets with positive reach is not even available in Euclidean spaces so far. However, useful foundational results on distance functions, cut sets and curvatures in  Riemannian spaces (or, more specifically, in Cartan--Hadamard manifolds) can be found in the recent works by Kapovitch and Lytchak \cite{MR4277411} and by Ghomi and Spruck  \cite{MR435870}.

 The present work is mainly motivated by classical problems in differential geometry and the calculus of variations. However,  local Steiner formulas and differentials of basic geometric functionals of convex bodies, as considered here for more general classes of sets, also play a crucial role in the Brunn--Minkowski theory \cite{MR3155183} and its applications. These formulas naturally lead to the curvature measures, which are a major topic of the current investigation, but also to surface area measures, quermassintegrals, and to $L_p$, Orlicz and dual versions of these fundamental functionals and measures. We refer to the seminal work by Huang, Lutwak, Yang and Zhang \cite{MR3573332}, where several new measures are introduced and connections to classical geometric measures are explored. In \cite{MR3573332} then  all these measures are combined in the investigation of associated Minkowski problems, which have received much attention in recent years (see, e.g.,  \cite{MR3037788,MR3037788,MR3825606,MR3851743,MR3783409,MR4008522,MR3882970,MR4040624}).

\section{Preliminaries}\label{Section: preliminaries}

\subsection{Notation and basic facts}
In general, but with few exceptions explained below, we follow the notation and terminology of \cite{MR0257325} (see \cite[pp. 669-676]{MR0257325}). In particular we adopt the terminology from \cite[3.2.14]{MR0257325} when dealing with rectifiable sets.

If $ X $ is a topological space and $ S \subseteq X $, then we denote by $ \Int (S)$ the interior part of $ S $, by $ \partial S $ the topological boundary of $ S $  and by $ \Clos(S)$ the closure of $S$; moreover, the characteristic function of $ S $ is $ \bm{1}_S $.
 If $ Q \subseteq X \times Y $ and $ S \subseteq X $, we set $Q|S = \{ (x, y) \in Q: x \in S  \} $.
We denote by $ \bullet $ \index{-1@$\bullet$} a fixed scalar product on $ \R^{n+1} $ and by $ | \cdot | $ \index{-2@$\mid\cdot\mid$} its associated norm. Hence $\mathbf{S}^n = \{x \in \mathbf{R}^{n+1}: |x|=1\} $ is the Euclidean unit sphere.   The maps $\mathbf{p} : \mathbf{R}^{n+1} \times \mathbf{R}^{n+1} \rightarrow \mathbf{R}^{n+1} $ \index{03@$\mathbf{p}$} and $ \bm{q}: \mathbf{R}^{n+1} \times \mathbf{R}^{n+1} \rightarrow \mathbf{R}^{n+1} $ \index{04@$\mathbf{q}$} are the projection onto the first and the second component respectively, i.e.\ $\mathbf{p}(x, \eta) = x$ and $ \bm{q}(x, \eta) = \eta $.

If $ S \subseteq \mathbf{R}^p $ and $ a \in \mathbf{R}^p $, then we denote by  $ \Tan(S,a) $ and $ \Nor(S,a) $ the \emph{tangent and normal cone} of $ S $ at $ a $ (see \cite[3.1.21]{MR0257325}). Always following \cite{MR0257325} we use the symbol $ \Tan^m(\mathcal{H}^m\restrict S, a) $ for the cone of all $(\mathcal{H}^m\restrict S, m) $ \emph{approximate tangent vectors} at $ a $ (see \cite[3.2.16]{MR0257325}). For an $ (\mathcal{H}^m, m) $ rectifiable and $ \mathcal{H}^m $ measurable set $ S \subseteq \mathbf{R}^p $, the cone $ \Tan^m(\mathcal{H}^m\restrict S, a) $ is an $ m $-dimensional linear subspace for $ \mathcal{H}^m $ a.e.\ $a \in S $. Each Lipschitz function $ f : \mathbf{R}^p \rightarrow \mathbf{R}^q$ has at $ \mathcal{H}^m $ almost all points of $ a \in S $ an \emph{$ (\mathcal{H}^m\restrict S, m)$ approximate differential} $ \ap \Der f(a) : \Tan^m(\mathcal{H}^m\restrict S,a) \rightarrow \mathbf{R}^q $ (see \cite[3.2.16, 3.2.19]{MR0257325}).  If this approximate differential exists, for   $ k \in \{1, \ldots, m\}$ we define the \emph{$ (\mathcal{H}^m\restrict S, m)  $ approximate $ k $-th Jacobian} of $ f $ aat $a$ as
\begin{equation}\label{eq: approx jacobian}
     \ap J^S_k f(a) = \big\| \textstyle \bigwedge_{k} \ap \Der f(a) \big\| = \sup\{ | [\textstyle \bigwedge_{k} \ap \Der f(a)](\xi) | : \xi \in \bigwedge_k \Tan^m(\mathcal{H}^m\restrict S,a), \; |\xi|=1\},
\end{equation}
where $ \bigwedge_{k} \ap \Der f(a): \bigwedge_k \Tan^m(\mathcal{H}^m\restrict S,a) \rightarrow \bigwedge_k \mathbf{R}^q $ is the linear map induced by $ \ap \Der f(a)$ (see \cite[1.3.1]{MR0257325}). The norms $ |\cdot| $ on the right-side of \eqref{eq: approx jacobian}  denote the norm induced on $ \bigwedge_k \Tan^m(\mathcal{H}^m\restrict S,a) $ and $ \bigwedge_k \mathbf{R}^q $
by the inner products of $\Tan^m(\mathcal{H}^m\restrict S,a)$ and $ \mathbf{R}^q$, respectively (see \cite[1.7.5, 1.7.6]{MR0257325}). The approximate Jacobian of a Lipschitz map will be repeatedly used in this paper in applying the following version of  Federer's coarea formula for Lipschitz maps on rectifiable sets, for which we refer to \cite[pp.\ 300--301]{MR0467473}.

\begin{Lemma}[Federer]\label{coarea}
If $ W$ is an $ (\mathcal{H}^m, m)$ rectifiable $ \mathcal{H}^m$ measurable subset of $ \mathbf{R}^p$, $ f : W \rightarrow \mathbf{R^q}$ is a Lipschitzian map, $ k \in \{0,\ldots,m\}$ and $ S \subseteq \mathbf{R}^q$ is a countable union of Borel subsets of $\mathbf{R}^q$ with finite $ \mathcal{H}^k$ measure, then
\begin{equation*}
    \int_{W \cap f^{-1}(S)}\phi(x) \, \ap J_k^Wf(x)\, d\mathcal{H}^m(x) = \int_S \int_{W \cap f^{-1}(\{y\})}\phi(x)\,d\mathcal{H}^{m-k}(x)\, d\mathcal{H}^k(y)
\end{equation*}
for every $\mathcal{H}^m$ measurable function $\phi:W\to [0,\infty]$.
\end{Lemma}

Following Federer \cite[page 15]{MR0257325} we denote by $ \Lambda(n,m)$ the set of all increasing maps from $ \{1, \ldots , m \}$ into $ \{1, \ldots , n \}$. We now introduce the $ k $-th elementary symmetric functions. For $x=(x_1,\ldots,x_n)\in\R^n$ and $k\in\{1,\ldots,n\}$, we define \index{-4@$S_k(x)$}
\begin{equation}\label{eq: symmtric functions}
S_k(x):=\sum_{\lambda \in \Lambda(n,k)}x_{\lambda(1)}\cdots x_{\lambda(k)}.
\end{equation}
Then
$$\Gamma_k^\circ:=\{x\in\R^n:S_1(x)>0,\ldots,S_k(x)>0\}$$
is an open convex cone whose closure is the (pointed) closed convex cone \index{-5@$\Gamma_k$}
$$\Gamma_k:=\{x\in\R^n:S_1(x)\ge 0,\ldots,S_k(x)\ge 0\}$$
with apex $0$ (see \cite[Section 2, page 582]{MR1726702} or \cite[Proposition 1.3.2]{Salani}).

\medskip

For a vector $x$ with positive components the next lemma is well known. In the present more general form, it is in fact harder to find an explicit reference.

\begin{Lemma}\label{lem:Maclaurin}
Let $k\in\{1,\ldots,n\}$. If $x\in\Gamma_k$, then
$$
\left(\frac{S_i(x)}{\binom{n}{i}}\right)^{\frac{1}{i}}\ge
 \left(\frac{S_j(x)}{\binom{n}{j}}\right)^{\frac{1}{j}}\quad
 \text{for $1\le i\le j\le k$}.
$$
\end{Lemma}
\begin{proof} See \cite[Proposition 1.3.3 (4)]{Salani}. We indicate an alternative argument here. First, Newton's inequality holds for any $x\in \R^n$, as shown in \cite{MR1033349}. Let $S_0(x):=1$ and
$E_r(x):=S_r(x)/\binom{n}{r}$ for $r=0,\ldots,k$.
If $x\in\Gamma_k^\circ$, then the asserted inequalities can be obtained (as usually) by repeated application of Newton's inequality $ E_\ell(x)E_{\ell+2}(x)\le E_\ell(x)^2 $ via
$$
\prod_{\ell=0}^{r-1}\left(E_\ell(x)E_{\ell+2}(x)\right)^{\ell+1}\le \prod_{\ell=1}^{r}E_\ell(x)^{2\ell}\quad \text{for } r\in\{1,\ldots,k-1\}.
$$
Since $\Gamma_k$ is the closure of the open convex cone $\Gamma_k^\circ$, the assertion for $x\in\Gamma_k$ follows by an obvious approximation argument.
\end{proof}

\subsection{Multivalued maps}\label{subsec: multivalued}

A map $ T $ defined on a set $ X $ is called \emph{$Y$-multivalued}, if $ T(x) $ is a subset of $ Y $ for every $ x \in X $. If $ T(x) $ is a singleton, with a little abuse of notation we denote by $ T(x) $ the unique element of the set $ T(x) \subseteq Y $. Suppose that $(X, \|\cdot \|)$ and $(Y, \|\cdot \|)$ are finite-dimensional normed
vectorspaces and $T$ is a $Y$-multivalued map such that $ T(x) \neq \varnothing $ for every $ x \in X $.
\begin{enumerate}
	\item  We say that $ T $ is \emph{weakly continuous at $ x\in X $} if and only if for every $ \epsilon > 0 $ there exists $ \delta > 0 $ such that
if $y\in X$ and $ \|y-x\| < \delta $, then
	\begin{equation*}
		T(y) \subseteq T(x) + \{v \in Y : \|v\| < \epsilon\} ;
	\end{equation*}
	if, additionally,  $ T(x)$ is a singleton, then we say that $ T $ is continuous at $ x $.
	\item  We say that \emph{$T$ is strongly
		differentiable at $x \in X$} if and only if $T(x)$ is a singleton and there exists a linear map $L : X \to Y$ such that for every $\varepsilon > 0$ 	there exists $\delta > 0$ such that if $y\in X$, $ \|y-x\| < \delta $ and $w \in T(y)$, then
	\begin{displaymath}
		\| w - T(x) - L(y-x)\| \leq \varepsilon \|y-x\|;
	\end{displaymath}
	cf.\ \cite[Definition 2.28]{kolasinski2021regularity}. The linear map $ L $ is unique (cf.\ \cite[Remark 2.29]{kolasinski2021regularity}) and we denote it by $ \Der T(x) $. Moreover we denote by $ \dmn \Der T $ the set of points $ x \in X $ at which $ T $ is strongly differentiable. In the following, we simply write ``differentiable'' when we actually mean ``strongly differentiable''.
\end{enumerate}

\noindent The following general fact on the Borel measurability of the differential of a multivalued map will be useful.

\begin{Lemma}\label{lem : Borel measurability differential}
Let $(X, \|\cdot \|)$ and $(Y, \|\cdot \|)$ be finite-dimensional normed
vectorspaces, and let $T$ be a $Y$-multivalued weakly continuous map such that $ T(x) \neq \varnothing $ for  $ x \in X $.

Then $\{x \in X : \textrm{$T(x)$ is a singleton}   \} $ and $ \dmn \Der T $ are Borel subsets of $ X $ and $ \Der T : \dmn \Der T \rightarrow \Hom(X,Y) $ is Borel measurable.
\end{Lemma}

\begin{proof}
We define $ U = \{x \in X: \textrm{$T(x)$ is a singleton}\} $ and the function $ \diam : \bm{2}^Y \setminus \{\varnothing\} \rightarrow [0, \infty] $ by $ \diam S = \sup \{\|y_1 - y_2 \| : y_1, y_2 \in S\} $ for every $ S \in \bm{2}^Y \setminus \{\varnothing\} $. Noting that $ \diam \circ T : X \rightarrow [0, +\infty] $ is upper semicontinuous, we conclude that $ U = \{x \in X : \diam(T(x))=0    \} $ is a Borel subset of $ X $.

For  positive integers $ i, j\in\N $ we define
\begin{equation*}
	C_{ij} = \bigg\{ (x, L) \in U \times \Hom(X,Y) : \|w - T(x) - L(h)\| \leq \frac{1}{i}\| h \| \quad \textrm{for $\|h\| < \frac{1}{j} $ and $ w \in T(x+h) $} \bigg\}.
\end{equation*}
We prove that $ C_{ij} $ is relatively closed in $ U \times \Hom(X,Y) $. By contradiction assume that $C_{ij} $ is not closed. Then there exists $ (x_0, L_0)\in (U \times \Hom(X,Y)) \setminus C_{ij} $ and  a sequence $(x_k, L_k) \in C_{ij} $ converging to $(x_0, L_0) $. Hence that there exist $ h_0 \in X $ with $ \|h_0 \| < \frac{1}{j} $ and $ w_0 \in T(x_0 + h_0) $ such that
\begin{equation*}
	\| w_0- T(x_0) - L_0(h_0)\| > \frac{1}{i}\|h_0 \|.
\end{equation*}
We define $ h_k = x_0 + h_0 - x_k $ for $k\geq1 $ and  select $ k_0\geq 1 $ so that $ \|h_k \| < \frac{1}{j} $ for  $ k \geq k_0 $. Since $(x_k,L_k)\in C_{ij}$, $ x_0 + h_0 = x_k + h_k $ and $ w_0 \in T(x_k+ h_k) $ for  $ k \geq 1$, we infer that
 \begin{equation*}
 	\|w_0 - T(x_k)-L_k(h_k)\| \leq \frac{1}{i}\|h_k \| \qquad \textrm{for  $ k \geq k_0 $}.
 \end{equation*}
Noting that $ T(x_k) \to T(x_0)$ and $ h_k \to h_0 $ as $ k \to \infty $, we deduce that
\begin{equation*}
\|w_0 - T(x_0)-L_0(h_0)\| \leq \frac{1}{i}\|h_0 \|,
\end{equation*}
and we obtain a contradiction.

Let $G : =\{ (x, \Der T(x)) : x \in \dmn \Der T  \}$ and $ \pi_X : X \times \Hom(X,Y) \rightarrow X $, $ \pi_X(x, T) = x $ for every $(x,T)\in X \times \Hom(X,Y) $. Noting that
\begin{equation*}
	G = \bigcap_{i=1}^\infty \bigcup_{j=1}^\infty C_{ij},
\end{equation*}
we infer that $ G $ is a Borel subset of $ U \times \Hom(X,Y) $. Since $ \pi_X|G $ is injective, we obtain
from \cite[2.2.10, bottom of page 67]{MR0257325} that $\{x\in \dmn \Der T: \Der T(x)\in B\}=\pi_X(G\cap (X\times B))$
is a Borel set in $X$ if  $B\subseteq \Hom(X,Y)$ is a Borel set, which implies the remaining assertions.
\end{proof}

\begin{Remark}
	The case of single-valued continuous functions is treated in \cite[page 211]{MR0257325} with a similar proof.
\end{Remark}

The following elementary lemma will be useful in Section \ref{Section: positive reach}.

\begin{Lemma}\label{lem: elementary diff}
Let $U\subseteq \R^k$ be open, let $F:U\to\R^k$ be differentiable at $a\in U$, and assume that $\Der F(a):\R^k\to\R^k$ is invertible. Let $V\subseteq \R^k$ be open with $x=F(a)\in V$ and assume that $G:V\to\R^k$ is Lipschitz. Further, assume that $F(U)\subseteq V$ and $G\circ F=\Id_U$.  Then $G$ is differentiable at $x$ and $\Der G(x)=\Der F(a)^{-1}$.
\end{Lemma}

\begin{proof}
Let $v\in\R^k$. Let $\varepsilon>0$. Then there is some $\delta>0$ such that if $t\in\R$ with $0<|t|<\delta$, then $a+tv\in U$ and
$$
\left|\frac{F(a+tv)-F(a)}{t}-\Der F(a)(v) \right|\le \frac{\varepsilon}{\text{Lip(G)}}.
$$
Then we obtain
\begin{align*}
&\left|\frac{G(x+t\Der F(a)(v))-G(x)}{t}-v \right|\\
&\qquad =\left|\frac{G(F(a)+t\Der F(a)(v))-G(F(a))}{t}-\frac{G(F(a+tv))-G(F(a))}{t} \right|\\
&\qquad =\left|\frac{G(F(a+tv))-G(F(a)+t\Der F(a)(v))}{t}\right|\\
&\qquad \le \text{Lip}(G)\cdot \left|\frac{F(a+tv)-F(a)}{t}-\Der F(a)(v) \right|\le \varepsilon.
\end{align*}
Since $\Der F(a):\R^k\to\R^k$ is invertible, this shows that if $w\in\R^k$, then
$$
\lim_{t\to 0} \left|\frac{G(x+tw)-G(x)}{t}-\Der F(a)^{-1}(w) \right|=0,\\
$$
and hence $G$ is differentiable at $x$ with $\Der G(x)=\Der F(a)^{-1}$.
\end{proof}

\subsection{Norms and Wulff shapes.}

Let $ \phi $ \index{07@$\phi$} be a norm on $ \mathbf{R}^{n+1} $. We say that $ \phi $ is a \emph{$ \mathcal{C}^k $-norm} if and only if $ \phi \in \mathcal{C}^k(\mathbf{R}^{n+1} \setminus \{0\}) $. We say that $ \phi $ is \emph{uniformly convex} if and only if there exists a constant $ \gamma > 0 $ (ellipticity constant) such that the function $ \mathbf{R}^{n+1} \ni u \mapsto \phi(u) - \gamma |u| $ is convex. If $ \phi $ is a uniformly convex $ \mathcal{C}^2 $-norm then
\begin{equation*}
\Der^2\phi(u)(v,v) \geq \gamma |v|^2
\end{equation*}
for all $ u \in \mathbf{R}^{n+1}$ with $|u|=1 $ and for all $ v \in \mathbf{R}^{n+1} $ perpendicular to $ u $. In the following, a compact convex set with non-empty interior will be called a \emph{convex body}. The symmetric (with respect to the origin $o$) convex body $B=\{x\in \R^{n+1}:\phi(x)\le 1\}$ is the unit ball or \emph{gauge body} associated with $\phi$. Conversely, the \emph{gauge function} (norm) of $B$ \index{05@$g_B$} is just $\phi$, i.e.~
$$
g(B,x):=g_B(x):=\|x\|_B:=\min\{\lambda\ge 0: x\in\lambda B\}=\phi(x).
$$
For a compact convex set $\varnothing\neq K\subset \mathbf{R}^{n+1}$, we write $h_K=h(K,\cdot)$ for its \emph{support function}, which is defined by \index{06@$h_K$} $h(K,x):=h_K(x):=\max\{x\bullet z:z\in K\}$, and we denote by \index{-3@$K^\circ$} $K^\circ:=\{x\in\R^{n+1}: x\bullet y\le 1 \text{ for }y\in K\}$ the \emph{polar body} of $K$.
Note that $K^\circ$ is again a convex body if the origin $o$ is an interior point of $K$. Hence we have $$\phi=g_B=h_{B^\circ}$$ (see \cite[Section 1.7.2, page 53]{MR3155183}).

For any norm $ \phi $ we denote by $ \phi^\ast $ \index{08@$\phi^\ast$} \emph{the conjugate norm} of $ \phi $ which is defined by  $ \phi^\ast(u) = \sup \{ v \bullet u : \phi(v) = 1    \} $ for $u\in \mathbf{R}^{n+1}$. Then we also have $$\phi^\ast=h_B=g_{B^\circ}.$$
 It is well known that if $ \phi $ is a uniformly convex $ \mathcal{C}^2 $-norm then $ \phi^\ast $ is a uniformly convex $ \mathcal{C}^2 $-norm. In geometric terms, this is equivalent to the property that $B$ and $B^\circ$ both have a boundary of class $C^2$ and positive Gauss curvature everywhere (we then say that these bodies are of class $C^2_+$). In particular, the \emph{spherical image map (Gauss map)} $\bm{u}_B:\partial B\to \mathbf{S}^{n}$ is a diffeomorphism of class $C^1$ whose inverse is given by the restriction of $\nabla h_B$ to $\mathbf{S}^{n}$. We refer to  \cite[Lemma 2.32]{MR4160798} for this and other basic facts on $ \phi $ and $ \phi^\ast $ and to \cite[Section 2.5]{MR3155183} for the relations between smoothness properties of $B$ and $B^\circ$. These facts will be tacitly used throughout the paper.

We define the \emph{Wulff shape} (or \emph{Wulff crystal}) of $ \phi $ as \index{09@$\mathcal{W}^\phi$}
\begin{equation*}
\mathcal{W}^\phi = \{x \in \mathbf{R}^{n+1} : \phi^\ast(x)\leq 1\}=B^\circ.
\end{equation*}
Hence, if $ \phi $ is a uniformly convex $ \mathcal{C}^2 $-norm, then the Wulff shape of $ \phi $ is a uniformly convex set with $ \mathcal{C}^2_+ $ boundary. In this case the exterior unit normal (spherical image) map of $ \mathcal{W}^\phi $ is the map \index{15@$\bm{n}^\phi $}
\begin{equation}\label{eq: normal of Wulff shape}
\bm{n}^\phi : \partial \mathcal{W}^\phi \rightarrow \mathbf{S}^n;
\end{equation}
we remark (see \cite[2.32]{MR4160798} or \cite[Section 2.5]{MR3155183}) that $ \bm{n}^\phi=\bm{u}_{B^\circ} $ is a $ \mathcal{C}^1 $-diffeomorphism onto $ \mathbf{S}^n $ and
\begin{equation}\label{eq: normal wulff shape and gradient phi}
\nabla \phi(\bm{n}^\phi(x)) = x \quad \textrm{for $ x \in \partial \mathcal{W}^\phi $,} \qquad \bm{n}^\phi(\nabla \phi(u)) = u \quad \textrm{for $u \in \mathbf{S}^n $.}
\end{equation}
Since $ \Der (\nabla \phi)(u) \bullet u = 0 $ for $ u \in \mathbf{S}^n $, we notice that
\begin{equation*}
    \Tan(\mathbf{S}^n, u) = \Der (\nabla \phi)(u)[\Tan(\mathbf{S}^n, u)] = \Tan(\partial \mathcal{W}^\phi, \nabla \phi(u)).
\end{equation*}

Moreover, we have $\phi(\bm{n}^\phi(x))=\bm{n}^\phi(x)\bullet x$ for $x\in \partial \mathcal{W}^\phi $.  We also point out (cf.\ \cite[Lemma 2.32(f)]{MR4160798}) that the compositions of the gradient maps $\nabla \phi:\R^{n+1}\setminus\{0\}\to \partial B^\circ$ and $\nabla \phi^*:\R^{n+1}\setminus\{0\}\to \partial B$ satisfy the relations
\begin{equation}\label{eq:grad}
\nabla\phi^*\circ\nabla\phi|_{\partial B}=\Id_{\partial B}\qquad\text{and}\qquad \nabla\phi\circ\nabla\phi^*|_{\partial B^\circ}=\Id_{\partial B^\circ}.
\end{equation}

\subsection{Distance function and normal bundle}\label{sec:distnbperi}

\paragraph{Warning.} In this paper we occasionally refer to \cite{kolasinski2021regularity}. However, notice that in this paper we use the same symbols with a different meaning (the roles of $\phi$ and $\phi^*$ are changed). Hence the definitions below have to be compared carefully with those given in \cite[Sections 1 and 2]{kolasinski2021regularity}.

\paragraph{Convention.}
If $ \phi $ is the Euclidean norm, then the dependence on $ \phi $ is omitted in all the symbols introduced below.

\bigskip

Let $ \varnothing\neq A\subseteq \mathbf{R}^{n+1} $ be a closed set, and let $ \phi $ be a uniformly convex $ \mathcal{C}^2 $-norm on $ \mathbf{R}^{n+1} $.
The \emph{$\phi $-distance function}  $ \bm{\bm{\delta}}^\phi_A : \mathbf{R}^{n+1} \rightarrow \mathbf{R} $ \index{D1@$\bm{\bm{\delta}}^\phi_A$} is defined by
\begin{equation}\label{eq: distance function}
	\bm{\bm{\delta}}^\phi_A(x) = \min\{ \phi^\ast(x-c) : c \in A  \}=\min\{\lambda\ge 0:x\in A+\lambda B^\circ\} \qquad \textrm{for $ x \in \mathbf{R}^{n+1} $}.
\end{equation}
Next we define the level and sublevel sets at distance $r>0$ with respect to $\bm{\delta}^\phi_A$ by \index{11@$S^\phi(A,r)$} \index{10@$B^\phi(A,r)$}
\begin{equation*}
	S^\phi(A,r) = \{x \in  \mathbf{R}^{n+1}: \bm{\delta}^\phi_A(x) = r   \} \qquad \textrm{and} \qquad 	B^\phi(A,r) = \{x \in  \mathbf{R}^{n+1}: \bm{\delta}^\phi_A(x) \leq r   \}.
\end{equation*}
For $a\in\R^{n+1}$ we set $B^\phi(a,r)=B^\phi(\{a\},r)=a+rB^\circ$. Moreover,  an open neighborhood of $A$ is defined by    \index{12@$U^\phi(A,r)$} $U^\phi(A,r)=\{x\in\R^{n+1}:\delta_A^\phi(x)<r\}=\Int(B^\phi(A,r))$, and again we set $U^\phi(a,r)=U(\{a\},r)$. Clearly, $\bm{\delta}^\phi_A$ is a Lipschitz map; moreover it is a classical fact that $ -\bm{\delta}^\phi_A $ is semiconvex on $ \mathbf{R}^{n+1} \setminus B^\phi(A, r)$ for $ r > 0 $; (cf.\ \cite[Lemma 2.41(b)]{kolasinski2021regularity} and the references therein).

The \emph{nearest $ \phi $-projection} $ \bm{\xi}^\phi_A : \mathbf{R}^{n+1}  \rightarrow \bm{2}^{A} $  \index{D2@$\bm{\xi}^\phi_A$} is the $A$-multivalued map defined by
\begin{displaymath}
\bm{\xi}^\phi_A(x) = \{c \in A: \bm{\delta}^\phi_A(x) = \phi^\ast(x-c)\}  \qquad \textrm{for  $ x\in \mathbf{R}^{n+1} $.}
\end{displaymath}
This is a weakly continuous map by \cite[Lemma 2.41(f)]{kolasinski2021regularity}. By $\Unp^\phi(A) $ \index{D5@$\Unp^\phi(A)$} we denote the set of all $ x \in \mathbf{R}^{n+1} \setminus A $ such that there exists a unique point $ c \in A $ with $ \phi^\ast(x-c) = \bm{\delta}^\phi_A(x) $, i.e., $\Unp^\phi(A)=\{x\in \mathbf{R}^{n+1} \setminus A: \mathcal{H}^0(\bm{\xi}^\phi_A(x))=1\} $. For $x\in \Unp^\phi(A)$ we simply write $\xi^\phi_A(x)=c$ if $\xi^\phi_A(x)=\{c\}$.  Notice that $ \Unp^\phi(A) $ is a Borel subset of $ \mathbf{R}^{n+1} $ by Lemma \ref{lem : Borel measurability differential} (see \cite[Lemma 3.12]{MR1782274} for a more general fact). It is well known that $ \mathbf{R}^{n+1} \setminus (A \cup \Unp^\phi(A))$ equals the set of points in $ \mathbf{R}^{n+1} \setminus A $ where $ \bm{\delta}^\phi_A $ is not differentiable (cf.\ \cite[Lemma 2.41(c)]{kolasinski2021regularity}. Moreover, $ \mathbf{R}^{n+1} \setminus (A \cup \Unp^\phi(A))$ can be covered outside a set of $ \mathcal{H}^n$ measure zero by a countable union of $ n $-dimensional graphs of $ \mathcal{C}^2$-functions; for the proof of this result  one can proceed as in the Euclidean case which is treated in \cite{Hajlasz2022}.

The  $ \phi $-Cahn--Hoffman map of $ A $ is the $ \partial \mathcal{W}^\phi $-multivalued function $ \bm{\nu}^\phi_A: \mathbf{R}^{n+1} \setminus A \rightarrow \bm{2}^{\partial \mathcal{W}^\phi} $ defined by \index{D3@$\bm{\nu}^\phi_A$}
\begin{equation}\label{eq: nu function}
\bm{\nu}^\phi_A(x) =\bm{\delta}_A^\phi(x)^{-1} (x - \bm{\xi}_A^\phi(x)) \qquad \textrm{for $ x \in \mathbf{R}^{n+1} \setminus A $.}
\end{equation}
Next we introduce the map $ \bm{\psi}_A^\phi : \mathbf{R}^{n+1} \setminus A \rightarrow \mathbf{2}^A \times \bm{2}^{\partial \mathcal{W}^\phi} $ by \index{D4@$\bm{\psi}_A^\phi$}
\begin{equation*}
\bm{\psi}_A^\phi(x) = (\bm{\xi}_A^\phi(x), \bm{\nu}_A^\phi(x)) \qquad \textrm{for $ x \in \mathbf{R}^{n+1} \setminus A $.}
\end{equation*}
Recall the relations
\begin{equation}\label{eq: gradient and normal}
\nabla \bm{\delta}^\phi_A(x) = \nabla \phi^\ast(x - \bm{\xi}^\phi_A(x)) \in \partial \mathcal{W}^{\phi^\ast}\qquad \textrm{and} \qquad   \nabla \phi(\nabla \bm{\delta}^\phi_A(x)) =  \bm{\nu}^\phi_A(x)\in \partial \mathcal{W}^\phi
\end{equation}
for $ x \in \Unp^\phi(A) $, cf.\ \cite[Lemma 2.41(c)]{kolasinski2021regularity} (but recall that the notation and in particular the roles of $\phi$ and $\phi^\ast$ are changed in comparison with \cite{kolasinski2021regularity}). The equivalence of these two relations can be seen from \eqref{eq:grad}. It follows from \cite[Lemma 2.32]{MR4160798} or from \eqref{eq: gradient and normal} and basic properties of strictly convex bodies that $\bm{\nu}^\phi_A(x) \bullet \nabla \bm{\delta}^\phi_A(x) = \phi(\nabla \bm{\delta}^\phi_A(x)) = 1 $ for $ x \in \Unp^\phi(A) $.

The \emph{$\phi $-unit normal bundle of $ A$} is defined by \index{N1@$N^\phi(A)$}
\begin{equation*}
	N^\phi(A) = \{ (x, \eta )  \in A\times \partial \mathcal{W}^\phi:  \bm{\delta}_A^\phi(x + r\eta) = r\; \textrm{for some $ r > 0 $}  \}=\{\bm{\psi}_A^\phi(z) : z\in \Unp^\phi(A)\},
\end{equation*}
and we set \index{N2@$N^\phi(A,x)$}
$$
N^\phi(A,x) = \{\eta\in \partial \mathcal{W}^\phi : (x,\eta)\in N^\phi(A)\}\qquad\text{ for $x\in A$ } .
$$
We recall (cf.\ \cite[Lemma 5.2]{MR4160798}) that $ N^{\phi}(A) $ is a Borel subset of $ \mathbf{R}^{n+1} \times \mathbf{R}^{n+1} $ and it can be covered up to a set of $ \mathcal{H}^n$ measure zero by a countable union of $ n $-dimensional graphs of $ \mathcal{C}^1 $-functions; moreover we have
\begin{equation}\label{eq: phi normal vs euclidean normal}
N^\phi(A) = \{(a, \nabla \phi(u)) : (a, u)\in N(A)\}.
\end{equation}
The {\em  $\phi$-reach function} of $ A $ is the upper semicontinuous  (see \cite[Lemma 2.35]{kolasinski2021regularity}) function $\bm{r}_{A}^{\phi}: N^{\phi}(A) \rightarrow (0, +
\infty]$ given by \index{G1@$\bm{r}_{A}^{\phi}$}
\begin{equation}\label{eq: reach function}
	\bm{r}_{A}^{\phi}(a,\eta) = \sup \bigl\{ s>0 : \bm{\delta}^\phi_A(a + s\eta) = s \bigr\}
	\qquad \text{for $(a,\eta)\in N^{\phi}(A) $},
\end{equation}
and \emph{the $ \phi $-cut locus of $ A $} is given by \index{D6@$\Cut^{\phi}(A)$}
\begin{equation*}
	\Cut^{\phi}(A) = \bigl\{ a + \bm{r}^\phi_A(a, \eta)\eta : (a, \eta) \in
	N^{\phi}(A) \bigr\}.
\end{equation*}
The strict convexity of $ \phi $ implies that $ \mathbf{R}^{n+1} \setminus (A \cup \Unp^\phi(A)) \subseteq \Cut^\phi(A)$; see Lemma 2.41 (c) and Remark 4.1   in  \cite{kolasinski2021regularity} as well as further references given there. Moreover $ \Cut^\phi(A)$ is always contained in the closure of $ \mathbf{R}^{n+1} \setminus (A \cup \Unp^\phi(A))$; cf.\ \cite[Theorem 3B]{MR1434446}. We recall that (cf.\ \cite[Remark 5.10]{MR4160798})
\begin{equation}\label{eq: cut locus}
    \mathcal{L}^{n+1}(\Cut^\phi(A)) =0.
\end{equation}
On the other hand, the closure of $ \mathbf{R}^{n+1} \setminus (A \cup \Unp^\phi(A))$ might have non-empty interior even if $ A $ is the closure of the complement of a convex body with $ \mathcal{C}^{1,1} $-boundary (see \cite{MR4279967} for this and other critical examples).  If $ A $ is convex, then $ \Cut^\phi(A) = \varnothing $. A related function which will be useful in the sequel is defined by
\index{G2@$\bm{\rho}^\phi_A$}
\begin{equation*}
	\bm{\rho}^\phi_A(x) = \sup\{ s \geq 0 :  \bm{\delta}^\phi_A(a + s (x - a)) = s \bm{\delta}^\phi_A(x)  \} \qquad \textrm{for $ x \in \mathbf{R}^{n+1} \setminus A $ and $ a \in \bm{\xi}^\phi_A(x) $.}
\end{equation*}
This definition does not depend on the choice of $ a  \in \bm{\xi}^\phi_A(x) $ and the function $\bm{\rho}^\phi_A : \mathbf{R}^{n+1} \setminus A \rightarrow [1, +\infty] $ is upper semicontinuous; cf.\ \cite[Lemma  2.33]{kolasinski2021regularity}. Notice that $  \{ x : \bm{\rho}^\phi_A(x)> 1  \} \subseteq \Unp^\phi(A) $  (see Lemma 2.33 and Remark 4.1 in  \cite{kolasinski2021regularity} and
\begin{equation}\label{eq : r and rho}
\bm{r}^\phi_A(a,\eta) = r \bm{\rho}^\phi_A(a+r\eta) \qquad \textrm{for  $ (a,\eta) \in N^\phi(A) $ and $ 0 < r < \bm{r}^\phi_A(a,\eta) $},
\end{equation}
as shown in  \cite[Lemma 2.35]{kolasinski2021regularity}.

The following two results from \cite{kolasinski2021regularity}, which we recall here for the ease of the reader, plays an important role in the next section. The norm $ \phi $ is always  assumed to be uniformly convex and $ \mathcal{C}^2 $ in $ \mathbf{R}^{n+1} \setminus \{0\}$.

\begin{Lemma}[\protect{cf.\ \cite[Corollary 3.10]{kolasinski2021regularity}}]
	\label{theo: projectionLipschitz}
Let $\varnothing\neq A \subseteq \mathbf{R}^{n+1}$ be closed,
	$1 < \lambda < \infty$, $0 < s < t < \infty$, and
	\begin{displaymath}
		A_{\lambda,s,t} = \bigl\{ x \in \mathbf{R}^{n+1} \setminus A: \bm{\rho}^\phi_A(x)\geq \lambda, \; s \leq \bm{\delta}^\phi_A(x) \leq t \bigr\} \,.
	\end{displaymath}
Then $ \bm{\xi}_{A}^{\phi} | A_{\lambda,s,t} $ is Lipschitz continuous.
\end{Lemma}

Before we can state the next result we need to recall from \cite{kolasinski2021regularity} the reach-type function  $ \arl{A}{\phi} : N^\phi(A) \rightarrow [0, +\infty]  $ defined by \index{G3@$\arl{A}{\phi}$}
\begin{equation*}
   \arl{A}{\phi}(a, \eta) = \sup\Bigg\{  \sigma r: \sigma > 1,\, 0 < r < \ar{A}{\phi}(a, \eta), \, \lim_{\rho \to 0+}\frac{\mathcal{L}^{n+1}(A_\sigma \cap B(a+r\eta, \rho))}{\mathcal{L}^{n+1}(B(a+r\eta, \rho))} = 1   \Bigg\} \cup \{0\},
\end{equation*}
where $ A_\sigma = \{\bm{\rho}^\phi_A \geq \sigma \}$. Notice that $ \arl{A}{\phi}(a,\eta) \leq \ar{A}{\phi}(a, \eta) $ for every $(a, \eta) \in N^\phi(A)$; see \cite[Remark 4.10]{kolasinski2021regularity}. This function plays a central role in \cite{kolasinski2021regularity} in the study of the structure of the set $ \dmn \Der \bm{\nu}^\phi_A $ which coincides with the set of twice differentiability points of $ \bm{\delta}^\phi_A $ by \cite[Lemma 2.41(e)]{kolasinski2021regularity}).

\begin{Lemma}[\protect{cf.\ \cite[Theorem  1.5]{kolasinski2021regularity}}]\label{theo: distance twice diff}
If $ \varnothing\neq A \subseteq \mathbf{R}^{n+1}$ is a closed set, then
\begin{equation*}
    \mathcal{L}^{n+1}\big( \mathbf{R}^{n+1} \setminus (A \cup \dmn \Der \bm{\nu}^\phi_A)\big) =0
\end{equation*}
and
       \begin{equation*}
\{ a + r \eta: 0 < r < \ar{A}{\phi}(a, \eta) \} \subseteq \dmn(\Der \bm{\nu}^\phi_A) \quad  \quad \textrm{for $\mathcal{H}^{n}$ almost all $(a, \eta) \in N^\phi(A) $}.
\end{equation*}
The function $ \arl{A}{\phi}$ is Borel measurable. Moreover, if $ a + s \eta \in \dmn \Der \bm{\nu}^\phi_A $  for some $ s\in (0, \arl{A}{\phi}(a, \eta))$, then  $ a + r \eta \in \dmn \Der \bm{\nu}^\phi_A $ for all $r\in ( 0 , \arl{A}{\phi}(a, \eta)) $. Finally,
\begin{equation*}
    \arl{A}{\phi}(a, \eta) = \ar{A}{\phi}(a, \eta) \qquad \textrm{for $ \mathcal{H}^n $ a.e.\ $(a, \eta) \in N^\phi(A) $.}
\end{equation*}
\end{Lemma}

\begin{Lemma}\label{lem: distance twice diff}
Suppose $ \varnothing\neq A \subseteq \mathbf{R}^{n+1}$ is a closed set,  $ \xi : \mathbf{R}^{n+1} \rightarrow \mathbf{R}^{n+1} $ is an arbitrary function such that $ \xi(y)\in \bm{\xi}^\phi_A(y) $ for $ y \in \mathbf{R}^{n+1} $ and
\begin{equation*}
    u(y) = \frac{\nabla \phi^\ast(y-\xi(y))}{|\nabla \phi^\ast(y-\xi(y))|} \qquad \textrm{for $
	y \in \mathbf{R}^{n+1} \setminus A $}.
\end{equation*}
 For $ x \in \dmn \Der \bm{\nu}^\phi_A $, $ r = \bm{\delta}^\phi_A(x)$ and $ T = \Tan(S^\phi(A,r), x) $ the following  statements hold.
\begin{enumerate}
    \item[{\rm (a)}] $ \bm{\rho}^\phi_A(x) >1 $, $  \im \Der \bm{\nu}_A^\phi(x) \subseteq T  $ and $ \Der \bm{\nu}_A^\phi(x)(\bm{\nu}^\phi_A(x)) =0  $.
    \item[{\rm (b)}]  The maps  $\Der(\nabla \phi)(u(x))|T $ and $ \Der u(x)|T $ are self-adjoint (with respect to the underlying scalar product $\bullet$) automorphisms of $ T $,
    \begin{equation*}
\Der \bm{\nu}^\phi_A(x) = \Der(\nabla \phi)(u(x)) \circ \Der u(x) \quad \textrm{and}  \quad        \Der u(x)(\tau_1) \bullet \tau_2 = \frac{\Der^2\bm{\delta}_{A}^{\phi}(x)(\tau_1,\tau_2)}{|\nabla \bm{\delta}_{A}^{\phi}(x) (x)|}
	\quad \textrm{for $ \tau_1, \tau_2 \in T$}.
    \end{equation*}
	\item[{\rm (c)}] There is a basis $ \tau_1, \ldots , \tau_{n} $ of $ T $ of eigenvectors of $ \Der \bm{\nu}^\phi_A(x)| T $ and the corresponding  eigenvalues $ \chi_{1} \leq \ldots \leq \chi_{n} $ of $\Der \bm{\nu}^\phi_A(x) $ are real numbers
        such that
        \begin{equation}\label{boundsonchi}
            \frac{1}{(1-\bm{\rho}^\phi_A(x))r}\leq \chi_{i} \leq \frac{1}{r} \,.
        \end{equation}
        \end{enumerate}
 \end{Lemma}

\medskip
\subsection{Boundaries and perimeter}
Let $ A \subseteq \mathbf{R}^{n+1} $ and $ a \in A$. For $x\in\R^{n+1}$, $u\in \bS^n$ and $r>0$, we define the open halfspace through $x$ with (inner and outer, respectively) normal $u$ by $H^+(x,u):=\{z\in\R^{n+1}:(z-x)\bullet u> 0\}$ and $H^-(x,u):=\{z\in\R^{n+1}:(z-x)\bullet u< 0\}$.
Following \cite[Section 4.5.5]{MR0257325}, we say that a vector $ u \in \mathbf{S}^n$ is an \emph{exterior normal of $ A $ at $ a $} if
\begin{equation*}
    \lim_{r \to 0+}\frac{\mathcal{L}^{n+1}\big( H^+(a,u)\cap  U(a,r) \cap A\big)}{r^{n+1}} =0
\end{equation*}
and
\begin{equation*}
    \lim_{r \to 0+}\frac{\mathcal{L}^{n+1}\big(H^-(a,u)\cap   U(a,r) \setminus A\big)}{r^{n+1}} =0.
\end{equation*}
Clearly, in this definition $U(a,r)$ can be replaced by $B(a,r)$ and the open halfspaces can be replaced by the corresponding closed halfspaces.  Recall also from \cite[Section 4.5.5]{MR0257325} that if $ u$ and $ v $ are exterior unit normals of $ A$ at $ a $,  then $ u = v $. The set of points where the exterior normal of $ A $ exists is denoted by \index{B1@$\partial^m A$} $\partial^m A$; we define $ \bm{n}(A, \cdot) : \partial^m A \rightarrow \mathbf{S}^n$ to be the Euclidean  exterior normal map of $ A $. We  extend this definition by $\bm{n}(A, x)=0$ for $x\notin \partial^m A$.  Notice the equality $ \bm{n}(\mathcal{W}^\phi, \cdot) = \bm{n}^\phi $ on $\partial W^\phi$.

For an $ \mathcal{L}^{n+1}$ measurable set $ A \subseteq \mathbf{R}^{n+1}$ one can also consider the \index{B2@$\partial^\ast A$} \emph{essential boundary} $ \partial^\ast A $; see \cite[Definition 3.60]{MR1857292}. Recalling the notions of \emph{approximate discontinuity set} $ S_u $ and \emph{approximate jump set} $ J_u $ of a function $ u \in L^1_{\loc}(\mathbf{R}^{n+1})$, see \cite[Definitions 3.63 and 3.67]{MR1857292}, we notice that if $ A \subseteq \mathbf{R}^{n+1}$ is an $ \mathcal{L}^{n+1}$ measurable set, then $ \partial^\ast A = S_{\bm{1}_A} $ and $ \partial^m A = J_{\bm{1}_A}$, and it follows from \cite[Proposition 3.64]{MR1857292} and \cite[Proposition 3.69]{MR1857292} that $ \partial^m A $ and $ \partial^\ast A$ are Borel subsets of $ \mathbf{R}^{n+1}$, $ \bm{n}(A, \cdot) $ is a Borel function and
\begin{equation*}
    \partial^m A \subseteq \partial^\ast A.
\end{equation*}
Employing an argument similar to \cite[Lemma 5.1]{MR3978264}, one can still prove that if $ A \subseteq \mathbf{R}^{n+1}$ is an arbitrary set, then $ \partial^m A $ is a Borel subset of $ \mathbf{R}^{n+1}$ and $ \bm{n}(A, \cdot) $ is a Borel function.

We recall that  an $ \mathcal{L}^{n+1}$ measurable subset $ A $ of $ \mathbf{R}^{n+1}$ is a set of locally finite perimeter in $ \mathbf{R}^{n+1}$ if   the characteristic function $ \bm{1}_A$ is a function of locally bounded first variation (see \cite[Chapter 3]{MR1857292}). If $ A \subseteq \mathbf{R}^{n+1}$ is a set of finite perimeter, we denote by $ \mathcal{F}A $ the {\em reduced boundary} of $ A $ (see \cite[3.54]{MR1857292}). An important result of De Giorgi, see \cite[Theorem 3.59]{MR1857292}, implies that
\begin{equation*}
    \mathcal{F}A \subseteq \partial^m A.
\end{equation*}

A result of Federer (see \cite[Theorem 3.61]{MR1857292}) yields that if $ A $ is a set of  locally finite perimeter, then
\begin{equation}\label{eq: reduced an essential boundary}
    \mathcal{H}^n(\partial^\ast A \setminus \mathcal{F} A) =0.
\end{equation}
Another result of Federer (see \cite[Theorem 4.5.11]{MR0257325}) implies that if $ A \subseteq \mathbf{R}^{n+1}$  and $ \mathcal{H}^n(K \cap \partial A) < \infty$ for every compact set $ K \subset \mathbf{R}^{n+1}$, then $ A $ is a set of  locally finite perimeter.

\begin{Definition}
Let $A\subset\R^{n+1}$ be a Borel set with locally finite perimeter, and let $\phi $ be a uniformly convex $ \mathcal{C}^2 $-norm on $ \mathbf{R}^{n+1} $.
The \emph{$ \phi$-perimeter}  of $A$ is the Radon measure $\mathcal{P}^\phi(A, \cdot)$ supported in $ \partial A $ such that
\begin{equation*}
    \mathcal{P}^\phi(A, S) = \int_{S \cap \partial^mA}\phi(\bm{n}(A,x))\, d\mathcal{H}^n(x) \qquad \textrm{for Borel sets $ S \subseteq \mathbf{R}^{n+1} $.}
\end{equation*}
The total measure is denoted by \index{P1@$\mathcal{P}^\phi(A)$} $ \mathcal{P}^\phi(A, \mathbf{R}^{n+1}) = \mathcal{P}^\phi(A, \partial^m A) = \mathcal{P}^\phi(A) \in [0,\infty]$.
\end{Definition}

\medskip

Clearly, we have $\mathcal{P}^\phi(A)>0$ if and only if $\mathcal{H}^n(\partial^m A)>0$.

\medskip

The following lemma will be needed in Section \ref{Section: positive reach}. We refer to this section and the references provided there, for the definition and the basic facts concerning sets  of positive reach.

\begin{Lemma}\label{lem:finite perimeter and positive reach}
\begin{enumerate}
    \item[{\rm (a)}] \label{lem:finite perimeter and positive reach:1} If $A\subset\R^{n+1}$ is a Borel set of locally finite perimeter such that $ 0< \mathcal{L}^{n+1}(A) < \infty$, then $\mathcal{H}^n(\mathcal{F} A)>0$.
    \item[{\rm (b)}] \label{lem:finite perimeter and positive reach:2} If $A\subseteq \mathbf{R}^{n+1}$ is a set of positive reach, then $ \mathcal{H}^n(K \cap \partial A) < \infty$ for every compact set $ K \subseteq \mathbf{R}^{n+1}$, and consequently $ A $ is a set of locally finite perimeter.
\end{enumerate}
\end{Lemma}

\begin{proof}
Noting \cite[Theorem 3.59]{MR1857292} and \cite[Theorem 3.36]{MR1857292}, the  statement in (a) directly follows from the isoperimetric inequality in \cite[Theorem 3.46]{MR1857292}.


We now prove (b). We fix $ 0 < r < \reach(A)$ and note that $ S(A,r)$ is a closed $ \mathcal{C}^{1}$-hypersurface and $ \xi : =\bm{\xi}_A | S(A,r)$ is a Lipschitz map with $ \xi(S(A,r)) = \partial A$. Since $ \partial A \cap{B}(0,s) \subseteq \xi (\xi^{-1}(\partial A \cap {B}(0,s)))$ and $\xi^{-1}(\partial A \cap {B}(0,s)) \subseteq S(A,r) \cap {B}(0, r+s) $ for  $ s > 0 $, we infer that
\begin{equation*}
    \mathcal{H}^n({B}(0,s) \cap \partial A) \leq \Lip(\xi)^n \mathcal{H}^n(S(A,r) \cap {B}(0, s+r)) \quad \textrm{for  $ s > 0$}.
\end{equation*}
The right-hand side is evidently finite, since $ S(A,r) \cap {B}(0, s+r)$ is a compact subset of the closed $ \mathcal{C}^1$-hypersurface $ S(A,r)$.
\end{proof}
It will be sometimes useful to consider another notion of boundary: if $ A \subseteq \mathbf{R}^{n+1}$ is a closed set,  then we define the \emph{viscosity boundary} of $ A $ by \index{B3@$\partial^{v}A$}
\begin{equation*}
	\partial^{v}A =   \{a\in\partial A : \mathcal{H}^0(N(A,a)=1\}.
\end{equation*}
This is precisely the set of boundary points $a\in \partial A$ for which there is a unique (outer unit normal) \textbf{v}ector $u\in\mathbf{S}^n$ with $(a,u)\in N(A) $.
For  $s > 0$ we also define $ \partial^v_sA$ to be the set of points $ a \in \partial^v A$ such that there exists a closed  Euclidean ball $ B $ of radius $ s $ such that $ B \subseteq A$ and $ a \in \partial B$. We set \index{B4@$\partial^{v}_+A$}
\begin{equation*}
    \partial^v_+A = \bigcup_{s >0} \partial^v_sA.
\end{equation*}
\begin{Remark}\label{rmk: viscosity boundary}
We notice that $\partial^v_+ A \subseteq \partial^m A \cap \bm{p}(N(A)) \subseteq \partial^v A$ and
\begin{equation*}
N(A,a) = \{   \bm{n}(A,a) \} \qquad \textrm{for $ a \in \partial^m A \cap \bm{p}(N(A))$.}
\end{equation*}
Moreover, $ N(\mathbf{R}^{n+1} \setminus \Int(A), a) = \{-\bm{n}(A,a)\}$ for every $ a \in \partial^v_+A$. Finally, if $ s > 0 $ then  $ \partial^v_s A $ is a closed subset of $ \partial A $  and $  B(a - s \bm{n}(A,a), s) \subseteq A $ for every $ a \in \partial^v_s A$.
\end{Remark}

\medskip

The following lemma (or rather the consequence of it discussed in Remark \ref{rem:lipnorm}) will be relevant in the special case of sets with positive reach in Section \ref{sec:5.3} (see Remark \ref{rem:5.13}).

\begin{Lemma}\label{lem:lipbound}
Let $A \subseteq \mathbf{R}^{n+1}$ be a closed set and $x,y\in\partial A$. Let $0<r<s/2$. Suppose that  $u,v \in \mathbf{S}^n$ are such that
\begin{equation*}
   {B}(x - ru, r) \subseteq A,\quad {U}(x+su,s)\cap A=\emptyset, \quad {B}(y - rv, r) \subseteq A,\quad{U}(y+sv,s)\cap A=\emptyset.
\end{equation*}
Then
$$
|u-v| \leq \max \left\{  \frac{2(s - 2r)}{r(s-r)}, \sqrt{\frac{2}{r(s-r)}}\right\} |x-y|.
$$
\end{Lemma}

\begin{proof}
Since $y\in A$, we have $y\notin{U}(x+su,s)$, hence $|y-x-su|^2\ge s^2$, which yields
$$|y-x|^2+s^2-2s(y-x)\bullet u\ge s^2$$ or
\begin{equation}\label{ref1a1}
(y-x)\bullet u\le \frac{|y-x|^2}{2s}.
\end{equation}
By symmetry, we also have
\begin{equation}\label{ref1a2}
(x-y)\bullet v\le \frac{|x-y|^2}{2s}.
\end{equation}
Noting that $x - ru + rv \in {B}(x-ru, r) \subseteq A$, we conclude
from \eqref{ref1a2} that
\begin{flalign*}
    (x-ru+rv-y) \bullet v & \leq \frac{|x-y + r(v-u)|^2}{2s} \\
    & \leq \frac{1}{2s}|x-y|^2 + \frac{r^2}{2s}|u-v|^2 + \frac{r}{s}(x-y) \bullet (v-u).
\end{flalign*}
Exchanging $ x$ and $ y $ (and using \eqref{ref1a1}), we also get
\begin{equation*}
   (y-rv+ru-x) \bullet u \leq  \frac{1}{2s}|x-y|^2 + \frac{r^2}{2s}|u-v|^2 + \frac{r}{s}(y-x) \bullet (u-v).
\end{equation*}
Now we sum the last two inequalities to obtain
\begin{equation*}
    (x-y) \bullet (v-u) + r|v-u|^2 \leq \frac{1}{s}|x-y|^2 + \frac{r^2}{s}|u-v|^2 + \frac{2r}{s}(x-y) \bullet (v-u),
\end{equation*}
and we infer
\begin{equation}\label{ref:rem}
    r\Big( 1 - \frac{r}{s}\Big)|u-v|^2 \leq \frac{1}{s}|x-y|^2 + \Big(1 - \frac{2r}{s}\Big) |x-y||u-v|.
\end{equation}
If $\frac{1}{s}|x-y|^2  \leq  \big(1 - \frac{2r}{s}\big) |x-y||u-v|$, then
\begin{equation*}
    |u-v|^2 \leq \frac{2(s-2r)}{r(s-r)}|x-y||u-v|.
\end{equation*}
If $\frac{1}{s}|x-y|^2  \geq  \big(1 - \frac{2r}{s}\big) |x-y||u-v|$,  then
\begin{equation*}
    |u-v|^2 \leq \frac{2}{r(s-r)}|x-y|^2,
\end{equation*}
which yields the asserted upper bound.
\end{proof}

\begin{Remark}\label{rem:convexcase}
For convex bodies, Lemma \ref{lem:lipbound} is provided in \cite[Lemma 1.28]{Hug99} (see also \cite[Lemma 2.1]{MR1416712} for a less explicit statement and the literature cited there). In this special case, it can be seen from \eqref{ref:rem}  that the Lipschitz constant is bounded from above by $1/r$ (with $s=\infty$).
\end{Remark}

\begin{Remark}\label{rem:lipnorm}
For a closed set $A\subset\R^{n+1}$ and $r,s>0$, let
$X_{r,s}(A)$ denote the set of all $a\in\partial A$ such that ${B}(a - ru, r) \subseteq A$ and ${U}(a+su,s)\cap A=\emptyset$  for some $u\in\mathbf{S}^n$. Then $X_{r,s}(A)\subset\partial^m A$, for any $a\in X_{r,s}(A)$ the unit vector $u$ is equal to $\bm{n}(A,a)$ (and uniquely determined)  and
$\{\bm{n}(A,a)\}=N(A,a)\cap \mathbf{S}^n$. If $0<r\le s/4$, then Lemma \ref{lem:lipbound} yields that $\bm{n}(A,\cdot)|X_{r,s}(A)$ is Lipschitz continuous with Lipschitz constant bounded from above by $3/r$, since
$$
\frac{2(s-2r)}{r(s-r)}\le \frac{2s}{r(s-s/4)}=\frac{2s}{r\frac{3}{4}s}=\frac{8}{3}\frac{1}{r}<\frac{3}{r}
$$
and
$$
\sqrt{\frac{2}{r(s-r)}}\le \sqrt{\frac{2}{r\frac{3}{4}s}}=\sqrt{\frac{8}{3}\frac{1}{r2r}}=  \frac{2}{\sqrt{3}}\frac{1}{r}<\frac{2}{r}.
$$
\end{Remark}

\section{A Steiner-type formula for arbitrary closed sets}\label{sec: steiner formula}

Throughout this section, we assume  that $ \phi $ is a uniformly convex
$ \mathcal{C}^2 $-norm.  Recalling Lemma \ref{theo: distance twice diff} we start by introducing the following definition.

\subsection{Normal bundle and curvatures}

We start introducing the principal curvature of the level sets $ S^\phi(A,r) $ of the distance function $ \bm{\delta}^\phi_A$ taking the eigenvalues of the normal vector field $ \bm{\nu}^\phi_A $ defined in equation \eqref{eq: nu function}.
\begin{Definition}
	Suppose  $\varnothing\neq A \subseteq \mathbf{R}^{n+1} $ is closed, $ x \in \dmn (\Der \bm{\nu}^\phi_A) $  and $ r = \bm{\delta}^\phi_A(x) $. Then the eigenvalues (counted with their algebraic multiplicities) of $ \Der \bm{\nu}^\phi_A(x)| \Tan(S^\phi(A,r), x) $ are denoted by
	\index{O1@$\rchi^\phi_{A,i}$}
	\begin{equation*}
		\rchi^\phi_{A,1}(x) \leq \ldots \leq \rchi_{A,n}^\phi(x).
	\end{equation*}
\end{Definition}

\begin{Lemma}\label{rem: borel meas of chi}
The set $ \dmn (\Der \bm{\nu}^\phi_A)\subseteq  \Unp^\phi(A) $ is a Borel subset of $ \mathbf{R}^{n+1} $ and the functions $ \rchi^\phi_{A,i} : \dmn (\Der \bm{\nu}^\phi_A) \rightarrow \mathbf{R} $ are Borel functions for $ i\in\{1,\ldots,n\}$.
\end{Lemma}

\begin{proof}
	Let $ \mathcal{X} $ be the set of all  $ \varphi \in \Hom(\mathbf{R}^{n+1}, \mathbf{R}^{n+1}) $ with real eigenvalues. For each $ \varphi \in   \mathcal{X}  $ we define $ \lambda_0(\varphi) \leq \ldots \leq \lambda_n(\varphi) $ to be the eigenvalues of $ \varphi $ counted with their algebraic multiplicity, and then we define the map $ \lambda : \mathcal{X} \rightarrow \mathbf{R}^{n+1} $ by
	\begin{equation*}
		\lambda(\varphi) = (\lambda_0(\varphi), \ldots , \lambda_n(\varphi)) \qquad \textrm{for $ \varphi \in\mathcal{X} $.}
	\end{equation*}
We observe that  $\mathcal{X}$ is a Borel set and $ \lambda $ is a continuous map by \cite[Theorem A]{MR884486}. Moreover we notice that $ \dmn (\Der \bm{\nu}^\phi_A )=\dmn (\Der \bm{\xi}^\phi_ A) $ and that this is a Borel subset of $\mathbf{R}^{n+1}$  by Lemma \ref{lem : Borel measurability differential}. For each $ x \in \dmn (\Der \bm{\xi}^\phi_A) $, we have $ \Der \bm{\xi}^\phi_A(x)(\bm{\nu}^\phi_A(x)) =0 $ and  $\bm{\delta}^\phi_A(x)\cdot  \bm{\nu}^\phi_A(x)=x-\bm{\xi}_A^\phi(x)  $, hence
\begin{equation*}
	\lambda_0(\Der \bm{\xi}^\phi_A(x)) =0 \qquad \textrm{and} \qquad  \lambda_i(\Der \bm{\xi}^\phi_A(x)) = 1 - \bm{\delta}^\phi_A(x)\rchi^\phi_{A, n+1-i}(x) \geq 0 \quad \textrm{for $ i = 1, \ldots , n $,}
\end{equation*}
where also  \eqref{boundsonchi} was used.
Since the map $ \Der \bm{\xi}^\phi_A : \dmn (\Der \bm{\xi}^\phi_A) \rightarrow \mathcal{X} $ is a Borel function, we obtain the assertion.
\end{proof}

\begin{Remark}[\protect{cf.\ \cite[Lemmas 2.41 and  2.44]{kolasinski2021regularity}}]\label{rem: tangent ot level sets}
Suppose $\varnothing\neq A \subseteq \mathbf{R}^{n+1} $ is closed, $ x \in \Unp^\phi(A) $, $ r = \bm{\delta}^\phi_A(x) $, $ 0 < t < 1 $ and $ y = \bm{\xi}^\phi_A(x) + tr \bm{\nu}^\phi_A(x) = \bm{\xi}^\phi_A(x) + t(x-\bm{\xi}^\phi_A(x)) $. Then $ y \in \Unp^\phi(A) $,
	\begin{equation*}
	\Tan(S^\phi(A,r) , x) = \{v \in \mathbf{R}^{n+1} : v \bullet \nabla \bm{\delta}^\phi_A(x) =0\},
	\end{equation*}
\begin{equation*}
	\nabla \bm{\delta}^\phi_A(x) = \nabla \bm{\delta}^\phi_A(y) \qquad \textrm{and} \qquad  \Tan(S^\phi(A,r) , x) = \Tan(S^\phi(A,tr) , y).
\end{equation*}
\end{Remark}

\begin{Remark}\label{rem: tangent level sets and Wulff shapes}
	For  $(a,\eta)\in N^\phi(A) $ and  $ 0 < r < \bm{r}^\phi_A(a,\eta) $ we have
	\begin{equation*}
		\Tan(S^\phi(A,r), a + r\eta) = \Tan(\partial \mathcal{W}^\phi, \eta).
	\end{equation*}
Setting $ u = \frac{\nabla \bm{\delta}^\phi_A(a+ r\eta)}{|\nabla\bm{\delta}^\phi_A(a+ r\eta)|} $, this assertion follows from Remark \ref{rem: tangent ot level sets}, noting that (see \eqref{eq: normal wulff shape and gradient phi} and \eqref{eq: gradient and normal})
	\begin{equation*}
		\nabla \phi(u) = \nabla \phi(\nabla\bm{\delta}^\phi_A(a+ r\eta)) = \eta, \qquad \bm{n}^\phi(\eta) = u.
	\end{equation*}
\end{Remark}

\medskip

\begin{Lemma}\label{lem: existence of curvatures}
	Suppose  $\varnothing\neq A \subseteq \mathbf{R}^{n+1} $ is a closed set, $(a,\eta)\in N^\phi(A) $, $ 0 < r < s < \bm{r}^\phi_A(a,\eta) $ so that $ a + r \eta, a+s\eta \in \dmn \Der \bm{\nu}^\phi_K $ and $ \tau_1, \ldots , \tau_n \in \Tan(\partial \mathcal{W}^\phi, \eta) $.
	
	Then $ \Der \bm{\nu}^\phi_A(a+r\eta)\tau_i = \chi^\phi_{A,i}(a+r\eta)\tau_i $ for $ i = 1, \ldots , n $ if and only if $ \Der \bm{\nu}^\phi_A(a+s\eta)\tau_i = \chi^\phi_{A,i}(a+s\eta)\tau_i $ for  $ i = 1, \ldots , n $, in which case it holds that
	\begin{equation*}
	\frac{\rchi^\phi_{A,i}(a + r \eta)}{1-r\rchi^\phi_{A,i}(a + r \eta)} = \frac{\rchi^\phi_{A,i}(a + s \eta)}{1-s\rchi^\phi_{A,i}(a + s \eta)} \quad \textrm{for $ i = 1, \ldots , n $.}
	\end{equation*}
\end{Lemma}

\begin{proof}
	We define $ x = a + r\eta $, $ y = a + s\eta $ and $ t = \frac{r}{s} \in (0,1)$. We notice that $\bm{\xi}^\phi_A $ is differentiable at $ y $ and
\begin{equation*}
\Der\bm{\xi}^\phi_A(y)|\Tan(\partial \mathcal{W}^\phi, \eta) = \textrm{Id}_{\Tan(\partial \mathcal{W}^\phi, \eta)} - s \Der \bm{\nu}^\phi_A(y)|\Tan(\partial \mathcal{W}^\phi, \eta).
\end{equation*}
Let $ \xi : \mathbf{R}^{n+1} \setminus A \rightarrow A $  be such that $ \xi(z) \in \bm{\xi}^\phi_A(z) $ for $ z \in \mathbf{R}^{n+1} \setminus A $. Then define  $ \nu : \mathbf{R}^{n+1} \setminus A \rightarrow \partial \mathcal{W}^\phi $ by $ \nu(z) = \bm{\delta}^\phi_A(z)^{-1}(z - \xi(z)) $ for $ z \in \mathbf{R}^{n+1} \setminus A $. It follows from the strict convexity of $ \phi $ (see \cite[Remark 2.17]{kolasinski2021regularity}) that
	\begin{equation*}
		\nu(\xi(z) + t(z-\xi(z)) ) = \nu(z) \qquad \textrm{for $ z \in \mathbf{R}^{n+1} \setminus A $}.
	\end{equation*}
Differentiating this equality in $ y $, we obtain
	\begin{equation*}
		\Der \nu(x) \circ [  \Der \xi(y) + t (\textrm{Id}_{\mathbf{R}^{n+1}} - \Der \xi(y)) ] = \Der \nu(y).
	\end{equation*}
Assume now that $ \Der \nu(y)\tau_i = \rchi^\phi_{A,i}(y)\tau_i $ for  $ i = 1, \ldots , n $. Then
$$
\Der \bm{\xi}_A^\phi(y) \tau_i=\tau_i-s\cdot \Der  \bm{\nu}_A^\phi(y) \tau_i=(1-s\chi_{A,i}^\phi(y))\tau_i,
$$
and hence we get
\begin{equation*}
\rchi^\phi_{A,i}(y)\tau_i = [1 - (s-r)\rchi^\phi_{A,i}(y)]\Der \nu(x)\tau_i \qquad \textrm{for $ i = 1, \ldots , n $.}
\end{equation*}
Note that by \eqref{eq : r and rho} we have $\bm{\rho}_A^\phi(y) s-s>r-s$, and hence by the lower bound in \eqref{boundsonchi} we get
$$
1 + (r-s)\rchi^\phi_{A,i}(y) \ge 1+(r-s)\frac{1}{(1-\bm{\rho}_A^\phi(y) )s}=1-\frac{r-s}{\bm{\rho}_A^\phi(y) s-s}>0,
$$
that is, $1 - (s-r)\rchi^\phi_{A,i}(y) > 0 $ for  $ i = 1, \ldots , n $. We conclude that
\begin{equation*}
	\Der \nu(x)\tau_i = \rchi^\phi_{A,i}(x)\tau_i, \qquad  \rchi^\phi_{A,i}(x) = \frac{\rchi^\phi_{A,i}(y)}{1 - (s-r)\rchi^\phi_{A,i}(y)}
\end{equation*}
and
\begin{equation*}
\frac{\rchi^\phi_{A,i}(x)}{1-r\rchi^\phi_{A,i}(x)} = \frac{\rchi^\phi_{A,i}(y)}{1-s\rchi^\phi_{A,i}(y)}
\end{equation*}
for  $ i = 1, \ldots , n $ (where this common ratio may be infinite).

The last paragraph shows in particular that $ \Der \nu(x)|\Tan(\partial \mathcal{W}^\phi, \eta) $ and $ \Der \nu(y)|\Tan(\partial \mathcal{W}^\phi, \eta) $ have the same number $ k $ of distinct eigenvalues. Denoting by $ N_1(x), \ldots , N_k(x) $ and $N_1(y), \ldots , N_k(y) $ the eigenspaces of $ \Der \nu(x)|\Tan(\partial \mathcal{W}^\phi, \eta) $ and $ \Der \nu(y)|\Tan(\partial \mathcal{W}^\phi, \eta) $ respectively, we can also derive from the last paragraph the inclusions $ N_i(y) \subseteq N_i(x) $ for  $ i = 1, \ldots , k $. Since
\begin{equation*}
N_1(y) \oplus \cdots \oplus N_k(y) = \Tan(\partial \mathcal{W}^\phi, \eta) = N_1(x) \oplus \cdots \oplus N_k(x),
\end{equation*}
we conclude that $ N_i(y) = N_i(x) $ for $ i = 1, \ldots , n $ and the proof is completed.
\end{proof}

\begin{Definition}\label{def:curv2.8}
Suppose $\varnothing\neq A \subseteq \mathbf{R}^{n+1} $ is closed.	We define \index{N3@$\widetilde{N}^\phi(A) $}
	\begin{equation*}
		\widetilde{N}^\phi(A) = \{ (a,\eta)\in N^\phi(A) : \textrm{$a  + r\eta  \in \dmn(\Der \bm{\nu}^\phi_A) $ for some $r\in ( 0 , \arl{A}{\phi}(a, \eta)) $}\}
	\end{equation*}
	and \index{O2@$\kappa_{A,i}^\phi$}
	\begin{equation*}
		\kappa_{A,i}^\phi(a, \eta) = \frac{\rchi^\phi_{A,i}(a + r \eta)}{1-r\rchi^\phi_{A,i}(a + r \eta)}\in (-\infty,\infty]
	\end{equation*}
	for $(a, \eta) \in \widetilde{N}^\phi(A) $, $ 0 < r < \arl{A}{\phi}(a, \eta) $ with $ a+ r \eta \in \dmn \Der \bm{\nu}^\phi_A $ and $ i = 1, \ldots , n $. The numbers $\kappa_{A,i}^\phi(a, \eta)$, $i\in\{1,\ldots,n\}$, are called \emph{anisotropic (with respect to $\phi$) generalized curvatures of $A$ at $(a,\eta)$} or \emph{generalized $\phi$-curvatures of $A$ at $(a,\eta)$}.
\end{Definition}

\begin{Remark}\label{rem: principla curvatures basic rem}
	Lemma \ref{lem: existence of curvatures} demonstrates that the definition of $ \kappa^\phi_{A,i}(a, \eta) $ does not depend on the choice of $ r $ and Lemma \ref{theo: distance twice diff} implies that $ \mathcal{H}^n(N^\phi(A) \setminus \widetilde{N}^\phi(A) ) =0 $. Moreover, Lemma \ref{theo: distance twice diff} ensures that all points of the open segment $ \{a + r \eta : 0 < r < \arl{A}{\phi}(a, \eta) \} $ are points of differentiability of $ \bm{\nu}^\phi_A $ for every $ (a, \eta) \in \widetilde{N}^\phi(A)$; in other words,
	\begin{equation*}
	    \widetilde{N}^\phi(A) = \{(a, \eta) \in N^\phi(A) : \arl{A}{\phi}(a, \eta) > 0,\,  a + r\eta \in \dmn \Der \bm{\nu}^\phi_A \, \textrm{for every $r\in (  0 , \arl{A}{\phi}(a, \eta) ) $}\}.
	\end{equation*}
Combining Lemma \ref{rem: borel meas of chi} and Lemma \ref{theo: distance twice diff} it follows that $ \widetilde{N}^\phi(A) $	is a Borel subset of $ N^\phi(A)$. Moreover by Lemma \ref{rem: borel meas of chi} we deduce that the functions $ \kappa^\phi_{A,i} $ are Borel measurable. Since the eigenvalues $\rchi^\phi_{A,i}(a + r \eta)$, $i=1,\ldots,n$, are arranged in increasing order, we also have $-\infty <\kappa_{A,1}^\phi(a, \eta)\le \ldots\le \kappa_{A,n}^\phi(a, \eta)\le \infty$.
\end{Remark}

\begin{Remark}\label{eq: curvatures and reach}
Using \eqref{eq : r and rho},  Lemma \ref{theo: distance twice diff}, Lemma \ref{lem: distance twice diff} (a), \eqref{boundsonchi} and Definition \ref{def:curv2.8}, we obtain that
	\begin{equation*}
		-\frac{1}{\bm{r}^\phi_A(a,\eta)} \le \kappa^\phi_{A,i}(a,\eta) \leq + \infty
	\end{equation*}
	for  $(a,\eta)\in \widetilde{N}^\phi(A) $ and $ i = 1, \ldots , n $.
\end{Remark}

\begin{Lemma}\label{lem: tangent of normal bundle}
Let  $\varnothing\neq A \subseteq \mathbf{R}^{n+1} $ be closed. Suppose $ \tau_i : \widetilde{N}^\phi(A)\rightarrow \mathbf{R}^{n+1} $, for $ i = 1, \ldots , n $, are defined so that $ \tau_1(a,\eta), \ldots, \tau_{n}(a,\eta) $ form a basis of $\Tan(\partial \mathcal{W}^\phi, \eta)$ with
\begin{equation}\label{eq:hneu1}
	\Der \bm{\nu}^\phi_A(a+ r \eta)(\tau_i(a,\eta))= \rchi^\phi_{A,i}(a + r \eta)\tau_i(a,\eta) \qquad \textrm{for $ i = 1, \dots , n $ and $ 0 < r < \arl{A}{\phi}(a,\eta) $}.
\end{equation}
Let $ \zeta_i : \widetilde{N}^\phi(A) \rightarrow \mathbf{R}^{n+1} \times \mathbf{R}^{n+1} $, for $ i = 1, \ldots , n $,  be defined so that
 \begin{equation}\label{eq:hneu2}
 \zeta_i(a,\eta) =
 \begin{cases}
 (\tau_i(a,\eta),  \kappa^\phi_{A,i}(a, \eta)\tau_i(a,\eta)), & \textrm{if $ \kappa^\phi_{A,i}(a,\eta) < \infty $},\\
 	(0, \tau_i(a,\eta)), & \textrm{if $\kappa^\phi_{A,i}(a, \eta) = + \infty$.}
 \end{cases}
 \end{equation}

Let $ W \subseteq N^\phi(A) $ be an $ \mathcal{H}^n $ measurable set  with $ \mathcal{H}^n(W) < \infty $.
Then, for $ \mathcal{H}^n $ almost all $(a,\eta)\in W $, the set $\Tan^n(\mathcal{H}^n\restrict W, (a,\eta)) $ is an $ n $-dimensional linear subspace and $\zeta_1(a,\eta), \ldots , \zeta_n(a,\eta) $ form a basis of $\Tan^n(\mathcal{H}^n\restrict W, (a,\eta)) $. Moreover,
\begin{equation*}
	\ap J^W_n\bm{p}(a,\eta) = \frac{|\tau_1(a,\eta) \wedge \ldots \wedge \tau_n(a,\eta)|}{|\zeta_1(a,\eta) \wedge \ldots \wedge \zeta_n(a,\eta)|}\, \bm{1}_{\widetilde{N}^\phi_n(A)}(a,\eta)
\end{equation*}
for $ \mathcal{H}^n $ almost all $(a,\eta)\in W $.
\end{Lemma}

\begin{proof}
Assume that $ W \subseteq N^\phi(A) $ is $ \mathcal{H}^n $ measurable with $ \mathcal{H}^n(W)< \infty $ and $ \lambda > 1 $. For $ r > 0 $ we define
\begin{equation*}
W_r = \{(a,\eta)\in W:  \bm{r}^\phi_A(a,\eta)=\arl{A}{\phi}(a,\eta) \geq \lambda r  \},
\end{equation*}
which is an $ \mathcal{H}^n $ measurable subset of $ W $. Furthermore, we denote by $ W^\ast_r $ the set of all $(a,\eta)\in W_r $ such that $ \Tan^n(\mathcal{H}^n\restrict W, (a,\eta)) $ is  an $ n $-dimensional linear subspace  and $ \Tan(\mathcal{H}^n \restrict W, (a,\eta))= \Tan^n(\mathcal{H}^n \restrict W_r, (a,\eta)) $. It follows from \cite[3.2.19]{MR0257325} that $ \mathcal{H}^n(W_r \setminus W^\ast_r) =0 $. Moreover, by the  coarea formula  there exists $ J \subseteq (0,\infty) $ with $ \mathcal{H}^1(J) =0 $ such that $ \mathcal{H}^n(S^\phi(A,r) \setminus \Unp^\phi(A)) =0 $ for  $ r \notin J $.

We fix $ r > 0 $, $ r \notin J $, and  define
\begin{equation*}
	M_r = \{a + r\eta: (a,\eta) \in W_r   \}.
\end{equation*}
It follows from \eqref{eq : r and rho} that $ M_r \subseteq \{x\in S^\phi(A,r)  : \bm{\rho}^\phi_A(x)\geq \lambda    \} $ (see also \eqref{eq:frho1}); moreover, $ M_r $ is $ \mathcal{H}^n $ measurable. By Lemma \ref{theo: projectionLipschitz} the function $ \bm{\psi}^\phi_A|M_r $ is Lipschitz; moreover, we notice that $ \bm{\psi}^\phi_A(M_r) = W_r $ and $ (\bm{\psi}^\phi_A|M_r)^{-1}(a,\eta ) = a + r\eta $ for $ (a,\eta)\in W_r $. We denote by $M^\ast_r $ the set of all $ x \in M_r $ such that $ \Tan(S^\phi(A,r), x) $ is an $ n $-dimensional linear subspace and $ \Tan^n(\mathcal{H}^n \restrict M_r,x) = \Tan(S^\phi(A,r), x) $. It follows from Remark \ref{rem: tangent ot level sets} and \cite[3.2.19]{MR0257325} that $ \mathcal{H}^n(M_r \setminus M^\ast_r) =0 $. We conclude that
\begin{equation*}
\mathcal{H}^n(W_r \setminus (W^\ast_r \cap \bm{\psi}^\phi_A(M^\ast_r)))=0.
\end{equation*}
Further, if $(a,\eta)\in (W^\ast_r \cap \bm{\psi}^\phi_A(M^\ast_r)) \cap \widetilde{N}^\phi(A) $ it follows from \cite[Lemma B.2]{MR4117503} and Remark \ref{rem: tangent level sets and Wulff shapes} that  $ \{\tau_1(a,\eta), \ldots , \tau_n(a,\eta) \}$ is a basis of $ \Tan(S^\phi(A,r), a + r\eta) $,
\begin{equation*}
\Der \bm{\psi}^\phi_A(a+r\eta)[\Tan(S^\phi(A,r), a+r\eta)] = \Tan^n(\mathcal{H}^n \restrict W, (a,\eta))
\end{equation*}
and
\[
\Der \bm{\psi}^\phi_A(a+r\eta)(\tau_i(a,\eta)) =
\begin{cases}
	\frac{1}{1 + r \kappa^\phi_{A,i}(a,\eta)}\zeta_i(a,\eta), & \textrm{if $ \kappa^\phi_{A,i}(a,\eta) < \infty $},\\
	\frac{1}{r}\zeta_i(a,\eta), & \textrm{if $\kappa^\phi_{A,i}(a, \eta) = + \infty$.}
\end{cases}
\]
This proves that $\{\zeta_1(a,\eta), \ldots , \zeta_n(a,\eta)\}  $ is a basis of $\Tan^n(\mathcal{H}^n \restrict W, (a,\eta))$ for $ \mathcal{H}^n $ a.e.\ $(a,\eta)\in W_r $ and for every $ r \notin J $.

Since $ W \setminus \bigcup_{r > 0}W_r $ has $\mathcal{H}^n$ measure zero and $ W_r \subseteq W_s $ for $ 0 < s < r $, there exists a sequence $ r_i \searrow 0 $, $r_i \notin J$, so that $ W \setminus \bigcup_{i=1}^\infty W_{r_i} $ has $\mathcal{H}^n$ measure zero, which completes the proof.
\end{proof}

\medskip

\begin{Remark}\label{rem:altdescreigen}
The existence of a basis $ \tau_1(a,\eta), \ldots, \tau_{n}(a,\eta) $ of $\Tan(\partial \mathcal{W}^\phi, \eta)$ such that \eqref{eq:hneu1} is satisfied is based on Lemma \ref{lem: distance twice diff} (c).

We now provide an alternative but equivalent description which leads to some additional information that will be useful in the proof of Theorem \ref{thm:disint}.
Let $(a,\eta)\in 	\widetilde{N}^\phi(A) $, $r\in (0,\arl{A}{\phi}(a,\eta))$ and $x=a+r\eta$. Let $ u : \mathbf{R}^{n+1} \setminus A \rightarrow \mathbf{S}^n$ be the map defined in Lemma \ref{lem: distance twice diff} and notice that $ u(x) = \bm{n}^\phi(\eta)$. Hence we have $\Tan(\partial \mathcal{W}^\phi, \eta)=u(x)^\perp =\{z\in\R^{n+1}:z\bullet u(x)=0\}$. We define a symmetric bilinear form (i.e.\ an inner product)
\begin{equation*}
    B_\eta : \Tan(\partial \mathcal{W}^\phi, \eta) \times \Tan(\partial \mathcal{W}^\phi, \eta) \rightarrow \mathbf{R}
\end{equation*}
setting
$$
B_\eta(\tau, \sigma) = \Der \bm{n}^\phi(\eta)(\tau)\bullet \sigma=[\Der(\nabla\phi)(u(x))|u(x)^\perp]^{-1}(\tau)\bullet \sigma
$$
for $\tau,\sigma\in \Tan(\partial \mathcal{W}^\phi, \eta)=u(x)^\perp$.
Using Lemma \ref{lem: distance twice diff}(b), we see that
$$
B_\eta(	\Der \bm{\nu}^\phi_A(x)(\tau),\sigma)=\Der u(x)(\tau)\bullet \sigma = \tau \bullet \Der u(x)(\sigma) = B_\eta(\tau, \Der \bm{\nu}_A^\phi(x)(\sigma))\quad\text{for } \tau,\sigma\in u(x)^\perp.
$$
This shows that $\Der \bm{\nu}^\phi_A(x)|u(x)^\perp$ is self-adjoint with respect to $B_\eta$. Hence there is an orthonormal basis $\tau_1(a,\eta),\ldots,\tau_n(a,\eta)$ of $u(x)^\perp$ with respect to $B_\eta$ consisting of eigenvectors of $\Der \bm{\nu}^\phi_A(x)$. Henceforth we can assume that $\tau_1(a,\eta),\ldots,\tau_n(a,\eta)$ are chosen in this way.

We now consider the natural extension of the inner product $ B_\eta $ to the product space $ \Tan(\partial \mathcal{W}^\phi, \eta) \times \Tan(\partial \mathcal{W}^\phi, \eta)=\bm{n}^\phi(\eta)^\perp\times\bm{n}^\phi(\eta) ^\perp $ given by
\begin{equation*}
    \overline{B}_\eta((\tau_1, \sigma_1), (\tau_2, \sigma_2)) = B_\eta(\tau_1, \sigma_1) + B_\eta(\tau_2, \sigma_2)
\end{equation*}
for $(\tau_1, \sigma_1), (\tau_2, \sigma_2) \in \Tan(\partial \mathcal{W}^\phi, \eta) \times \Tan(\partial \mathcal{W}^\phi, \eta)$ and for every $ \eta \in \partial \mathcal{W}^\phi $. With respect to this inner product, the linearly independent vectors $\zeta_1(a,\eta),\ldots,\zeta_n(a,\eta)$ from \eqref{eq:hneu2} are pairwise orthogonal. Let  $ |\cdot|_\eta $ denote the norm induced by  $\overline{B}_\eta$  on $\bigwedge_m(\bm{n}^\phi(\eta)^\perp\times \bm{n}^\phi(\eta)^\perp)$, for $m\in\{1,\ldots,n\}$. Then using continuity and compactness one can show that
\begin{align*}
   & 0<c : = \inf\left\{|A|_\eta:A\in \mbox{$\bigwedge_m$}(\bm{n}^\phi(\eta)^\perp\times\bm{n}^\phi(\eta) ^\perp)\,,|A|=1\,,\eta\in \partial \mathcal{W}^\phi,\, m \in\{ 1, \ldots , n\}\right\}\\
    &\qquad \quad  \le \sup\left\{|A|_\eta:A\in \mbox{$\bigwedge_m$}(\bm{n}^\phi(\eta)^\perp\times\bm{n}^\phi(\eta) ^\perp)\,,|A|=1\,,\eta\in \partial \mathcal{W}^\phi,\, m \in\{ 1, \ldots , n\}\right \} =:C<\infty.
\end{align*}
Thus we obtain the inequality
\begin{align}\label{eq:dlate1}
    \left|\zeta_1(a,\eta)\wedge\cdots\wedge \zeta_n(a,\eta)\right|&\ge \frac{1}{C}
      \left|\zeta_1(a,\eta)\wedge\cdots\wedge \zeta_n(a,\eta)\right|_\eta\nonumber\\
      &=\frac{1}{C}
      \left|\zeta_1(a,\eta)\wedge\cdots\wedge \zeta_m(a,\eta)\right|_\eta\cdot
       \left|\zeta_{m+1}(a,\eta)\wedge\cdots\wedge \zeta_n(a,\eta)\right|_\eta\nonumber\\
      &\ge \frac{c^2}{C}
      \left|\zeta_1(a,\eta)\wedge\cdots\wedge \zeta_m(a,\eta)\right| \cdot
       \left|\zeta_{m+1}(a,\eta)\wedge\cdots\wedge \zeta_n(a,\eta)\right| ,
\end{align}
which will be used in the proof of Theorem \ref{thm:disint}.
\end{Remark}

\paragraph{}In the following, it will be useful to distinguish how many of the generalized curvatures are finite for a given $(a,\eta)\in\widetilde{N}^\phi(A)$.

\begin{Definition}\label{Def:finitecurv}
Let $\varnothing\neq A \subseteq \mathbf{R}^{n+1} $ be closed.	We define \index{N4@$\widetilde{N}^\phi_d(A)$}
	\begin{equation*}
		\widetilde{N}^\phi_d(A) =   \{ (a,\eta)\in \widetilde{N}^\phi(A)    : \kappa^\phi_{A, d}(a,\eta) < \infty ,  \kappa^\phi_{A, d+1}(a,\eta) = \infty  \} \quad \textrm{for $ d \in\{ 1, \ldots , n-1\} $,}
	\end{equation*}
	\begin{equation*}
	   	\widetilde{N}^\phi_0(A) =  \{ (a,\eta) \in \widetilde{N}^\phi(A)   : \kappa^\phi_{A, 1}(a,\eta) = \infty  \}
	\end{equation*}
	and
	\begin{equation*}
	    	\widetilde{N}^\phi_n(A) = \{ (a,\eta) \in\widetilde{N}^\phi(A)   : \kappa^\phi_{A, n}(a,\eta) < \infty  \}.
	\end{equation*}
For $a\in A$ we set $ \widetilde{N}^\phi_d(A,a) = \{\eta:(a,\eta)\in\widetilde{N}^\phi_d(A)  \}$. Moreover, for   $ j \in\{0, \ldots , n\}  $ we define  the map $ E^\phi_{A,j} : \widetilde{N}^\phi(A) \rightarrow \mathbf{R} $ by \index{O3@$E^\phi_{A,j}$}
\[
E^\phi_{A,j}(a,\eta) =
\begin{cases}
	\displaystyle{\sum_{\lambda \in \Lambda(d, j)}\,\kappa^\phi_{A,\lambda(1)}(a,\eta)} \cdots \kappa^\phi_{A, \lambda(j)}(a, \eta), & \textrm{if $(a,\eta) \in \widetilde{N}^\phi_d(A) $ and $ d \geq j $},\\
	0 , & \textrm{if $(a,\eta) \in \widetilde{N}^\phi_d(A) $ and $ d < j $},
\end{cases}
\]
and for $ j =0 $ this means that $ E^\phi_{A,0} \equiv 1 $. Finally, for  $ r \in\{0, \ldots , n\} $ we define the \emph{$ r$-th $ \phi $-mean curvature of $ A $} as \index{O4@$\bm{H}^\phi_{A,r}$}
\begin{equation*}
	\bm{H}^\phi_{A,r} =\sum_{j=0}^{r}  E^\phi_{A,j}\, \bm{1}_{\widetilde{N}^\phi_{j+n-r}(A)}=\sum_{i=0}^{r}  E^\phi_{A,r-i}\, \bm{1}_{\widetilde{N}^\phi_{n-i}(A)}.
\end{equation*}
\end{Definition}

\begin{Remark}\label{remcurvdef}
The sets $\widetilde{N}^\phi_d(A)$ and functions $\bm{H}^\phi_{A,r}$ introduced in Definition \ref{Def:finitecurv} are Borel measurable (see Remark \ref{rem: principla curvatures basic rem}). In particular, by definition we have  $\bm{H}^\phi_{A,0}=\bm{1}_{\widetilde{N}^\phi_{n}(A)}$.
\end{Remark}

The next result will be used repeatedly in Section \ref{Section: positive reach},  in the special case of sets with positive reach. We prepare it by
recalling a fact from linear algebra, which can be easily deduced from the standard spectral theorem (cf.\ \cite[Remark 2.25]{MR4160798}).

\begin{Remark}\label{rem: diagonalization}
Suppose $ X $ is an $ n $-dimensional Hilbert space, $ M_1 , M_2 \in \Hom(X,X) $ are self-adjoint and $ M_1 $ is positive definite. Then there exist $ C \in \Hom(X,X) $ self-adjoint and positive definite, $ n $ real numbers $ \lambda_1 \leq \ldots \leq \lambda_n $ and an orthonormal basis $ v_1, \ldots , v_n $ of $ X $ such that $ C \circ C = M_1 $ and
\begin{equation*}
(M_1 \circ M_2)(C(v_i)) = \lambda_i C(v_i) \qquad \textrm{for $ i = 1, \ldots , n $}.
\end{equation*}
\end{Remark}

\begin{Lemma}\label{lem suff condition for finite curvature}
For every closed set $ \varnothing \neq A \subseteq \mathbf{R}^{n+1}$ the following statements hold.
\begin{enumerate}[{\rm (a)}]
\item\label{lem suff condition to finite curvature 0} Suppose $ s,r > 0 $, $ x \in \dmn(\Der \bm{\nu}^\phi_A) \cap S^\phi(A,r) $ and $ V $ is an open neighbourhood of $ x $ in $ \mathbf{R}^{n+1}$ such that $ U^\phi(x-s\bm{\nu}^\phi_A(x), s) \cap S^\phi(A,r) \cap V = \varnothing $. Then
\begin{equation*}
    \chi^\phi_{A,i}(x) \leq \frac{1}{s} \qquad \textrm{for  $ i=1, \ldots , n$.}
\end{equation*}
    \item\label{lem suff condition to finite curvature 1} If $(a, \eta) \in \widetilde{N}^\phi(A)$, $ a \in \partial^v_+ A $ and $ r > 0 $ such that $ U^\phi(a-r\eta, r) \subseteq \Int A $, then $(a, \eta) \in \widetilde{N}^\phi_n(A)$ and
    \begin{equation*}
        \kappa^\phi_{A,n}(a, \eta) \leq \frac{1}{r}.
    \end{equation*}
    \item\label{lem suff condition to finite curvature 2} Suppose $ a \in A^{(n)} \setminus \partial^v A $, $(a, \eta) \in \widetilde{N}^\phi(A)$,  $ W \subseteq \mathbf{R}^{n+1} $ is an open set with $ a \in W $ and  $ f : W \cap (a +\eta^\perp) \rightarrow \mathbf{R}$ is a function such that $ f(a)=0$, $ f $ is continuous at $ a $ and $ \graph(f) \subseteq A $.  Then $(a, \eta) \in \widetilde{N}^\phi_n(A)$.
\end{enumerate}
\end{Lemma}

\begin{proof} \ref{lem suff condition to finite curvature 0}
 Choose a function $ \xi : \mathbf{R}^{n+1} \rightarrow \mathbf{R}^{n+1}$ with $ \xi(y) \in \bm{\xi}^\phi_A(y) $ for every $ y \in \mathbf{R}^{n+1}$ and define
\begin{equation*}
    u(y) = \frac{\nabla \phi^\ast(y - \xi(y))}{|\nabla \phi^\ast(y - \xi(y))|  } \qquad \textrm{for $ y \in \mathbf{R}^{n+1} \setminus A $}.
\end{equation*}
Note that \eqref{eq: gradient and normal} and Remark \ref{rem: tangent ot level sets} yield $ u(x) = \frac{\nabla \bm{\delta}^\phi_A(x)}{|\nabla \bm{\delta}^\phi_A(x)|}$. Define the linear subspace $ T = \Tan(S^\phi(A,r), x)$ and the open convex set $ U = U^\phi(x-s\bm{\nu}^\phi_A(x), s)$. Since $ x \in \dmn(\Der \bm{\nu}^\phi_A) $, recalling \cite[Lemma 2.44]{kolasinski2021regularity}, we deduce that $ u(x) \perp  T =\Tan(\partial  U, x) $ and there exist a relatively open subset $W$ of $ x + T $ with $ x \in W $ and a continuous function $ f : W \rightarrow \mathbf{R}$ that is pointwise twice differentiable at $ x $ such that $ f(x)=0$, $ \Der f(x)=0$,
\begin{equation*}
    \{ b - f(b)u(x): b \in W\} \subseteq S^\phi(A,r) \cap V \quad \textrm{and} \quad \Der^2f(x) = \frac{1}{|\nabla \bm{\delta}^\phi_A(x)|}\Der^2 \bm{\delta}^\phi_A(x)|(T \times T).
\end{equation*}
Let $ g : W \rightarrow \mathbf{R}$ be the smooth function such that $ g(x) =0 $, $ \Der g(x)=0$ and $\{  b - g(b)u(x): b \in W\} \subseteq \partial U $. Since $ u(x) \bullet \bm{\nu}^\phi_A(x) = |\nabla \bm{\delta}^\phi_A(x)|^{-1}>0$ and $ U \cap S^\phi(A,r) \cap V = \varnothing $, we infer that $ f(b) \leq g(b) $ for  $ b \in W $. Let $ v : \partial U \rightarrow \mathbf{S}^n$ be the exterior unit normal of $ \partial U$ and notice that $v(x) = u(x) $. From \cite[Lemma 2.45]{kolasinski2021regularity} and Lemma \ref{lem: distance twice diff}(2b) we infer that
\begin{equation*}
    \Der u(x)(\tau) \bullet \tau = \Der^2f(x)(\tau, \tau) \leq \Der^2 g(x)(\tau, \tau) =  \Der v(x)(\tau) \bullet \tau
\end{equation*}
for  $ \tau \in T$. Recalling again Lemma \ref{lem: distance twice diff}(2b),  we can apply Remark \ref{rem: diagonalization} with the automorphisms $ M_1 = \Der(\nabla \phi)(u(x))|T$ and $ M_2 = \Der u(x)|T $ on $T$ (and the induced scalar product of $\R^{n+1}$) to infer the existence of a selfadjoint linear map $ C : T \rightarrow T$ and an orthonormal basis $ \tau_1, \ldots , \tau_n $ of $ T $ such that $ \Der(\nabla \phi)(u(x))|T = C \circ C $ and $ C(\tau_1), \ldots , C(\tau_n)$ is a basis of eigenvectors of $ \Der \bm{\nu}^\phi_A(x)|T $, i.e.\ $ \Der \bm{\nu}^\phi_A(x)(C(\tau_i)) = \chi^\phi_{A,i}(x)\, C(\tau_i) $ for $ i \in \{1, \ldots , n \}$. Employing \cite[Lemma 2.33]{MR4160798}, we notice that
\begin{flalign*}
    s^{-1}|C(\tau)|^2 & = (C^{-1} \circ \Der(\nabla \phi \circ v)(x) \circ C)(\tau) \bullet \tau  = (C^{-1} \circ \Der(\nabla \phi)(v(x)) \circ \Der v(x) \circ C)(\tau) \bullet \tau \\
    & = (C \circ \Der v(x) \circ C)(\tau) \bullet \tau = \Der v(x)( C(\tau)) \bullet C(\tau) \\
   & \geq \Der u(x)(C(\tau)) \bullet C(\tau)  = C^{-1} \circ \Der(\nabla \phi \circ u)(x) \circ C)(\tau) \bullet \tau\\
   & = (C^{-1} \circ \Der \bm{\nu}^\phi_A(x) \circ C)(\tau) \bullet \tau
\end{flalign*}
for $ \tau \in T $. Evaluating this inequality at $ \tau = \tau_i $ for $ i = 1, \ldots, n $ we obtain the conclusion.

\medskip

\ref{lem suff condition to finite curvature 1}
 Choose $  0 < s < \arl{A}{\phi}(a, \eta) $. We observe that $ S^\phi(A,s) \cap  U^\phi(a-r\eta,r+s) = \varnothing$ and $ a+s\eta \in S^\phi(A, s) \cap \partial \mathbf{B}^\phi(a-r\eta,r+s)$. Since $a + s \eta \in \dmn \Der \bm{\nu}^\phi_A $ and $ \bm{\nu}^\phi_A(a+s\eta) = \eta $, we can apply \ref{lem suff condition to finite curvature 0} with $ x $ and $ s $ replaced by $ a + s \eta $ and $ r + s $ respectively,  to conclude that
\begin{equation*}
    \chi^\phi_{A,n}(a+s\eta) \leq \frac{1}{r+s} \quad \textrm{and} \quad \kappa^\phi_{A,n}(a,\eta) = \frac{\chi^\phi_{A,n}(a+s\eta)}{1-s\chi^\phi_{A,n}(a+s\eta)} \leq \frac{\frac{1}{r+s}}{1-\frac{s}{r+s}}=\frac{1}{r} < +\infty.
\end{equation*}

\medskip

\ref{lem suff condition to finite curvature 2}
 We may assume that $ a=0 $, and  we denote by $ \pi $ the orthogonal projection onto $ \eta^\perp$. We notice that there is some $ s > 0 $ such that $ U^\phi(-s\eta, s) \cap A = \varnothing$. Then we choose $  0 < r < \inf\{s, \arl{A}{\phi}(a, \eta) \}$ and we notice that $ U^\phi(r\eta, r) \cap A = \varnothing $. Set $U_0=\pi(U^\phi(-s\eta, s))$ and let $ g : U_0 \rightarrow \mathbf{R}$ be the continuous function such that $ g(0)=0$, \begin{equation*}
    \{b + g(b) \eta : b \in U_0\} \subseteq \partial U^\phi(-s\eta, s) \quad \textrm{and} \quad U^\phi(-s\eta, s) \subseteq \{b + t \eta: t < g(b), \; b \in U_0\}.
\end{equation*}
Since $ U^\phi(-s\eta,s) \cap A = \varnothing $ it follows from the continuity of $ f $ at $ 0$ that there is $ \epsilon > 0 $ such that  $   g(b)\leq f(b) \leq \frac{r}{2} $ for every $ b \in \eta^\perp$ with $ |b| < \epsilon $. We can assume that $ W =   \{b \in \eta^\perp: |b|< \epsilon \} \subseteq U_0$. Since $ r \eta \in \dmn \Der \bm{\nu}^\phi_A $ it follows from \cite[Lemma 2.44]{kolasinski2021regularity} that there exists an open neighbourhood $ V $ of $ r \eta $ in $ \mathbf{R}^{n+1}$ such that $ S^\phi(A,r) \cap V$ is equal to the graph of a continuous function. Therefore we can choose $ V $ small enough so that $ \pi(V) \subseteq W $ and
\begin{equation*}
    S^\phi(A,r) \cap V \subseteq \Big\{y : |(y-r\eta) \bullet \eta | < \frac{r}{4}\cdot |\eta|^2\Big\}.
\end{equation*}
We claim that $ [(S^\phi(A,r) \cap V ) - r\eta] \cap U^\phi(-s\eta, s) = \varnothing $. In fact, assume by contradiction that there is some $ z \in [(S^\phi(A,r) \cap V ) - r\eta] \cap U^\phi(-s\eta, s)$. Then we obtain that
\begin{equation*}
| z \bullet \eta | < \frac{r}{4}\cdot |\eta|^2  \quad \textrm{and} \quad
\frac{z \bullet \eta}{|\eta|^2} < g(\pi(z)) \leq f(\pi(z)) \leq \frac{r}{2} <   \frac{(z + r\eta) \bullet \eta}{|\eta|^2}.
\end{equation*}
It follows that $ \pi(z) + f(\pi(z))\eta $ lies in the open segment joining $ z $ with $ z + r\eta $. Since $\pi(z) + f(\pi(z))\eta \in A  $ and $ z + r\eta \in S^\phi(A,r)$, it follows that $r =\bm{\delta}^\phi_A(z + r\eta) < r$, which is a contradiction. Now we can use \ref{lem suff condition to finite curvature 0} to conclude that
\begin{equation*}
    \rchi^\phi_{A,n}(r\eta) \leq \frac{1}{s} \quad \textrm{and} \quad \kappa^\phi_{A,n}(a, \eta) \leq \frac{1}{s-r} < \infty.
\end{equation*}
\end{proof}

\begin{Remark}\label{rmk: curvature positive reach}
It follows from Lemma \ref{lem suff condition for finite curvature} that $ (N^\phi(A)| \partial^v_+A) \setminus \widetilde{N}_n^\phi(A) \subseteq N^\phi(A) \setminus \widetilde{N}^\phi(A)$. In particular,
\begin{equation*}
    \mathcal{H}^n\big((N^\phi(A)| \partial^v_+A) \setminus \widetilde{N}_n^\phi(A)\big) =0.
\end{equation*}
\end{Remark}

\subsection{Steiner-type formula and disintegration of Lebesgue measure}

The following result extends the Steiner-type formula from \cite{MR1782274} (see also the literature cited there such as \cite{MR0534512}) to the anisotropic setting.

\begin{Theorem}[Steiner-type formula for closed sets]\label{theo: Steiner closed}
Let	$\varnothing\neq A \subseteq \mathbf{R}^{n+1} $ be a closed set, let $ \tau_1, \ldots , \tau_n$ and $ \zeta_1, \ldots, \zeta_n $ be functions satisfying the hypotheses in Lemma \ref{lem: tangent of normal bundle} and let $ J  $ be the function defined on $ \mathcal{H}^n $ almost all of $ N^\phi(A) $ by
\begin{equation*}
	J(a,\eta) = \frac{|\tau_1(a,\eta) \wedge \ldots \wedge \tau_n(a,\eta)|}{|\zeta_1(a,\eta) \wedge \ldots \wedge \zeta_n(a,\eta) |}\in (0,\infty) \qquad \textrm{for $ \mathcal{H}^n $ a.e.\ $(a,\eta)\in N^\phi(A) $.}
\end{equation*}
Then the following statements hold.
\begin{enumerate}[{\rm (a)}]
	\item\label{theo: Steiner closed 1} For $ 0 < s < \infty $ and $ 0 < t < s $ let
	\begin{equation*}
S^s_t =  \{x\in S^\phi(A,t) : \rho^\phi_A(x) \geq s/t\} \qquad \textrm{and} \qquad  N_s = \{(a,\eta)\in N^\phi(A) : \bm{r}^\phi_A(a,\eta)\geq s  \}.
\end{equation*}
Then  $ \mathcal{H}^n(N_s \cap (B \times \partial\mathcal{W}^\phi)) < \infty $ for every compact set $ B \subset \mathbf{R}^{n+1} $ and $ \bm{\psi}^\phi_A|S^s_t $ is a bi-lipschitzian homeomorphism with $ \bm{\psi}^\phi_A [S^s_t]= N_s $ and $ (\bm{\psi}^\phi_A |S^s_t)^{-1}(a,\eta)= a + t\eta $ for each $(a,\eta)\in N_s $.
\item\label{theo: Steiner closed 2}  If $ \tau_1', \ldots, \tau'_n, \zeta'_1, \ldots , \zeta'_n $ is another set of functions satisfying the hypotheses of Lemma \ref{lem: tangent of normal bundle}, then
\begin{equation*}
	J(a,\eta) =  \frac{|\tau'_1(a,\eta) \wedge \ldots \wedge \tau'_n(a,\eta)|}{|\zeta'_1(a,\eta) \wedge \ldots \wedge \zeta'_n(a,\eta) |}
\end{equation*}
for $ \mathcal{H}^n $ a.e.\ $(a,\eta)\in N^\phi(A) $. Moreover $ J $ is $ \mathcal{H}^n\restrict N^\phi(A) $ measurable.
\item\label{theo: Steiner closed 3} If $ \rho > 0 $ and $ B \subset \mathbf{R}^{n+1} $ is compact, then
\begin{equation}\label{theo: Steiner closed 3.1}
	\int_{N^\phi(A) \cap (B \times \partial \mathcal{W}^\phi)}\inf\{\rho, \bm{r}^\phi_A\}^{j+1}\cdot J \cdot |\bm{H}^\phi_{A,j}|\, d\mathcal{H}^n < \infty
\end{equation}
for   $ j = 0, \ldots , n $ and
\begin{flalign}\label{theo: Steiner closed formula}
		& \int_{B^\phi(A, \rho) \setminus A} (\varphi \circ \bm{\psi}^\phi_A)\, d\mathcal{L}^{n+1} \notag \\
		&  \qquad = \sum_{j=0}^n \frac{1}{j+1} \int_{N^\phi(A)}\varphi(a,\eta)\,\phi(\bm{n}^\phi(\eta))\, J(a,\eta)\,  \inf\{ \rho, \bm{r}^\phi_A(a,\eta)  \}^{j+1}\bm{H}^\phi_{A,j}(a,\eta)\,  d\mathcal{H}^n(a,\eta)
	\end{flalign}
for every bounded Borel function $ \varphi: \mathbf{R}^{n+1} \times \mathbf{R}^{n+1} \rightarrow \mathbf{R} $ with compact support.
\end{enumerate}
\end{Theorem}
\begin{proof}
Fix $ s > 0 $ and a compact set $ B \subset \mathbf{R}^{n+1} $. We define $ B_t = \{x\in\R^{n+1} : \bm{\delta}^\phi_B(x)\leq t\} $ for  $t > 0 $. One can easily check that $ \bm{\psi}^\phi_A [S^s_t]= N_s $, $ \bm{\psi}^\phi_A|S^s_t $ is injective and $ (\bm{\psi}^\phi_A |S^s_t)^{-1}(a,\eta)= a + t\eta $ for  $(a,\eta)\in N_s $ and  $0 < t < s $. Consequently $ \bm{\psi}^\phi_A|S^s_t $ is a bi-lipschitzian homeomorphism by Lemma \ref{theo: projectionLipschitz} for  $ 0 < t < s $. Moreover, we notice that
\begin{equation*}
N_s\cap (B \times \partial \mathcal{W}^\phi) \subseteq \bm{\psi}_A^\phi(S^s_t \cap B_t)\subseteq \bm{\psi}_A^\phi(S^s_t \cap B_s) \qquad \textrm{for   $ 0 < t < s $}
\end{equation*}
and infer that
\begin{equation}\label{theo: Steiner closed 3.2}
		\mathcal{H}^n(N_s \cap (B \times \partial \mathcal{W}^\phi)) \leq \Lip(\bm{\psi}^\phi_A|S^s_{t})^n\, \mathcal{H}^n(S^\phi(A,t)\cap B_s)  \qquad \textrm{for  $ 0 < t < s $}.
\end{equation}
Using the coarea formula, we get
\begin{equation*}
	\int_{0}^\infty \mathcal{H}^n(S^\phi(A,r)\cap B_s)\, dr =\int_{B_s} J_1 \bm{\delta}^\phi_A(x)\, d\mathcal{L}^{n+1}(x) < \infty,
\end{equation*}
whence we infer that $\mathcal{H}^n(S^\phi(A,r)\cap B_s) < \infty $ for $\mathcal{L}^1 $ a.e.\ $ r > 0 $. Consequently, there exists some $t\in ( 0, s) $ so that $\mathcal{H}^n(S^\phi(A,t)\cap B_s) < \infty $, and hence \eqref{theo: Steiner closed 3.2} implies that
\begin{equation*}
	\mathcal{H}^n(N_s \cap (B \times \partial \mathcal{W}^\phi)) < \infty.
\end{equation*}
This proves \ref{theo: Steiner closed 1}.

For each $ 0 < s \le s' $ we define $ f_s : N_{s'} \rightarrow \mathbf{R}^{n+1} $ by $ f_s(a,\eta) = a + s\eta  $ for $(a,\eta)\in N_{s'} $. We apply Lemma \ref{lem: tangent of normal bundle} to compute
	\begin{equation}\label{theo: Steiner closed: formula 0}
		\ap J_n^{N_{s'}}f_s(a,\eta) \cdot\bm{1}_{\widetilde{N}^\phi_d(A)}(a,\eta) = J(a,\eta) s^{n-d} \Bigg(\prod_{j=1}^{d}(1 + s \kappa^\phi_{A,j}(a,\eta))\Bigg)\, \bm{1}_{\widetilde{N}^\phi_d(A)}(a,\eta)
	\end{equation}
	for $ \mathcal{H}^n $ a.e.\ $ (a,\eta)\in N_{s} $. Since $ 1 + s \kappa^\phi_{A,i}(a,\eta) > 0 $ for $ i = 1, \ldots , n $ and for $ \mathcal{H}^n $ a.e.\ $(a,\eta)\in N_{2s} $ by Remark \ref{eq: curvatures and reach}, we conclude that
	\begin{equation*}
		J(a,\eta) = \ap J_n^{N_{2s}}f_s(a,\eta)\sum_{d=0}^ns^{d-n} \Bigg(\prod_{j=1}^{d}(1 + s \kappa^\phi_{A,j}(a,\eta))\Bigg)^{-1}\, \bm{1}_{\widetilde{N}^\phi_d(A)}(a,\eta)
	\end{equation*}
	for $ \mathcal{H}^n $ a.e.\ $(a,\eta)\in N_{2s} $ and for every $ s > 0 $. Noting that the right-hand side of the last equation does not depend on the choice of $ \tau_1, \ldots , \tau_n, \zeta_1, \ldots , \zeta_n $ and  defines an $ \mathcal{H}^n \restrict N^\phi(A) $-measurable function, we obtain \ref{theo: Steiner closed 2}.

We fix $ \rho > 0 $. Then we define
\begin{equation*}
\delta(a,\eta) = \inf\{ \rho, \bm{r}^\phi_A(a,\eta)  \} \qquad \textrm{for $(a,\eta)\in N^\phi(A) $,}
\end{equation*}
\begin{equation*}
\Omega =\{(a,\eta,t): (a,\eta) \in N^\phi(A), \; 0 < t < \bm{r}^\phi_A(a,\eta)   \}
\end{equation*}
and the bijective (locally) Lipschitz map
	\begin{equation*}
		f : \Omega \rightarrow \mathbf{R}^{n+1} \setminus(A \cup \Cut^\phi(A)), \qquad f(a,\eta,t) = a + t \eta \quad \textrm{for $(a,\eta,t)\in \Omega $.}
	\end{equation*}
We choose an arbitrary sequence $ s_i \to 0 + $ and define
\begin{equation*}
	\Omega_i =\{(a,\eta,t): (a,\eta) \in N_{s_i}, \; 0 < t < \bm{r}^\phi_A(a,\eta)   \} \qquad \textrm{for $ i \geq 1 $}.
\end{equation*}
Notice that $ \Omega = \bigcup_{i=1}^\infty \Omega_i $,
	\begin{equation*}
		\Tan^{n+1}(\mathcal{H}^{n+1}\restrict \Omega_i, (a,\eta,t)) = \Tan^n(\mathcal{H}^n\restrict N_{s_i}, (a,\eta))\times \mathbf{R} \qquad \textrm{for $\mathcal{H}^{n+1}$ a.e. $(a,\eta,t) \in \Omega_i $}
	\end{equation*}
and $ \Tan^n(\mathcal{H}^n\restrict N_{s_i}, (a,\eta)) $ is an $ n $-dimensional linear subspace for $ \mathcal{H}^{n} $ a.e.\ $(a,\eta)\in N_{s_i} $ by \cite[3.2.19]{MR0257325}.
We  apply again Lemma \ref{lem: tangent of normal bundle},  for $ \mathcal{H}^{n+1} $ a.e.\ $(a,\eta,t)\in \Omega_i $ and for every $ i \geq 1 $, to compute
	\begin{flalign}\label{theo: Steiner closed: jacobian}
		&  \bm{1}_{\widetilde{N}^\phi_d(A)}(a,\eta)\cdot\ap J_{n+1}^{\Omega_i}f(a,\eta,t) \notag \\
		& \quad =  \bm{1}_{\widetilde{N}^\phi_d(A)}(a,\eta)\cdot\frac{|\tau_1(a,\eta) \wedge \ldots \wedge \tau_n(a,\eta)\wedge \eta|}{|\zeta_1(a,\eta) \wedge \ldots \wedge \zeta_n(a,\eta)|} t^{n-d}\left(\prod_{j=1}^{d}(1 + t \kappa^\phi_{A,j}(a,\eta))\right)\notag  \\
		& \quad =   \bm{1}_{\widetilde{N}^\phi_d(A)}(a,\eta)\cdot  (\bm{n}^\phi(\eta) \bullet \eta)\frac{|\tau_1(a,\eta) \wedge \ldots \wedge \tau_n(a,\eta) \wedge \bm{n}^\phi(\eta) |}{|\zeta_1(a,\eta) \wedge \ldots \wedge \zeta_n(a,\eta)|} t^{n-d} \left(\prod_{j=1}^{d}(1 + t \kappa^\phi_{A,j}(a,\eta))\right) \notag \\
		& \quad =  \bm{1}_{\widetilde{N}^\phi_d(A)}(a,\eta)\cdot \phi(\bm{n}^\phi(\eta))J(a,\eta) t^{n-d} \prod_{j=1}^{d}(1 + t \kappa^\phi_{A,j}(a,\eta)).
		\end{flalign}
Note that
	\begin{equation*}
	B^\phi(A, \rho) \setminus \big(A \cup \Cut^\phi(A)\big) = f\big(\{(a,\eta,t)\in \Omega: 0 < t \leq \rho    \}\big)
	\end{equation*}
and recall that $ \mathcal{L}^{n+1}(\Cut^\phi(A)) =0 $.  Let $ h : \mathbf{R}^{n+1} \rightarrow [0,\infty)$ be a Borel function. Then we employ the monotone convergence theorem, the coarea formula, Fubini's theorem and \eqref{theo: Steiner closed: jacobian} to get
	\begin{flalign}\label{theo: Steiner closed: formula 1}
		& \int_{B^\phi(A, \rho) \setminus A} h(x)\, d\mathcal{L}^{n+1}(x) \notag \\
		& \qquad = \lim_{i \to \infty} \sum_{d=0}^n\int_{N_{s_i} \cap \widetilde{N}^\phi_d(A)} \int_{0}^{\delta(a, \eta)} h(a + t \eta)\, \ap J^{\Omega_i}_{n+1} f(a, \eta, t) \,dt\, d\mathcal{H}^{n}(a, \eta) \notag \\
		& \qquad = \sum_{d=0}^n\int_{\widetilde{N}^\phi_d(A)}  \phi(\bm{n}^\phi(\eta))\,J(a,\eta)\bigg(\int_{0}^{\delta(a, \eta)}h(a + t \eta)\,t^{n-d}\, \prod_{j=1}^{d}(1 + t \kappa^\phi_{A,j}(a,\eta)) \,dt\bigg) \, d\mathcal{H}^{n}(a, \eta).
	\end{flalign}
	We notice the equality
	\begin{equation}\label{theo: Steiner closed: formula 1'}
t^{n-d} \prod_{j=1}^{d}(1 + t \kappa^\phi_{A,j}(a,\eta)) =	   \sum_{j=0}^d t^{n-d+j}  E^\phi_{A,j}(a,\eta)
	\end{equation}
for $(a,\eta) \in \widetilde{N}^\phi_d(A) $ and $ t > 0 $. If $ h(x) = (\varphi \circ \bm{\psi}^\phi_A)(x) $ for every $ x \in \Unp^\phi(A) $ for some Borel function $ \varphi : \mathbf{R}^{n+1} \times \mathbf{R}^{n+1} \rightarrow [0, +\infty)$, then $ h(a + t\eta) = \varphi(a, \eta) $ for every $ (a, \eta, t) \in \Omega $ and we obtain from \eqref{theo: Steiner closed: formula 1} and \eqref{theo: Steiner closed: formula 1'} that
	\begin{flalign}\label{theo: Steiner closed formula 2}
&\int_{B^\phi(A, \rho) \setminus A} (\varphi \circ \bm{\psi}^\phi_A)\, d\mathcal{L}^{n+1}\notag \\
&\qquad =\sum_{d=0}^n\int_{\widetilde{N}_d^\phi(A)}\varphi(a,\eta)  \phi(\bm{n}^\phi(\eta))\,J(a,\eta)\, \sum_{j=0}^d \frac{\delta(a,\eta)^{n-d+j+1}}{n-d+j+1}\, \,E^\phi_{A,j}(a,\eta)\;  d\mathcal{H}^n(a,\eta).
	\end{flalign}
 If $ B \subset \mathbf{R}^{n+1} $ is a compact set, we can choose $ \varphi = \bm{1}_{B \times \partial \mathcal{W}^\phi} $ in \eqref{theo: Steiner closed formula 2} to infer that
\begin{equation}\label{theo: Steiner closed: integrability estimate II}
\int_{\widetilde{N}_d^\phi(A) \cap (B \times \partial \mathcal{W}^\phi)} \phi(\bm{n}^\phi(\eta))\,J(a,\eta)\, \sum_{j=0}^d \frac{\delta(a,\eta)^{n-d+j+1}}{n-d+j+1}\, \,E^\phi_{A,j}(a,\eta)\;  d\mathcal{H}^n(a,\eta) < \infty
\end{equation}
for  $ d =0, \ldots , n $.

We now proceed as in \cite{MR2031455} to prove that for  $ d = 0, \ldots , n $ and $ l = 0, \ldots , d $ there exists a constant $ \bar{c}(l,d,n) > 0 $ so that
\begin{equation}\label{theo: Steiner closed: integrability estimate I}
\sum_{j=0}^d \frac{\delta(a,\eta)^{n-d+j+1}}{n-d+j+1}\, \,E^\phi_{A,j}(a,\eta) \geq \bar{c}(l,d,n)\, \delta(a,u)^{n-d+l+1} \,|E^\phi_{A,l}(a,\eta)|\ge 0
\end{equation}
for  $ (a,\eta) \in \widetilde{N}^\phi_d(A) $. For $ i =1, \ldots , d $ and $ (a,\eta)\in \widetilde{N}^\phi_d(A) $ we define $ \kappa^+_i(a,\eta) = \sup\{\kappa^\phi_{A,i}(a,\eta), 0\} $ and $ \kappa_i^-(a,\eta) = \inf \{\kappa^\phi_{A,i}(a,\eta), 0\} $. We set
\begin{equation*}
E_j^+(a,\eta) = \sum_{\lambda \in \Lambda(d,j)}\;\prod_{h=1}^j\kappa^+_{\lambda(h)}(a,\eta), \qquad E_j^-(a,\eta) = \sum_{\lambda \in \Lambda(d,j)}\;\prod_{h=1}^j\kappa^-_{\lambda(h)}(a,\eta)
\end{equation*}
for $ (a,\eta)\in \widetilde{N}^\phi_d(A) $ and $ j =0, \ldots , d $ (notice $E_0^+ \equiv 1 $ and $E^0_- \equiv 1 $). We set $ k = n-d+1 $. For $(a,\eta)\in \widetilde{N}^\phi_d(A) $ we observe that
\begin{equation*}
E^\phi_{A,j}(a,\eta) = \sum_{l=0}^{j}E_l^-(a,\eta)E^+_{j-l}(a,\eta) \qquad \textrm{for $ j =0, \ldots , d $}
\end{equation*}
and, noting that $ E_s^+(a,\eta) E^-_l(a,\eta) =0 $ for $ s > d-l $ and $ -1 \leq \delta(a,u)\kappa^-_i(a,\eta) \leq 0 $ for $ i = 1, \ldots , d $ by Remark \ref{eq: curvatures and reach}, we employ \cite[Lemma 2.4]{MR2031455} to conclude
\begin{flalign*}
\sum_{j=0}^d \frac{\delta(a,\eta)^{k+j}}{k+j}\, \,E^\phi_{A,j}(a,\eta) & = \sum_{j=0}^d \sum_{l=0}^j \frac{\delta(a,\eta)^{k+j}}{k+j}\, E_l^-(a,\eta)E^+_{j-l}(a,\eta)\\
& = \sum_{l=0}^d \sum_{s=0}^{d-l} \frac{\delta(a,\eta)^{k+l+s}}{k+l+s}\, E_l^-(a,\eta)E^+_{s}(a,\eta)\\
& = \sum_{s=0}^d E^+_{s}(a,\eta)\delta(a,\eta)^{k+s} \left( \sum_{l=0}^{d} \frac{\delta(a,\eta)^{l}}{k+l+s}\, E_l^-(a,\eta)\right) \\
& \geq \sum_{s=0}^d E^+_{s}(a,\eta)\delta(a,\eta)^{k+s} c(d,n,s),
\end{flalign*}
where $ c(d,n,s) $ is a positive constant depending only on $ s $, $ d $ and $ n $. Since
\begin{equation*}
\sup_{(a,\eta) \in \widetilde{N}^\phi_d(A)} \delta(a,\eta)^l|E^-_l(a,\eta)| \leq
\binom{d }{ l} \qquad \textrm{for $ l =0, \ldots , d $,}
\end{equation*}
we conclude that
\begin{flalign*}
\sum_{j=0}^d \frac{\delta(a,\eta)^{k+j}}{k+j}\, \,E^\phi_{A,j}(a,\eta) & \geq \sum_{s=0}^d E^+_{s}(a,\eta)\delta(a,\eta)^{k+s} c(d,n,s)\binom{d}{l-s}^{-1}|E^-_{l-s}(a,\eta)|\delta(a,\eta)^{l-s}\\
&  \geq \delta(a,\eta)^{k+l} \bar{c}(l,d,n) \sum_{s=0}^d E^+_{s}(a,\eta)|E^-_{l-s}(a,\eta)| \\
& \ge \delta(a,\eta)^{k+l} \bar{c}(l,d,n) |E^\phi_{A,l}(a,\eta)|
\end{flalign*}
for $(a,\eta)\in \widetilde{N}^\phi_d(A) $ and $ l =0, \ldots , d $, where $ \bar{c}(l,d,n) = \min \big\{ c(d,n,s)\binom{d }{ l-s}^{-1} : s =0, \ldots , d    \big\} > 0 $.

Let $ B \subseteq \mathbf{R}^{n+1} $ be a compact set. It follows from  \eqref{theo: Steiner closed: integrability estimate II} and \eqref{theo: Steiner closed: integrability estimate I} that
\begin{equation}\label{integrbound1}
	\int_{\widetilde{N}_d^\phi(A) \cap (B \times \partial \mathcal{W}^\phi)} \phi(\bm{n}^\phi(\eta))\,J(a,\eta)\,  \delta(a,\eta)^{n-d+l+1} \,|E^\phi_{A,l}(a,\eta)|\;  d\mathcal{H}^n(a,\eta) < \infty
\end{equation}
for  $ l =0, \ldots , d $ and $ d =0, \ldots , n $.
Since
\begin{flalign*}
\sum_{i=0}^n \frac{\delta(a,\eta)^{i+1}}{i+1}|\bm{H}^\phi_{A,i}(a,\eta)|\le \sum_{d=0}^n\sum_{j=0}^d \frac{\delta(a,\eta)^{n-d+j+1}}{n-d+j+1}\, \,|E^\phi_{A,j}(a,\eta)|\, \bm{1}_{\widetilde{N}^\phi_d(A)}(a,\eta) ,
\end{flalign*}
where equality holds if the absolute values are omitted on both sides of this equation,
we obtain  \eqref{theo: Steiner closed 3.1} from \eqref{integrbound1}. In addition, we
conclude from \eqref{theo: Steiner closed formula 2} that
\begin{flalign}
		& \int_{\{x\in\R^{n+1} : 0 < \bm{\delta}^\phi_A(x) \leq \rho\}} (\varphi \circ \bm{\psi}^\phi_A)\, d\mathcal{L}^{n+1} \notag \\
	& \qquad = \int_{N^\phi(A)} \varphi(a,\eta)\,\phi(\bm{n}^\phi(\eta))\,J(a,\eta)\,  \sum_{r=0}^n \frac{\delta(a,\eta)^{r+1}}{r+1}\bm{H}^\phi_{A,r}(a,\eta)\;  d\mathcal{H}^n(a,\eta)\notag \\
	&  \qquad =\sum_{r=0}^n \frac{1}{r+1}\int_{N^\phi(A)} \varphi(a,\eta)\,\phi(\bm{n}^\phi(\eta))\,J(a,\eta)\,   \delta(a,\eta)^{r+1}\, \bm{H}^\phi_{A,r}(a,\eta)\;  d\mathcal{H}^n(a,\eta)\label{startag}
\end{flalign}
for every bounded Borel function $ \varphi: \mathbf{R}^{n+1} \times \mathbf{R}^{n+1} \rightarrow \mathbf{R} $ with compact support, where we use the integrability property in \eqref{theo: Steiner closed 3.1} to obtain the equality in \eqref{startag}.
\end{proof}

It is convenient to introduce the following function.

\begin{Definition}\label{def: jacobian}
For every non-empty closed set $ A \subseteq \mathbf{R}^{n+1} $, we denote by \index{O6@$J^\phi_A $}  $ J^\phi_A  $ the $ \mathcal{H}^n \restrict N^\phi(A) $ measurable function $ J $ introduced in Theorem \ref{theo: Steiner closed}.
\end{Definition}

The  arguments in the proof of Theorem \ref{theo: Steiner closed} readily provide the following change-of-variable-type formula. This formula plays a central role in the proof of Theorem \ref{theo: heintze karcher}.

\begin{Corollary}[Disintegration of Lebesgue measure]\label{theo: change of variable}
Let $ A $ be a closed set, and let $ h : \mathbf{R}^{n+1} \rightarrow \mathbf{R}$ be a non-negative Borel function. Then
\begin{flalign*}
		&\int_{\mathbf{R}^{n+1} \setminus A} h(x)\, d\mathcal{L}^{n+1}(x)\\
		& \qquad = \sum_{d=0}^{n}\int_{\widetilde{N}^\phi_d(A)}\phi(\bm{n}^\phi(\eta))\, J^\phi_A(a,\eta) \int_{0}^{\bm{r}^\phi_A(a,\eta)}
		h(a + t\eta)\,t^{n-d} \,\prod_{j=1}^{d}(1 + t \kappa^\phi_{A,j}(a,\eta))\, dt\, d\mathcal{H}^n(a,\eta).
		\end{flalign*}
\end{Corollary}

\begin{proof}
During the proof of Theorem \ref{theo: Steiner closed} we have proved (see eq.\ \eqref{theo: Steiner closed: formula 1}) that
	\begin{flalign*}
		&\int_{B^\phi(A, \rho)\setminus A} h(x)\, d\mathcal{L}^{n+1}(x)\\
		& = \sum_{d=0}^{n}\int_{\widetilde{N}^\phi_d(A)} \phi(\bm{n}^\phi(\eta))\, J^\phi_A(a,\eta) \int_{0}^{\inf\{\rho, \bm{r}^\phi_A(a,\eta)\}}h(a + t\eta)\,t^{n-d}\,\prod_{j=1}^{d}(1 + t \kappa^\phi_{A,j}(a,\eta))\, dt\, d\mathcal{H}^n(a,\eta)
	\end{flalign*}
	for every $ \rho > 0 $. The conclusion now follows letting $ \rho \to +\infty $.
\end{proof}

\subsection{A Heintze--Karcher inequality for closed sets}\label{sec:HKI}

The following theorem provides a very general version of a Heintze--Karcher inequality under minimal assumptions on a closed set and its complement. Several consequences will be derived for sets with positive reach in Section \ref{Section: positive reach}.

We start with a lemma.

\begin{Lemma}\label{theo: positive reach}
	Suppose $ A \subseteq \mathbf{R}^{n+1} $ is a closed set, $ s_0 > 0 $ and $ \bm{r}^\phi_A(a,u) \geq s_0 $ for $ \mathcal{H}^n $ a.e.\ $(a,u)\in N^\phi(A) $. Then $ \{ x \in \mathbf{R}^{n+1} : \bm{\delta}^\phi_A(x) < s_0   \} \subseteq \Unp^\phi(A) $.
\end{Lemma}
\begin{proof}
Let $ 0 < s < s_0 $. Define $\Omega_s^\ast = \{ (a,\eta,t) : (a,\eta)\in N^\phi(A), \; 0 < t < \inf\{s, \bm{r}^\phi_A(a,\eta)\}   \} $ and the bijective map
\begin{equation*}
f : \Omega^\ast_s \rightarrow \{ x : 0 < \bm{\delta}^\phi_A(x) < s    \} \setminus \Cut^\phi(A), \qquad f(a,\eta,t) = a + t\eta.
\end{equation*}
Moreover we define $ \Omega_s = \Omega^\ast_s \cap \{  (a,\eta,t) : \bm{r}^\phi_A(a,\eta)\geq s_0 \} $ and we notice that the hypothesis implies
\begin{equation*}
\mathcal{H}^{n+1}(\Omega^\ast_s \setminus \Omega_s) =0.
\end{equation*}
Consequently $ \mathcal{L}^{n+1}(f(\Omega^\ast_s) \setminus f(\Omega_s)) =0 $ and, recalling that $ \mathcal{L}^{n+1}(\Cut^\phi(A)) =0 $, we conclude that $ f(\Omega_s) $ is dense in $\{x : 0 < \bm{\delta}^\phi_A(x) < s\} $. We choose now $ x \in \mathbf{R}^{n+1} $ so that $ 0 < \bm{\delta}^\phi_A(x) < s $ and a sequence $(a_i, \eta_i, t_i)\in\Omega_s $ so that $ a_i+ t_i\eta_i \to x $. Up to subsequences we can assume that there exist $ a \in A $, $ \eta \in \partial \mathcal{W}^\phi $ and $ 0 \leq t \leq s $ so that $ a_i \to a $, $ \eta_i \to \eta $ and $ t_i \to t $. Therefore $ x = a + t\eta $ and
\begin{equation*}
0 < \bm{\delta}^\phi_A(x) = \lim_{i \to \infty}\bm{\delta}^\phi_A(a_i + t_i \eta_i) = \lim_{i \to \infty} t_i = t.
\end{equation*}
It follows that $(a,\eta)\in N^\phi(A) $ and the upper semicontinuity of $ \bm{r}^\phi_A $ implies that $ \bm{r}^\phi_A(a,\eta)\geq s_0 > s \geq t $ and $ x \in \Unp^\phi(A)  $. In conclusion, we have proved that $\{ x : \bm{\delta}^\phi_A(x) < s    \} \subseteq \Unp^\phi(A)$ for every $ 0 < s < s_0 $.
\end{proof}

\begin{Theorem}\label{theo: heintze karcher}
	Let $ C \subset \mathbf{R}^{n+1} $ be a closed set with $ 0 < \mathcal{L}^{n+1}(\Int(C)) < \infty $. Let $ K = \mathbf{R}^{n+1} \setminus \Int(C)$ and assume that
	\begin{equation}\label{eq:negmeancurv}
		\sum_{i=1}^n \kappa^\phi_{K,i}(a,\eta) \leq 0 \qquad \textrm{for $ \mathcal{H}^n $ a.e.\ $(a,\eta)\in N^\phi(K) $.}
	\end{equation}
	 Then
	\begin{equation}\label{theo: heintze karcher: inequality}
		(n+1)\mathcal{L}^{n+1}(\Int(C)) \leq n	\int_{\widetilde{N}^\phi_n(K)} J^\phi_K(a,\eta)\frac{\phi(\bm{n}^\phi(\eta))}{|\bm{H}^\phi_{K,1}(a,\eta)|}\, d\mathcal{H}^n(a,\eta).
	\end{equation}
If  equality holds  in \eqref{theo: heintze karcher: inequality} and there exists $ q < \infty $ so that $|\bm{H}^\phi_{K,1}(a,\eta)| \leq q $ for $ \mathcal{H}^n $ a.e.\ $(a,\eta)\in \widetilde{N}^\phi_n(K) $, then there are   $ N\in\mathbb{N} $, $ c_1, \ldots , c_N \in \mathbf{R}^{n+1} $ and $ \rho_1, \ldots , \rho_N \geq \frac{n}{q} $ such that
	\begin{equation*}
	 \Int(C) = \bigcup_{i=1}^N \Int (c_i + \rho_i \mathcal{W}^\phi), \qquad \Int\big(c_i + \rho_i \mathcal{W}^\phi\big) \cap \Int\big(c_j + \rho_j \mathcal{W}^\phi\big) = \varnothing \quad \textrm{for $ i \neq j $.}
	\end{equation*}

\end{Theorem}

\begin{proof}
Note that \eqref{eq:negmeancurv} implies that $\mathcal{H}^n(N^\phi(K)\setminus\widetilde{N}^\phi_n(K))=0 $.

	For the proof we may assume that $ |\bm{H}^\phi_{K,1}(a,\eta)| > 0 $ for $ \mathcal{H}^n $ a.e.\ $(a,\eta)\in \widetilde{N}^\phi_n(K) $, since otherwise the inequality \eqref{theo: heintze karcher: inequality} is obviously true (with strict inequality).  By Remark \ref{eq: curvatures and reach} we infer that
	\begin{equation*}
		\bm{r}^\phi_K(a,\eta) \leq -\frac{n}{\bm{H}^\phi_{K,1}(a,\eta)} \qquad \textrm{for $ \mathcal{H}^n $ a.e.\ $(a,\eta)\in \widetilde{N}^\phi_n(K) $}
	\end{equation*}
	and $ 1 + t \kappa^\phi_{K,i}(a,\eta) >0 $ for $ \mathcal{H}^n $ a.e.\ $ (a,\eta)\in  \widetilde{N}^\phi_n(K) $,  $ 0 < t < \bm{r}^\phi_K(a,\eta) $ and $i=1,\ldots,n$. Employing the change of variable formula in Corollary  \ref{theo: change of variable} (with $ h \equiv 1 $) and the classical arithmetic-geometric mean inequality, we can estimate
	\begin{flalign}
		\mathcal{L}^{n+1}(\Int(C)) & = \int_{\widetilde{N}^\phi_n(K)}\phi(\bm{n}^\phi(\eta))J^\phi_K(a,\eta)\,  \int_{0}^{\bm{r}^\phi_K(a,\eta)} \prod_{j=1}^{n}(1 + t \kappa^\phi_{K,j}(a,\eta))\, dt\, d\mathcal{H}^n(a,\eta)\notag\\
		& \leq \int_{\widetilde{N}^\phi_n(K)}\phi(\bm{n}^\phi(\eta))J^\phi_K(a,\eta)\,  \int_{0}^{\bm{r}^\phi_K(a,\eta)} \Big( 1 + \frac{t}{n}\bm{H}^\phi_{K,1}(a,\eta)\Big)^n\, dt\, d\mathcal{H}^n(a,\eta) \notag\\
		& \leq \int_{\widetilde{N}^\phi_n(K)}\phi(\bm{n}^\phi(\eta))J^\phi_K(a,\eta) \int_{0}^{-\frac{n}{\bm{H}^\phi_{K,1}(a,\eta)}} \Big( 1 + \frac{t}{n}\bm{H}^\phi_{K,1}(a,\eta)\Big)^n\, dt\, d\mathcal{H}^n(a,\eta)\notag\\
		& = \frac{n}{n+1}\int_{\widetilde{N}^\phi_n(K)}J^\phi_K(a,\eta)\, \frac{\phi(\bm{n}^\phi(\eta)) }{|\bm{H}^\phi_{K,1}(a,\eta)|}\, d\mathcal{H}^n(a,\eta).\label{eq:turn}
	\end{flalign}
	
	We discuss now the equality case.  We assume that $|\bm{H}^\phi_{K,1}(a,\eta)| \leq q $ for $ \mathcal{H}^n $ a.e.\ $(a,\eta)\in \widetilde{N}^\phi_n(K) $. If \eqref{theo: heintze karcher: inequality} holds with equality, then the inequalities in the derivation of \eqref{eq:turn} become equalities. In particular, we deduce that
	\begin{equation}\label{eq:4a}
		\bm{r}^\phi_K(a,\eta) = -\frac{n}{\bm{H}^\phi_{K,1}(a,\eta)} \geq \frac{n}{q} \qquad \textrm{for $ \mathcal{H}^n $ a.e.\ $(a,\eta)\in \widetilde{N}^\phi_n(K) $}
	\end{equation}
	and the  condition
	\begin{equation}\label{eq:4b}
		\kappa^\phi_{K,1}(a,\eta) = \ldots = \kappa^\phi_{K,n}(a,\eta) \qquad \textrm{for $ \mathcal{H}^n $ a.e.\ $(a,\eta)\in \widetilde{N}^\phi_n(K) $}.
	\end{equation}
	Consequently, we infer from Lemma \ref{theo: positive reach} that $ \{ x \in \mathbf{R}^{n+1} : \bm{\delta}^\phi_K(x) < \frac{n}{q}  \} \subseteq \Unp^\phi(K) $, which means that $ \reach^\phi(K) \geq \frac{n}{q} $ (see Definition \ref{def: reach}). We define $ E_r = \{x \in \mathbf{R}^{n+1}: \bm{\delta}^\phi_K(x)\geq r\} $ for  $ r > 0 $ and   notice that $ \partial E_r = S^\phi(K,r) $ for  $ r > 0 $. We fix now $ 0 < r < \frac{n}{q} $. It follows that $ \partial E_r  $ is a closed $ \mathcal{C}^{1,1} $-hypersurface  by \cite[Corollary 5.8]{MR4160798} and $\rchi^\phi_{K,1},  \ldots,   \rchi^\phi_{K,n} $ are the anisotropic principal curvatures of $ \partial E_r $ with respect to the anisotropic normal $ \bm{\nu}^\phi_K|\partial E_r $ (which points towards $C$). It follows from \eqref{eq:4a} and \eqref{eq:4b} that
	\begin{equation*}
		\rchi^\phi_{K,1}(x) = \ldots = \rchi^\phi_{K,n}(x)= \frac{1}{r - \bm{r}^\phi_K(\bm{\xi}^\phi_K(x), \bm{\nu}^\phi_K(x))} \geq \frac{1}{r - \frac{n}{q}} = \Big( \frac{qr-n}{q}\Big)^{-1}
	\end{equation*}
	for $ \mathcal{H}^n $ a.e.\ $x \in \partial E_r $; in particular, $\rchi^\phi_{K,i}(x) < 0 $ for  $ i = 1, \ldots , n $ and $ \mathcal{H}^n $ a.e.\ $x \in \partial E_r $. Hence, an application of \cite[Lemma 3.2]{MR4160798}\footnote{Notice that the last line of \cite[Lemma 3.2]{MR4160798}  contains a typo: one should replace the equality $ M = \partial \mathbf{B}^F(a, |\lambda|^{-1}) $ with $ M = \partial \mathbf{B}^{F^\ast}(a, |\lambda|^{-1}) $, which is what the proof shows.},  to each of the at most countably many connected components of $\partial E_r$ shows that there exist at most countably many points  $ c_1, c_2, \ldots  \in \mathbf{R}^{n+1} $ and numbers $ \lambda_1, \lambda_2 \ldots  \geq \frac{n-qr}{q} $ so that
	\begin{equation*}
		 E_r = \bigcup_{i=1}^N (c_i + \lambda_i  \mathcal{W}^\phi) \qquad \textrm{and}\qquad 	(c_i + \lambda_i \mathcal{W}^\phi) \cap (c_j + \lambda_j  \mathcal{W}^\phi) = \varnothing \quad \textrm{for  $ i \neq j $,}
	\end{equation*}
where $ N \in \mathbf{N} \cup \{\infty\} $. Since $ \mathcal{L}^{n+1}(\Int(C)) < \infty $, it follows that $ N < \infty $.  If $ i \in \{ 1,\ldots,N\} $ and $ z \in \partial \mathcal{W}^\phi $, we have
\begin{equation*}
	\frac{\nabla \bm{\delta}^\phi_K(c_i + \lambda_i z)}{|\nabla \bm{\delta}^\phi_K(c_i + \lambda_i z)|}= - \bm{n}^\phi(z)
\end{equation*}
and, noting \eqref{eq: normal wulff shape and gradient phi} and \eqref{eq: gradient and normal}, we conclude
\begin{equation*}
	z = \nabla \phi(\bm{n}^\phi(z)) = - \nabla \phi(\nabla \bm{\delta}^\phi_K(c_i + \lambda_i z)) = - \bm{\nu}^\phi_K(c_i + \lambda_i z).
\end{equation*}
For $ 0 \leq s < r $ we define the bilipschitz homeomorphism $ f_s : \partial E_r \rightarrow \partial E_{r-s} $ by $ f_s(x) = x - s \bm{\nu}^\phi_K(x) $ for $ x \in \partial E_r $. Then we get
\begin{equation*}
f_s(c_i + \lambda_i \partial \mathcal{W}^\phi) = c_i + (\lambda_i + s)\partial \mathcal{W}^\phi \quad \textrm{for $ i \geq 1 $,} \qquad \partial E_{r-s} = \bigcup_{i=1}^N \big(c_i + (\lambda_i +s) \partial \mathcal{W}^\phi\big)
\end{equation*}
and
	\begin{equation*}
\big(c_i + (\lambda_i + s) \partial \mathcal{W}^\phi) \cap \big(c_j + (\lambda_j +s)  \partial \mathcal{W}^\phi\big) = \varnothing \quad \textrm{for  $ i \neq j $}.
\end{equation*}
Consequently, for $ 0 \leq s < r $,
\begin{equation*}
	E_{r-s} = \bigcup_{i=1}^N \big(c_i + (\lambda_i +s) \mathcal{W}^\phi\big), \qquad \big(c_i + (\lambda_i + s) \mathcal{W}^\phi) \cap \big(c_j + (\lambda_j +s)  \mathcal{W}^\phi\big) = \varnothing \quad \textrm{for every $ i \neq j $}.
\end{equation*}
and the proof is complete.
\end{proof}

\subsection{A general disintegration formula}\label{subsec: disintegration}

For the next definition we need to recall that if $ A $ is a closed set in $ \mathbf{R}^{n+1}$ and $ a \in A$, then the set
\begin{equation*}
  \Dis(A,a) : =  \{u \in \mathbf{R}^{n+1} : \bm{\delta}_A(a+u) = |u| \}
\end{equation*}
is a closed convex set (see \cite[Theorem 4.8 (2)]{MR0110078} or \cite{MR4012808}). If $ X $ is a convex set, then $ \dim X$ denotes the dimension of the affine hull of $ X $.

\begin{Definition}\label{def: stratification}
    Let $\varnothing\neq A \subseteq \mathbf{R}^{n+1}$ be a closed set and $ i \in \{0, \ldots , n+1\}$. We define the $ i$-th stratum of $ A $ as \index{B5@$A^{(i)}$}
    \begin{equation*}
        A^{(i)} =  \{ a\in A  : \dim \Dis(A,a) = n+1-i \}.
    \end{equation*}
\end{Definition}
\begin{Remark}
Evidently, we have $ A = \bigcup_{i=0}^{n+1}A^{(i)}$ and $ A^{(n+1)} = A \setminus \bm{p}(N(A))$. Noting that $  N(A,a) = \{  u/|u|\in \bS^n : u \in \Dis(A,a) \setminus \{0\} \} $ for every $ a \in \bm{p}(N(A))$, we can easily deduce from \eqref{eq: phi normal vs euclidean normal} that
\begin{equation*}
   A^{(i)} = \{ a \in A : 0 < \mathcal{H}^{n-i}(N^\phi(A,a)) < \infty \} \qquad \textrm{for $ i =0, \ldots , n$.}
\end{equation*}
We recall that $ A^{(i)}$ is a countably $i$-rectifiable Borel set which can be covered outside a set of $ \mathcal{H}^i$-measure zero by a countable union of $ \mathcal{C}^2$-submanifolds of dimension $ i $, for  $ i =0, \ldots, n $; see \cite{MR4012808}.
\end{Remark}

\begin{Lemma}\label{lem: stratification and curvature}
If $\varnothing\neq A \subseteq \mathbf{R}^{n+1}$ is a closed set, $ m \in\{0, \ldots , n\}$ and $ S \subseteq \mathbf{R}^{n+1}$ is a countable union of Borel subsets with finite $ \mathcal{H}^m$ measure, then
\begin{equation}\label{lem: stratification and curvature: 1}
    \mathcal{H}^n \big((N^\phi(A)|A^{(j)}) \setminus \textstyle\bigcup_{l=0}^j \widetilde{N}^\phi_l(A)\big) =0 \qquad \textrm{for $ j =0, \ldots , n $}
\end{equation}
and
\begin{equation}\label{lem: stratification and curvature: 2}
    \mathcal{H}^n \big((N^\phi(A)|A^{(j)} \cap S) \setminus \textstyle\bigcup_{l=0}^{m-1} \widetilde{N}^\phi_l(A)\big) =0 \qquad \textrm{for $ j >m $.}
\end{equation}
\end{Lemma}

\begin{proof}
Recall that $ N_s =  \{(a,\eta)\in N^\phi(A)  : \bm{r}^\phi_A(a,\eta) \geq s \}$, for every $s>0$,  has finite $ \mathcal{H}^n$ measure on each bounded set (see  Theorem \ref{theo: Steiner closed})   and therefore is $(\mathcal{H}^n, n)$ rectifiable.

Fix $ j \in \{0, \ldots , n-1\} $ and notice that $ \mathcal{H}^{j+1}(A^{(j)}) =0$. Therefore we can apply Lemma \ref{coarea} to conclude that
\begin{equation*}
  \int_{N_s|A^{(j)}} \ap J_{j+1}^{N_s}\bm{p}(a, \eta) \, d\mathcal{H}^n(a,\eta) =   \int_{A^{(j)}}\mathcal{H}^{n-1-j}(N_s|\{a\})\, d\mathcal{H}^{j+1}(a) =0
\end{equation*}
for every $ s > 0 $. It follows that $  \ap J_{j+1}^{N_s}\bm{p}(a, \eta)=0 $ for $ \mathcal{H}^n$ a.e.\ $ (a,\eta) \in N_s|A^{(j)} $ and  $ s > 0$. We infer from Lemma \ref{lem: tangent of normal bundle} that
\begin{equation*}\label{eq:latecurvinf}
    \kappa^\phi_{A,j+1}(a,\eta) = + \infty \qquad \textrm{for $ \mathcal{H}^n$ a.e.\ $ (a,\eta) \in N^\phi(A)|A^{(j)} $,}
\end{equation*}
which is precisely the assertion in \eqref{lem: stratification and curvature: 1}.

If $ m =0 $, then $ S $ is a countable set and $ \mathcal{H}^n(N^\phi(A,a))=0 $ for every $ a \in A^{(j)} $ and for every $ j \geq 1 $. This implies \eqref{lem: stratification and curvature: 2} for $ m=0$. Fix now $ j > m \geq 1$. Noting that $ \mathcal{H}^{n-m}(N^\phi(A,a)) =0$ for  $ a \in A^{(j)}$, we apply again the coarea formula in Lemma \ref{coarea} to obtain
\begin{equation*}
  \int_{N_s|S \cap A^{(j)}}\ap J^{N_s}_m\bm{p}(a, \eta) \, d\mathcal{H}^n(a,\eta) =   \int_{S \cap A^{(j)}}\mathcal{H}^{n-m}(N_s|\{a\})\, d\mathcal{H}^{m}(a) =0
\end{equation*}
for  $ s > 0$. As above this implies that $ \kappa^\phi_{A, m}(a,\eta) = + \infty$ for $ \mathcal{H}^n$ a.e.\ $ (a,\eta) \in N^\phi(A)|A^{(j)} \cap S $, which is equivalent to the assertion in \eqref{lem: stratification and curvature: 2}.
\end{proof}

\begin{Remark}\label{rmk: stratification}
If $ A $, $ m $ and $ S $ are as in Lemma \ref{lem: stratification and curvature} then
\begin{equation*}
    \bm{H}^\phi_{n-m}(a,\eta) = \bm{1}_{\widetilde{N}^\phi_m(A)}(a, \eta) \qquad \textrm{for $ \mathcal{H}^n$ a.e.\ $ (a,\eta) \in N^\phi(A)|A^{(m)}$},
\end{equation*}
\begin{equation*}
    \bm{H}^\phi_{n-m}(a,\eta) = 0 \qquad \textrm{for $ \mathcal{H}^n$ a.e.\ $ (a,\eta) \in N^\phi(A)|A^{(j)}$ and    $ j < m $}
\end{equation*}
and
\begin{equation*}
    \bm{H}^\phi_{n-m}(a,\eta) = 0 \qquad \textrm{for $ \mathcal{H}^n$ a.e.\ $ (a,\eta) \in N^\phi(A)|(A^{(j)} \cap S)$ and  $ j > m $.}
\end{equation*}
\end{Remark}

\medskip

\begin{Lemma}\label{lem: exterior normal basic properties closed}
	For every closed set $ A \subseteq \mathbf{R}^{n+1} $ the following statements hold.
	\begin{enumerate}[{\rm (a)}]
	    \item \label{lem: exterior normal basic properties closed: a} $	\mathcal{H}^0(N^\phi(A,a)) \in \{1,2\} $ for  $ a \in \bm{p}(\widetilde{N}^\phi_n(A)) $.
	   \item \label{lem: exterior normal basic properties closed: b} $  \mathcal{H}^m \left(     \big\{  a\in A^{(m)}  : \mathcal{H}^{n-m}\big(N^\phi(A,a) \setminus  \widetilde{N}^\phi_m(A,a)\big) > 0\big\}\right)=0 $ for $ m \in \{0, \ldots , n\} $.
	    \item \label{lem: exterior normal basic properties closed: c} $\mathcal{H}^n\big(\bm{p}\big[N^\phi(A) \setminus \widetilde{N}_n^\phi(A)\big]\big) =0   $; in particular,   $ \mathcal{H}^0(N^\phi(A,a) )\in \{1,2\}$ for $ \mathcal{H}^n $ a.e.\ $ a \in \bm{p}(N^\phi(A)) $.
	    \item  \label{lem: exterior normal basic properties closed: d} $	N^\phi(A,a) = - N^\phi(A,a)  $ for  $ a \in A $ with $ \mathcal{H}^0(N^\phi(A,a)) =2 $.
	    \item \label{lem: exterior normal basic properties closed: e} If $ X =  \bm{p}(N^\phi(A) \setminus \widetilde{N}_n^\phi(A)) \cap \bm{p}(\widetilde{N}^\phi_n(A))$, then $ \mathcal{H}^n(N^\phi(A)|X) =0  $.
	\end{enumerate}
\end{Lemma}
\begin{proof}
(a) Let $ (a,\eta)\in \widetilde{N}^\phi_n(A) $ and $ 0 < r <  \arl{A}{\phi}(a,\eta) $. Then $ 1 - r \chi^\phi_{A,i}(a+r\eta) > 0 $ for  $ i = 1, \ldots , n $ and, since these numbers are the eigenvalues of $ \Der \bm{\xi}^\phi_{A}(a+ r\eta)|\Tan(S^\phi(A, r), a + r\eta) $,  we conclude (noting Remark \ref{rem: tangent level sets and Wulff shapes} and \cite[3.1.21]{MR0257325})
\begin{equation*}
	\Tan(\partial \mathcal{W}^\phi, \eta) = \Der \bm{\xi}^\phi_{A}(a+ r\eta)[\Tan(S^\phi(A, r), a + r\eta)	] \subseteq \Tan(A,a).
\end{equation*}
Since $ N(A,a) \subseteq \Nor(A,a) \subseteq \Nor(\partial \mathcal{W}^\phi, \eta) $ and $ \dim \Nor(\partial \mathcal{W}^\phi, \eta) = 1 $, it follows that $ 	\mathcal{H}^0(N^\phi(A,a)) \in \{1,2\} $ and \ref{lem: exterior normal basic properties closed: a} is proved.

\medskip

(b) First, let $ m \in \{1, \ldots , n\} $, define $ P = \bigcup_{i=m}^n\widetilde{N}^\phi_i(A)$ and recall that $ A^{(m)} $ is a countably $m$-rectifiable set (see above). For $ s > 0 $ define $ N_s $ as in Theorem  \ref{theo: Steiner closed} and set $W_s = [N_s  \setminus P]|A^{(m)}$, which is an  $(\mathcal{H}^n, n)$ rectifiable set. Noting from Lemma \ref{lem: tangent of normal bundle} that $  \ap J^{W_s}_m \bm{p}(a,\eta) =0  $ for $ \mathcal{H}^n $ a.e.\ $(a,\eta) \in W_s $,
we conclude from Lemma \ref{coarea} that
\begin{equation*}
\int_{A^{(m)}} \mathcal{H}^{n-m}\big((N_s|\{x\}) \setminus (P|\{x\})\big)\, d\mathcal{H}^m(x)=0,
\end{equation*}
and thus we infer that
$$
\mathcal{H}^m\left(      \left\{ x \in A^{(m)}:  \mathcal{H}^{n-m}((N_s|\{x\}) \setminus P|\{x\})>0 \right\}\right)=0   .
$$
Furthermore, we have
\begin{equation*}
      \{ x\in  A^{(m)} :  \mathcal{H}^{n-m}(N^\phi(A,x) \setminus (P|\{x\}))>0 \} = \bigcup_{i=1}^\infty  \{ x\in A^{(m)}  :  \mathcal{H}^{n-m}((N_{s_i}|\{x\}) \setminus (P|\{x\}))>0 \}
\end{equation*}
for any positive sequence $ s_i \searrow 0 $. Since also $\mathcal{H}^n(N^\phi(A,a)\setminus \widetilde{N}^\phi(A,a))=0$ for $a\in A^{(0)}$, we obtain that
\begin{equation}\label{eq:eins}
\mathcal{H}^m \left( \left\{  a\in A^{(m)}  : \mathcal{H}^{n-m}\bigg(N^\phi(A,a) \setminus \bigcup_{i=m}^n \widetilde{N}^\phi_i(A,a)\bigg) > 0\right\}\right)=0
\end{equation}
for $ m \in \{0, \ldots , n\}$. Now let $m\in \{0,\ldots,n-1\}$. Since by Lemma  \ref{lem: stratification and curvature} it holds that
\begin{equation*}
    \mathcal{H}^n \bigg( \big(N^\phi(A)|A\big) \cap \bigcup_{i=m+1}^n \widetilde{N}^\phi_i(A)\bigg) =0,
\end{equation*}
an application of the coarea formula in Lemma \ref{coarea} yields that
\begin{equation*}
\int_{A^{(m)}}\mathcal{H}^{n-m} \left(N_s|\{x\}\cap\bigcup _{i=m+1}^n \widetilde{N}^\phi_i(A,x)\right)\, d\mathcal{H}^m(x)=0 \qquad \textrm{for every $ s > 0 $,}
\end{equation*}
whence, as above, we infer
\begin{equation}\label{eq:zwei}
\mathcal{H}^m \left( \left\{  a\in A^{(m)}  : \mathcal{H}^{n-m}\bigg(N^\phi(A,a) \cap \bigcup_{i=m+1}^n \widetilde{N}^\phi_i(A,a)\bigg) > 0\right\}\right)=0 .
\end{equation}
Now the assertion follows from \eqref{eq:eins} and \eqref{eq:zwei}.

\medskip

\ref{lem: exterior normal basic properties closed: c} Since $ \bm{p}(N^\phi(A)) = \bigcup_{i=0}^n A^{(i)}$, $ \mathcal{H}^n(A^{(j)}) =0 $ for  $ j\in\{0,\ldots,n-1\} $ and $ \bm{p}\big[N^\phi(A) \setminus \widetilde{N}_n^\phi(A)\big] = \bm{p}(N^\phi(A)) \cap \{a : \mathcal{H}^0(N^\phi(A,a) \setminus \widetilde{N}^\phi_n(A,a)) >0 \} $, we conclude that \ref{lem: exterior normal basic properties closed: c} directly follows from \ref{lem: exterior normal basic properties closed: b} with $ m = n $.

\medskip

\ref{lem: exterior normal basic properties closed: d} The cone $ \{ tu : t > 0, \, u \in N(A,a)   \} $ is convex for  $ a \in \bm{p}(N^\phi(A)) $. Consequently, we have $N(A,a) = - N(A,a) $ for  $ a \in A $ with $ \mathcal{H}^0(N(A,a)) =2 $. Since $ \nabla \phi(N(A,a)) = N^\phi(A,a) $ for  $ a \in \bm{p}(N^\phi(A))$, we conclude that $N^\phi(A,a) = - N^\phi(A,a) $ for $ a \in A $ with $ \mathcal{H}^0(N^\phi(A,a)) =2 $.

\medskip

\ref{lem: exterior normal basic properties closed: e} Finally, let $ X $ be the set defined in \ref{lem: exterior normal basic properties closed: e}. Employing again the sets $ N_s $ as defined in Theorem \ref{theo: Steiner closed}, we notice that $ \ap J^{N_s}_n\bm{p}(a,u) > 0 $ for $ \mathcal{H}^n $ a.e.\ $(a,u) \in \widetilde{N}^\phi_n(A) $ and, by the coarea formula and \ref{lem: exterior normal basic properties closed: c}, we get
\begin{equation*}
    \int_{(N_s \cap \widetilde{N}^\phi_n(A))|X}\ap J^{N_s}_n\bm{p}(a,\eta)\, d\mathcal{H}^n(a,\eta) =0
\end{equation*}
for every $ s > 0 $. It follows that $ \mathcal{H}^n(\widetilde{N}^\phi_n(A)|X) =0$. Since by \ref{lem: exterior normal basic properties closed: a} and \ref{lem: exterior normal basic properties closed: d} we have that
\begin{equation*}
    N^\phi(A)|X = (\widetilde{N}^\phi_n(A)| X) \cup \{   (a, -\eta) : (a,\eta)\in \widetilde{N}^\phi_n(A)| X\},
\end{equation*}
we obtain \ref{lem: exterior normal basic properties closed: e}.
\end{proof}

\begin{Remark}\label{rem normal cone convex}
If $ A \subseteq \mathbf{R}^{n+1} $ is a convex body (with non-empty interior), then $\mathcal{H}^0(N^\phi(A,a)) = 1 $ for every $ a \in \bm{p}(\widetilde{N}_n^\phi(A)) $.
\end{Remark}

\medskip

We now prove a very general disintegration formula. This result, which is of independent interest, plays a key role in the proof of Corollary \ref{Cor: Alexandrov} through Lemma \ref{Lem:abs} and the subsequent Lemma \ref{lem:thetared}.

\begin{Theorem}\label{thm:disint}
Let $\varnothing\neq A \subseteq \mathbf{R}^{n+1}$ be a closed set and   $m\in\{0,\ldots,n\}$. Then there exists a positive real-valued $ \mathcal{H}^n \restrict N^\phi(A)$ measurable function $ \rho^\phi_{A,m}$ on $ N^\phi(A) $ such that the following statements hold.
\begin{enumerate}[{\rm (a)}]
\item $  0< \rho^\phi_{A,m}(a, \eta) \leq c  $ for $ \mathcal{H}^n $ a.e.\ $(a, \eta) \in N^\phi(A) $, where $ c $ depends only on $ \phi $ and $ n $.
\item If $m\in\{0,n\}$ or if $ \phi $ is the Euclidean norm, then $ \rho^\phi_{A,m}(a, \eta) = 1 $ for $ \mathcal{H}^n $ a.e.\ $ (a, \eta) \in N(A) $.
    \item For every Borel set $ B \subseteq \mathbf{R}^{n+1}$ with $ \sigma $-finite $\mathcal{H}^m $ measure it holds that
\begin{equation}\label{thm:disint eq1}
\bm{H}^\phi_{A, n-m}(a, \eta)\, \bm{1}_{N^\phi(A)|B}(a, \eta) =  \bm{1}_{\widetilde{N}^\phi_m(A)|(A^{(m)} \cap B)}(a, \eta)
\end{equation}
for $ \mathcal{H}^n$ a.e.\ $(a, \eta) \in N^\phi(A)|B $   and
\begin{flalign*}
  & \int_{ N^\phi(A)| B}\mathbf{1}_D(a,\eta)  \phi(\bm{n}^\phi(\eta))\,J^\phi_A(a,\eta)\, \bm{H}^\phi_{A,n-m}(a,\eta)\, d\mathcal{H}^n(a,\eta)\\
  & \qquad = \int_{B \cap A^{(m)}}\int_{N^\phi(A,a)}\mathbf{1}_D(a,\eta)\phi(\bm{n}^\phi(\eta)) \, \rho^\phi_{A,m}(a, \eta)\,  d\mathcal{H}^{n-m}(\eta)\, d\mathcal{H}^m(a)
\end{flalign*}
for every Borel set $ D \subseteq \R^{n+1}\times\R^{n+1}$.
\end{enumerate}
\end{Theorem}

 \begin{proof}
 Since $ N^\phi(A)|B=\bigcup_{j=0}^n ( N^\phi(A)|(B \cap A^{(j)})) $, the equality in \eqref{thm:disint eq1} follows from Remark \ref{rmk: stratification}.

 Since the case $m=0$ is easy to check directly, we assume that $m\ge 1$ in the following.

Let $ \tau_i : \widetilde{N}^\phi(A) \rightarrow \mathbf{R}^{n+1}$ and $ \zeta_i : \widetilde{N}^\phi(A) \rightarrow \mathbf{R}^{n+1} \times \mathbf{R}^{n+1}$ be the maps defined in Lemma \ref{lem: tangent of normal bundle} for $ i = 1, \ldots , n$.
For  $(a, \eta) \in \widetilde{N}^\phi(A) $ we define $ T(a, \eta)$ to be the linear space generated by $ \zeta_1(a, \eta), \ldots , \zeta_n(a, \eta)$ and we notice that
\begin{align}\label{thm:disint eq}
1\ge \big\| \bm{p}| T(a, \eta) \|^m & \ge \big\| {\textstyle\bigwedge_m} \big(\bm{p} |T(a, \eta)\big) \big\| \notag \\
&\geq \bigg|{\textstyle\bigwedge_m} \bm{p} \bigg( \frac{\zeta_1(a,\eta) \wedge \ldots \wedge \zeta_m(a,\eta)}{|\zeta_1(a,\eta) \wedge \ldots \wedge \zeta_m(a,\eta)|}\bigg)\bigg|
    = \frac{|\tau_1(a,\eta) \wedge \ldots \wedge \tau_m(a,\eta)|}{| \zeta_1(a,\eta) \wedge \ldots \wedge \zeta_m(a,\eta)|} > 0
\end{align}
for  $(a, \eta) \in \widetilde{N}^\phi_m(A) $. Therefore we define
\[
\varrho^\phi_{A,m}(a,\eta) =
\begin{cases}\displaystyle{
    \frac{J^\phi_A(a,\eta)}{\big\|{\textstyle\bigwedge_m} (\bm{p} |T(a, \eta))\big\|} }& \textrm{for $(a, \eta) \in \widetilde{N}^\phi_m(A)$},\\[2ex]
    1 & \textrm{for $(a, \eta) \in N^\phi(A) \setminus \widetilde{N}^\phi_m(A)$}.
\end{cases}
\]
Notice that $ \rho^\phi_{A,n}(a, \eta) = 1 $ for $ \mathcal{H}^n $ a.e.\ $ (a, \eta) \in N^\phi(A) $ by Lemma \ref{lem: tangent of normal bundle}.

If $m\in\{1,\ldots,n-1\}$, we combine
\begin{equation*}
    | \tau_1(a, \eta) \wedge \ldots \wedge \tau_n(a, \eta)| \leq | \tau_1(a, \eta) \wedge \ldots \wedge \tau_m(a, \eta)|\cdot | \tau_{m+1}(a, \eta) \wedge \ldots \wedge \tau_n(a, \eta)|
\end{equation*}
(see \cite[1.7.5]{MR0257325}) with the estimate in \eqref{eq:dlate1} from Remark \ref{rem:altdescreigen} to get
\begin{flalign*}
\rho^\phi_{A,m}(a, \eta) & \leq \frac{ | \tau_1(a, \eta) \wedge \ldots \wedge \tau_n(a, \eta)|}{ | \zeta_1(a, \eta) \wedge \ldots \wedge \zeta_n(a, \eta)|} \cdot \frac{ | \zeta_1(a, \eta) \wedge \ldots \wedge \zeta_m(a, \eta)|}{ | \tau_1(a, \eta) \wedge \ldots \wedge \tau_m(a, \eta)|} \\
& \leq \frac{C}{c^2}\cdot \frac{| \tau_{m+1}(a, \eta) \wedge \ldots \wedge \tau_n(a, \eta)|}{|\zeta_{m+1}(a, \eta) \wedge \ldots \wedge \zeta_n(a, \eta)|} =  \frac{C}{c^2}
\end{flalign*}
for $(a, \eta) \in \widetilde{N}^\phi_m(A)$, which provides the required finite upper bound with constants $0<c\le C<\infty$ depending only on $n,\phi$.

Let $ N_s $  be the set defined in Theorem \ref{theo: Steiner closed} for $ s > 0 $. By Lemma \ref{lem: tangent of normal bundle},  we have  $ \ap J^{N_s}_m \bm{p}(a, \eta) = \big\|{\textstyle\bigwedge_m} \big(\bm{p} |T(a, \eta)\big)\big\|$ for $ \mathcal{H}^n$ a.e.\ $(a, \eta) \in N_s  $, for   $ s > 0$.
An application of the coarea formula shows that
\begin{align*}
&\int_{ N_s| (B\cap A^{(m)})}  \phi(\bm{n}^\phi(\eta))\,J^\phi_A(a,\eta)\, \mathbf{1}_{D}(a,\eta)\,\mathbf{1}_{\widetilde{N}^\phi_m(A)}(a,\eta)\, d\mathcal{H}^n(a,\eta)\notag\\
&\qquad=\int_{B\cap A^{(m)}}\int_{N_s\cap \bm{p}^{-1}(\{a\})}\phi(\bm{n}^\phi(\eta))\varrho_{A,m}^\phi(a,\eta)\mathbf{1}_{\widetilde{N}^\phi_m(A)}(a,\eta)\,\mathbf{1}_{D}(a,\eta)\, d\mathcal{H}^{n-m}(a,\eta)\, d\mathcal{H}^m(a)
\end{align*}
for  $ s > 0 $. Applying the monotone convergence theorem in combination with Lemma \ref{lem: exterior normal basic properties closed} (b) and \eqref{thm:disint eq1}, we obtain
\begin{align*}
&\int_{ N^\phi(A)| B}\mathbf{1}_D(a,\eta)  \phi(\bm{n}^\phi(\eta))\,J^\phi_A(a,\eta)\, \bm{H}^\phi_{A,n-m}(a,\eta)\,  d\mathcal{H}^n(a,\eta)\\
&\qquad =\int_{\widetilde{N}^\phi_m(A)| (B\cap A^{(m)})}  \phi(\bm{n}^\phi(\eta))\,J^\phi_A(a,\eta)\,  \mathbf{1}_{D}(a,\eta)\, d\mathcal{H}^n(a,\eta)\notag\\
&\qquad=\int_{B\cap A^{(m)}}\int_{\widetilde{N}^\phi_m(A,a)}\phi(\bm{n}^\phi(\eta))\varrho_{A,m}^\phi(a,\eta)\, \mathbf{1}_{D}(a,\eta)\, d\mathcal{H}^{n-m}(\eta)\, d\mathcal{H}^m(a)\\
&\qquad=\int_{B\cap A^{(m)}}\int_{N^\phi(A,a)}\phi(\bm{n}^\phi(\eta))\varrho_{A,m}^\phi(a,\eta)\, \mathbf{1}_{D}(a,\eta)\, d\mathcal{H}^{n-m}(\eta)\, d\mathcal{H}^m(a).
\end{align*}

Finally, suppose that $\phi$ is the Euclidean norm. Then $ \{\zeta_{\lambda(1)}(a, \eta) \wedge \ldots \wedge \zeta_{\lambda(m)}(a, \eta) : \lambda \in \Lambda(n,m) \}$ is an orthogonal basis of $ \bigwedge_m T(a, \eta) $ for every $(a, \eta) \in \widetilde{N}(A)$. Fix now $(a, \eta) \in \widetilde{N}^\phi_m(A)$ and $ \xi \in \bigwedge_m T(a, \eta)$ such that
\begin{equation*}
    |\xi| = 1 \qquad \textrm{and} \qquad \xi = \sum_{\lambda \in \Lambda(n,m)} c_\lambda \cdot \frac{\zeta_{\lambda(1)}(a, \eta) \wedge \ldots \wedge \zeta_{\lambda(m)}(a, \eta)}{|\zeta_{\lambda(1)}(a, \eta) \wedge \ldots \wedge \zeta_{\lambda(m)}(a, \eta)|}.
\end{equation*}
It follows from the orthogonality that $ |c_\lambda| \leq 1 $ for every $ \lambda \in \Lambda(n,m)$. Therefore, denoting with $ \lambda_0 \in \Lambda(n,m)$ the map such that $ \lambda_0(i) = i$ for every $ i \in \{1, \ldots , m\}$, we notice that
\begin{equation*}
   {\textstyle \bigwedge_m} \bm{p}(\xi)  = c_{\lambda_0} \frac{\tau_1(a, \eta) \wedge \ldots \wedge \tau_m(a, \eta)}{| \zeta_1(a, \eta) \wedge \ldots \wedge \zeta_m(a, \eta)|}
\end{equation*}
and we infer in combination with \eqref{thm:disint eq} that
\begin{equation*}
    \big\| {\textstyle\bigwedge_m} \big(\bm{p} |T(a, \eta)\big) \big\| = \frac{|\tau_1(a, \eta) \wedge \ldots \wedge \tau_m(a, \eta)|}{| \zeta_1(a, \eta) \wedge \ldots \wedge \zeta_m(a, \eta)|}.
\end{equation*}
We can now easily conclude that $ \varrho^\phi_{A,m}(a, \eta) = 1 $ for every $(a, \eta) \in \widetilde{N}^\phi_m(A)$ is a suitable choice.
\end{proof}

\subsection{Relation between Euclidean and anisotropic curvatures}

We consider the map
$$ T : \mathbf{R}^{n+1} \times \partial \mathcal{W}^\phi \rightarrow \mathbf{R}^{n+1} \times \mathbf{S}^n, \qquad (a,\eta)\mapsto T(a, \eta) = (a, \bm{n}^\phi(\eta)). $$
By equation \eqref{eq: normal wulff shape and gradient phi} in the introduction, $ T $ is a $ \mathcal{C}^1$-diffeomorphism whose inverse is $ T^{-1}(a,u) = (a, \nabla \phi(u)) $ for $(a,u) \in \mathbf{R}^{n+1} \times \mathbf{S}^n$. In particular, $ \Der T(a, \eta) : \mathbf{R}^{n+1} \times \Tan(\partial \mathcal{W}^\phi, \eta) \rightarrow \mathbf{R}^{n+1} \times \Tan(\mathbf{S}^n, \bm{n}^\phi(\eta)) $ is an isomorphism and
$$ \Der T(a,\eta)(\tau, \upsilon) = (\tau, \Der \bm{n}^\phi(\eta)(\upsilon)) \quad \textrm{for $(\tau, \upsilon) \in \mathbf{R}^{n+1} \times \Tan(\partial \mathcal{W}^\phi, \eta) $}.$$
As already recalled in \eqref{eq: phi normal vs euclidean normal}, we have $ T(N^\phi(A)) = N(A) $.

\begin{Remark}\label{rem: approx tangent cones}
Suppose $ X $ is a normed vector space, $ \mu $ is a measure over $ X $ and $ f, g : X \rightarrow Y$ are functions differentiable at $ a $. If the $ m $ dimensional density of $ \mu \restrict \{x : f(x) \neq g(x) \}$ is zero at $a $, then
$$ \Der f(a) (v) = \Der g(a)(v) \qquad \textrm{for every $ v \in \Tan^m(\mu, a) $,} $$
where $ \Tan^m(\mu, a)$ is the cone of the $ (\mu,m)$ approximate tangent vectors of $ \mu $ at $ a $ (see \cite[3.2.16]{MR0257325}). One can check this remark from the definitions.
\end{Remark}

The following result expresses the mean curvature functions $ \bm{H}^\phi_{A,j} \cdot \bm{1}_{\widetilde{N}^\phi_n} = E^\phi_{A,j} \cdot \bm{1}_{\widetilde{N}^\phi_n(A)}$ in terms of the Euclidean generalized curvatures $ \kappa_{A,1}, \ldots , \kappa_{A,n} $ of $ A $.
\begin{Theorem}
Let $\varnothing\neq A \subseteq \mathbf{R}^{n+1}$ be a closed set, and let $ \phi $ be a uniformly convex $ \mathcal{C}^2$ norm.

Then $ \mathcal{H}^n(T(\widetilde{N}^\phi_d(A)) \triangle \widetilde{N}_d(A)) =0 $ for  $ d =0, \ldots, n $. Moreover, if  $ e_1, \ldots , e_n : \widetilde{N}(A) \rightarrow \mathbf{R}^{n+1}$ are maps such that
$$ \Der \bm{\nu}_A(a + r u)(e_i(a,u)) = \rchi_{A,i}(a+ru) e_i(a,u) \quad \textrm{for $ i =1, \ldots , n $, $ 0 < r < \underline{\bm{r}_A}(a,u) $} \text{ and }(a,u)\in \widetilde{N}(A) , $$
then
\begin{flalign*}
    E^\phi_{A,j}(a, \eta)  &= \sum_{\lambda \in \Lambda(n,j)}\bigg(\prod_{i=1}^j\kappa_{A, \lambda(i)}(T(a, \eta)) \bigg)  \\
    & \qquad \times \bigg(  \bigwedge_{i=1}^j\Der(\nabla \phi)(\bm{n}^\phi(\eta))( e_{\lambda(i)}(T(a,\eta))) \bigg)  \bullet \big(\bigwedge_{i=1}^j e_{\lambda(i)}(T(a,\eta))\big) \bigg)
\end{flalign*}
for $ \mathcal{H}^n $ a.e.\ $(a, \eta) \in \widetilde{N}^\phi_n(A) $ and $ j\in\{1,\ldots, n\} $.
\end{Theorem}

\begin{proof}
Given the maps $ e_1, \ldots , e_n $ as in the statement of the Theorem, we define the maps $ z_i : \widetilde{N}(A) \rightarrow \mathbf{R}^{n+1} \times \mathbf{R}^{n+1} $, for $ i = 1, \ldots , n $,  so that
 \begin{equation*}
 z_i(a,u) =
 \begin{cases}
 (e_i(a,u),  \kappa_{A,i}(a, u)e_i(a,u)), & \textrm{if $ \kappa_{A,i}(a,u) < \infty $},\\
 	(0, e_i(a,u)), & \textrm{if $\kappa_{A,i}(a, u) = + \infty$.}
 \end{cases}
 \end{equation*}
 Then we choose the maps $ \tau_1, \ldots , \tau_n $ and $ \zeta_1, \ldots , \zeta_n $ as in Lemma \ref{lem: tangent of normal bundle}. Notice that $ e_1(a,u). \ldots , e_n(a,u) $ form an orthonormal basis of $ \Tan(\mathbf{S}^n, u) = \Tan(\partial \mathcal{W}^\phi, \nabla \phi(u))$.

Suppose $ W \subseteq N^\phi(A)$ is $ \mathcal{H}^n$ measurable and $ \mathcal{H}^n(W) < \infty$ and define $ W^\ast $ as
the set of $(a, \eta) \in W $ such that
$$ \Tan^n(\mathcal{H}^n\restrict W, (a, \eta)) = \textrm{span}\{\zeta_1(a, \eta), \ldots , \zeta_n(a, \eta) \} $$
and
$$\Tan^n(\mathcal{H}^n\restrict T(W), T(a, \eta)) = \textrm{span}\{z_1(T(a, \eta)), \ldots , z_n(T(a, \eta)) \}. $$
Since $ T $ is a bi-lipschitz map it follows from Lemma \ref{lem: tangent of normal bundle} that $ \mathcal{H}^n(W \setminus W^\ast) =0$. Fix now $(a,\eta) \in W^\ast$ and define $ V = \ker [\bm{p} | \Tan^n(\mathcal{H}^n\restrict W, (a, \eta)) ] $,
$ d = n-\dim V $,
$ V'  = \ker [\bm{p} | \Tan^n(\mathcal{H}^n\restrict T(W), T(a, \eta)) ]$ and $ d' = n- \dim V'$. By \cite{MR4117503}[Lemma B.2]  we infer that
$$
\Der T(a,\eta) \big(\Tan^n(\mathcal{H}^n\restrict W, (a, \eta))\big) = \Tan^n(\mathcal{H}^n \restrict T(W), T(a,\eta)), $$
whence we can easily deduce that $ \Der T(a,\eta)(V) = V' $. This implies in particular that $ d = d' $. Consequently it holds that
$$ T(W^\ast \cap \widetilde{N}^\phi_d(A)) \subseteq \widetilde{N}_d(A) \quad \textrm{and} \quad \widetilde{N}_d(A) \cap T(W^\ast) \subseteq T(\widetilde{N}^\phi_d(A)) $$
for every $ d=0, \ldots , n $. Since $ N^\phi(A) $ is a countable union of $ \mathcal{H}^n $ measurable sets $ W $ with finite $ \mathcal{H}^n$ measure, we conclude
$$ \mathcal{H}^n\big(T(\widetilde{N}^\phi_d(A)) \triangle \widetilde{N}_d(A)\big) =0. $$

For  $ r > 0 $ we define
$$ Q_r = \{ (a, \eta) \in N^\phi(A): \bm{r}^\phi_A(a, \eta) \ge r, \; \bm{r}_A(T(a,\eta)) \ge r \}. $$
From the upper semicontinuity of $ \bm{r}^\phi_A $ and $ \bm{r}_A $ (see \cite[Lemma 2.35]{kolasinski2021regularity}) it follows that $ Q_r $ is relatively  closed in $ N^\phi(A) $ (in particular it is a Borel set) and $ \mathcal{H}^n \restrict Q_r $ is finite on compact sets by Theorem \ref{theo: Steiner closed}(a). Notice that $ Q_r \subseteq Q_s $ is $ s \leq r $ and $ \bigcup_{r >0} Q_r = N^\phi(A) $. In addition, for $ r > 0$ we  define $ Q^\ast_r $ as the set of all $(a, \eta) \in Q_r $ such that $ \underline{\bm{r}_A}(a, \eta) = \bm{r}_A(a, \eta)$, the $ n $-dimensional density of $ \mathcal{H}^n \restrict (\mathbf{R}^{n+1} \times \mathbf{R}^{n+1}) \setminus Q_r $ is zero at $(a, \eta)$ and $ \Tan^n(\mathcal{H}^n \restrict Q_r, (a, \eta)) = \lin\{\zeta_1(a, \eta), \ldots , \zeta_n(a, \eta) \}$. By Lemma \ref{lem: tangent of normal bundle}, \cite[2.10.19]{MR0257325} and Lemma \ref{theo: distance twice diff} it follows that $ \mathcal{H}^n(Q_r \setminus Q^\ast_r) =0$ for every $ r > 0 $; moreover, one can check directly from the definitions (see \cite[3.2.16]{MR0257325}) that $ Q^\ast_r \subseteq Q^\ast_s $ if $ s \leq r $.

Fix now $(\hat{a}, \hat{\eta}) \in \widetilde{N}^\phi_d(A) \cap T^{-1}(\widetilde{N}_d(A)) \cap Q^\ast_r $ for some $ r > 0 $. For  $ s > 0 $ define $ L_s : \mathbf{R}^{n+1} \times \partial \mathcal{W}^\phi \rightarrow \mathbf{R}^{n+1} $ by $ L_s(a, \eta) = a + s \bm{n}^\phi(\eta) $. Choose a function $ u : \mathbf{R}^{n+1} \rightarrow \mathbf{R}^{n+1}$ such that $ u(x) \in \bm{\nu}_A(x) $ for every $ x \in \mathbf{R}^{n+1} \setminus A$. For $ 0 < s \leq r $,
$$ (\nabla \phi \circ u \circ L_s)(a,\eta) = \bm{q}(a, \eta)= \eta \quad \textrm{for   $(a,\eta) \in Q_s $} $$
and $ \nabla \phi \circ u \circ L_s $ is differentiable at $(\hat{a}, \hat{\eta})$ by Remark \ref{rem: principla curvatures basic rem}. Remark \ref{rem: approx tangent cones} yields $ \Der (\nabla \phi \circ u \circ L_s)(\hat{a}, \hat{\eta})(\xi) = \bm{q}(\xi)$ for  $ \xi \in \Tan^n(\mathcal{H}^n \restrict Q_s, (\hat{a}, \hat{\eta}))$. Hence, for $ 1 \leq i \leq d$ and  $ s < r $ we compute
\begin{equation}\label{theo: comparison eq1}
    \big[\Der(\nabla \phi)(\bm{n}^\phi(\hat{\eta})) \circ \Der u( \hat{a} + s\bm{n}^\phi(\hat{\eta} )\big]\big(\tau_i(\hat{a}, \hat{\eta}) + s \kappa^\phi_{A,i}(\hat{a}, \hat{\eta}) \Der \bm{n}^\phi(\hat{\eta})(\tau_i(\hat{a}, \hat{\eta}))\big) = \kappa^\phi_{A,i}(\hat{a}, \hat{\eta})\tau_i(\hat{a}, \hat{\eta}).
\end{equation}
Noting that
\begin{flalign*}
 &s \Der u( \hat{a} + s\bm{n}^\phi(\hat{\eta} ))\bigg( \sum_{j=1}^d \big[\Der\bm{n}^\phi(\hat{\eta})(\tau_i(\hat{a}, \hat{\eta})) \bullet e_j(T(\hat{a}, \hat{\eta}))\big] e_j(T(\hat{a}, \hat{\eta}))\bigg) \\
 & \qquad  \qquad = s\sum_{j=1}^d\big[\Der\bm{n}^\phi(\hat{\eta})(\tau_i(\hat{a}, \hat{\eta})) \bullet e_j(T(\hat{a}, \hat{\eta}))\big] \frac{\kappa_{A,j}(T(\hat{a}, \hat{\eta}))}{1 + s \kappa_{A,j}( T(\hat{a}, \hat{\eta}))}e_j(T(\hat{a}, \hat{\eta})) \to 0 \quad \textrm{as $ s \to 0 $}
\end{flalign*}
and
\begin{flalign*}
     &s \Der u( \hat{a} + s\bm{n}^\phi(\hat{\eta} ))\bigg( \sum_{j=d+1}^n \big[\Der\bm{n}^\phi(\hat{\eta})(\tau_i(\hat{a}, \hat{\eta})) \bullet e_j(T(\hat{a}, \hat{\eta}))\big] e_j(T(\hat{a}, \hat{\eta}))\bigg)  \\
      & \qquad  \qquad = \sum_{j=d+1}^n\big[\Der\bm{n}^\phi(\hat{\eta})(\tau_i(\hat{a}, \hat{\eta})) \bullet e_j(T(\hat{a}, \hat{\eta}))\big]e_j(T(\hat{a}, \hat{\eta})) \quad \textrm{for  $ s > 0 $,}
\end{flalign*}
we conclude from \eqref{theo: comparison eq1} that
\begin{flalign}\label{theo: comparison eq2}
&\lim_{s \to 0}\big[\Der(\nabla \phi)(\bm{n}^\phi(\hat{\eta})) \circ \Der u( \hat{a} + s\bm{n}^\phi(\hat{\eta} )\big]\big(\tau_i(\hat{a}, \hat{\eta})) \\
& \qquad  = \kappa^\phi_{A,i}(\hat{a}, \hat{\eta})\tau_i(\hat{a}, \hat{\eta}) \notag \\
& \qquad \qquad -\kappa^\phi_{A,i}(\hat{a}, \hat{\eta}) \sum_{j=d+1}^n \big[\Der\bm{n}^\phi(\hat{\eta})(\tau_i(\hat{a}, \hat{\eta})) \bullet e_j(T(\hat{a}, \hat{\eta}))\big] \Der(\nabla \phi)(\bm{n}^\phi(\hat{\eta}))(e_j(T(\hat{a}, \hat{\eta}))) \notag
\end{flalign}
for $i \leq d $.

Now  choose $ d = n $ in the previous paragraph and define the linear maps $ T_s : \Tan(\partial \mathcal{W}^\phi, \hat{\eta}) \rightarrow \Tan(\partial \mathcal{W}^\phi, \hat{\eta}) $ by
$$ T_s = \Der(\nabla \phi)(\bm{n}^\phi(\hat{\eta})) \circ \Der u( \hat{a} + s\bm{n}^\phi(\hat{\eta} ) $$
and $ T_0 : \Tan(\partial \mathcal{W}^\phi, \hat{\eta}) \rightarrow \Tan(\partial \mathcal{W}^\phi, \hat{\eta}) $ by $ T_0 (\tau_i(\hat{a},\hat{\eta}) ) = \kappa^\phi_{A,i}(\hat{a},\hat{\eta} )\tau_i (\hat{a},\hat{\eta})$ for $ i = 1, \ldots , n$. Denoting by $ \| \cdot \| $ the operator norm, we notice that
$ \sup_{s > 0}\| T_s \| < \infty $  and $ T_s(v) \rightarrow T_0(v)$ for each $ v \in \Tan(\partial \mathcal{W}^\phi. \hat{\eta}) $ by \eqref{theo: comparison eq2}. Therefore, $ \lim_{s \to \infty}\| T_s - T_0 \| = 0$ and by continuity
\begin{flalign*}
H^\phi_{A,j}(\hat{a}, \hat{\eta}) & = \trace \big({\textstyle \bigwedge_j}T_0\big) \\
& = \trace\big({\textstyle \bigwedge_j}\lim_{s \to 0}T_s \big) = \lim_{s \to 0}\trace \big({\textstyle \bigwedge_j}T_s \big).
\end{flalign*}
Computation of the $ \trace \big({\textstyle \bigwedge_j}T_s \big)  $ by means of the orthonormal basis $ \{\bigwedge_{i=1}^j e_{\lambda(i)}(T(\hat{a}, \hat{\eta})) : \lambda \in \Lambda(n,j) \}$ of $ \bigwedge_j \Tan(\partial \mathcal{W}^\phi, \hat{\eta})$, we get the conclusion.
\end{proof}

\section{Differentiability of the volume function}\label{sec:three}

In this section, employing the Steiner-type formula from the previous section, we study the differentiability properties of the (localized) parallel volume function $V $ of an anisotropic tubular neighbourhood around an arbitrary compact (closed) set. In particular, we obtain an expression for the left and right derivative of $V$ in terms of the anisotropic curvatures of the compact set and we deduce a geometric characterization of the differentiability points of $ V $. The results of this section extend \cite[eq.\ (4.5) and (4.6), Corollary 4.5]{MR2031455} and complement some of the results in \cite{MR4329249}, see also Remark \ref{rem:finsec3} below and the preceding work \cite{MR442202,MR2865426,MR3939268}.

\begin{Remark}\label{remark: bdry tubular neigh}
Notice that if $ \rho > 0 $, $ y \in S^\phi(A, \rho)$ and $a\in \bm{\xi}^\phi_A(x)$, that is, $ a \in A $ and $ \phi^\ast(y-a) = \bm{\delta}^\phi_A(y) $, then
\begin{equation*}
U^\phi(a, \rho) \cap S^\phi(A, \rho) = \varnothing \quad \textrm{and} \quad B^\phi(a, \rho) \subseteq B^\phi(A, \rho).
\end{equation*}
Hence we deduce that $ \bm{p}\big(N(B^\phi(A, \rho))\big) = \partial^v B^\phi(A, \rho) = \partial^v_+ B^\phi(A, \rho)$ and the function $ u $ defined in Theorem \ref{theo: derivative of volume} satisfies $ u(x) = \bm{n}(B^\phi(A, \rho), x)$ for every $ x \in \partial^v_+ B^\phi(A, \rho)$.
\end{Remark}

The following auxiliary result is used in the proof of Theorem \ref{theo: derivative of volume}. At the same time it provides an interesting insight into  the nature of the cut points in $ \Cut^\phi(A) \cap \Unp^\phi(A)$. In fact, in this regard we recall that there are  closed (compact) sets $ A \subseteq \mathbf{R}^{n+1}$ such that $ \mathcal{H}^n[S^\phi(A,\rho)\setminus (\partial B^\phi(A, \rho) \cup \Unp^\phi(A))] >0 $ for some $ \rho >0$; for instance, let $ A \subseteq \mathbf{R}^2$   be the union of two parallel lines (segments) at distance $ 2 $ and $ \rho = 1 $. In view of these simple examples, the second equation in \eqref{lem: unp and level sets eq2} (which holds for every $ \rho > 0 $!) is quite surprising since $ \Cut^\phi(A) \cap \Unp^\phi(A) $ can be a much larger set than $\mathbf{R}^{n+1} \setminus (A \cup \Unp^\phi(A)) $. In fact, since $\mathbf{R}^{n+1} \setminus (A \cup \Unp^\phi(A)) $ is always an $ n $-dimensional set (see section \ref{sec:distnbperi}), it follows from the example in \cite{MR2390513} that the set $ \Cut^\phi(A) \cap \Unp^\phi(A)$ can be an $( n+1 )$-dimensional set!

\begin{Lemma}\label{lem: unp and level sets}
Let $ A \subseteq \mathbf{R}^{n+1}$ be a closed set and $ \rho > 0 $. Let $ f_\rho : N^\phi(A) \cap \{  \bm{r}^\phi_A \geq \rho   \} \rightarrow S^\phi(A, \rho) $ be defined by $ f_\rho(a, \eta) = a + \rho \eta $. Then
\begin{align}
    f_\rho(N^\phi(A) \cap \{  \bm{r}^\phi_A \geq \rho   \})& = S^\phi(A, \rho),\label{eq:frho1}\\
  f_\rho(N^\phi(A) \cap \{ \bm{r}^\phi_A > \rho   \}) &= \partial^v_+ B^\phi(A, \rho) \subseteq S^\phi(A, \rho)\cap \Unp^\phi(A) ,   \label{eq:frho2}
\end{align}
\begin{equation}\label{lem: unp and level sets eq1}
S^\phi(A, \rho) \cap \Unp^\phi(A) \subseteq \partial B^\phi(A, \rho), \quad   \mathcal{H}^{n}(S^\phi(A, \rho) \cap \Unp^\phi(A) \setminus \partial^v_+ B^\phi(A, \rho)) =0,
\end{equation}
\begin{equation}\label{lem: unp and level sets eq2}
    \partial^v_+B^\phi(A, \rho) \cap \Cut^\phi(A) = \varnothing \quad \textrm{and} \quad \mathcal{H}^n(S^\phi(A, \rho) \cap \Unp^\phi(A) \cap \Cut^\phi(A) ) =0.
\end{equation}
\end{Lemma}

\begin{proof} We start with \eqref{eq:frho1}. ``$\subseteq$'': Let $(x,\eta)\in N^\phi(A)\cap \{  \bm{r}^\phi_A \ge\rho\}$. Then $\bm{\delta}^\phi_A(x+t\eta)=t$ for $0\le t<\rho$, hence also for $t=\rho$ since $\delta^\phi_A$ is continuous. Thus, $f_\rho(x,\eta)\in S^\phi(A,\rho)$.

``$\supseteq$'': Let $z\in S^\phi(A,\rho)$, that is, $\delta^\phi_A(z)=\rho$ and there is some $x\in A$ with $\phi^*(z-x)=\rho$. Set $\eta=\rho^{-1}(z-x)$. Then $\bm{\xi}^\phi_A(x+t\eta)=x$ for $0\le t< \rho$ by \cite[Lemma 2.38 (g)]{MR4160798}. It follows that $(x,\eta)\in N^\phi(A)$,  $\bm{r}^\phi_A(x,\eta)\ge \rho$ and $f_\rho(x,\eta)=z$.

\medskip

Next we deal with \eqref{eq:frho2}. Let $(x,\eta)\in N^\phi(A) \cap \{  \bm{r}^\phi_A >\rho\}$. Then $f_\rho(x,\eta)\in S^\phi(A,\rho)$ by \eqref{eq:frho1}. Since  $\bm{r}^\phi_A(x,\eta)>\rho$ it follows from \eqref{eq : r and rho} that $\bm{\rho}^\phi_A(x+\rho\eta)>1$, and therefore also  $f_\rho(x,\eta)=x+\rho\eta\in\Unp^\phi(A)$. This yields $f_\rho(N^\phi(A) \cap \{ \bm{r}^\phi_A > \rho   \})  \subseteq S^\phi(A, \rho)\cap \Unp^\phi(A) $.

Now we show that $f_\rho(N^\phi(A) \cap \{ \bm{r}^\phi_A > \rho   \})= \partial^v_+ B^\phi(A, \rho) $ by proving two inclusions.
``$\subseteq$'':  Let $(x,\eta)\in N^\phi(A) \cap \{  \bm{r}^\phi_A >\rho\}$. In view of Remark \ref{remark: bdry tubular neigh} it suffices to show that $x+\rho\eta\in\partial B^\phi(A,\rho)$. For this, choose $\bar\rho\in(\rho,\bm{r}^\phi_A(x,\eta))$. Then we have already shown that $x+\bar\rho\eta\in S^\phi(A,\bar\rho)\cap\Unp^\phi(A)$ and clearly  $\bm{\xi}^\phi_A(x+\bar\rho\eta)=x$. We now claim that $\bm{\xi}^\phi_{B^\phi(A,\rho)}(x+\bar\rho\eta)=x+\rho\eta$, which will imply the required inclusion. Since
$$x+\rho\eta\in (x+\bar\rho\eta+(\bar\rho-\rho)B^\circ)\cap B^\phi(A,\rho)\qquad\text{and}\qquad (x+\bar\rho\eta+(\bar\rho-\rho)\Int(B^\circ))\cap B^\phi(A,\rho)=\varnothing,
$$
we get $x+\rho\eta\in \bm{\xi}^\phi_{B^\phi(A,\rho)}(x+\bar\rho\eta)$ and $\delta^\phi_{B^\phi(A,\rho)}(x+\bar\rho\eta)=\bar\rho-\rho>0$. Let $z\in \xi^\phi_{B^\phi(A,\rho)}(x+\bar\rho\eta)$ be arbitrarily chosen. Then
$$
z\in (x+\bar\rho\eta+(\bar\rho-\rho)\partial B^\circ)\cap (a+\rho\partial B^\circ)\qquad \text{for some }a\in A
$$
and
$$
(x+\bar\rho\eta+(\bar\rho-\rho)\Int( B^\circ))\cap (a+\rho\Int( B^\circ))=\varnothing.
$$
Hence $z=a+\rho\eta_1=x+\bar\rho\eta+(\bar\rho-\rho)\eta_2$ for some $\eta_1,\eta_2\in\partial B^\circ$. Since $B^\circ$ is smooth and strictly convex, it follows that $\eta_1=-\eta_2$. This implies that $a=x+\bar\rho\eta+\bar\rho\eta_2\in \bm{\xi}^\phi_A(x+\bar\rho\eta)$, thus $a=x$ and $\eta=-\eta_2=\eta_1$, which yields $z=x+\rho\eta$.

``$\supseteq$'': Let $z\in \partial^v_+B(A,\rho)$. Then there is some $x\in A$ with
$$
z\in x+\rho B^\circ\subseteq B^\phi(A,\rho)\qquad\text{and}\qquad z\in S^\phi(A,\rho).
$$
Furthermore there are $y\notin B^\phi(A,\rho)$ and $\epsilon>0$ such that
$$
z\in y+\epsilon B^\circ\qquad\text{and}\qquad (y+\epsilon\Int(B^\circ))\cap B^\phi(A,\rho)=\varnothing.
$$
This implies that $\bm{\delta}_A^\phi(y)=\rho+\epsilon$. Moreover,
$$
\rho+\epsilon\le \phi^*(y-x)\le \phi^*(y-z)+\phi^*(z-x)\le \epsilon+\rho,
$$
and hence $\phi^*(y-x)=\rho+\epsilon$, $\phi^*(y-z)=\epsilon$ and $\phi^*(z-x)=\rho$. Since $\phi^*$ is strictly convex, we also get
$y-z=s(z-x)$ with some $s>0$ and therefore $z=\frac{1}{1+s}y+\frac{s}{1+s}x$. We set $\eta=(\rho+\epsilon)^{-1}(y-x)\in\partial B^\circ$ and thus get
$\bm{\delta}^\phi_A(y)=\bm{\delta}^\phi_A(x+(\rho+\epsilon)\eta)=\rho+\epsilon$. From \cite[Lemma 2.38 (g)]{MR4160798} we now conclude that $(x,\eta)\in N^\phi(A)$ and $\bm{r}^\phi_A(x,\eta)\ge \rho+\epsilon>\rho$, which yields $z=f_\rho(x,\eta)\in f_\rho(N^\phi(A) \cap \{ \bm{r}^\phi_A > \rho   \})$.

\medskip

Now we turn to \eqref{lem: unp and level sets eq1}.
We fix an arbitrary $ \rho > 0$ and $ x \in S^\phi(A, \rho) \cap \Unp^\phi(A) $. Recall that  $ \bm{\delta}^\phi_A $ is semiconcave on $  \mathbf{R}^{n+1} \setminus B^\phi(A, s) $ for every $ s > 0 $ (see section \ref{sec:distnbperi}).  Since $ x \in S^\phi(A, \rho) \cap \Unp^\phi(A) $, the distance function  $\bm{\delta}^\phi_A $ is differentiable at $x$ with $0  \notin \{\nabla \bm{\delta}^\phi_A(x)\}=\partial \bm{\delta}^\phi_A(x) $, where we use that the generalized subgradient coincides with the subgradient from convex analysis for semiconcave functions (see \cite[Remark 1.4]{MR0816398} and the references to \cite{MR1058436}  given there). Hence we infer from \cite[Theorem 3.3 and Proposition 1.7]{MR0816398}  that there exist $ \epsilon, \delta > 0$, $ u \in \mathbf{S}^n$ and a lipschitzian semiconvex function $ f : x+u^\perp \rightarrow \mathbf{R}$ such that
\begin{equation*}
 \epi(f) \cap U_{\epsilon, \delta}(x, u) = B^\phi(A, \rho) \cap U_{\epsilon, \delta}(x, u) \quad \textrm{and} \quad    \graph(f) \cap U_{\epsilon, \delta}(x, u) = S^\phi(A, \rho) \cap U_{\epsilon, \delta}(x, u).
\end{equation*}
(One can easily see that $ u = \nabla \bm{\delta}^\phi_A(x)/ |\nabla \bm{\delta}^\phi_A(x)| $, but this is not relevant here). Since semiconvex functions are pointwise twice-differentiable almost everywhere by a classical theorem of Alexandrov, we conclude that for $ \mathcal{H}^n $ a.e.\ $ y \in S^\phi(A, \rho) \cap U_{\epsilon, \delta}(x,u) $  there exists an open ball $ B \subseteq U_{\epsilon, \delta}(x,u) $ such that $ B \cap B^\phi(A, \rho) = \varnothing$ and $ y \in \Clos(B) \cap B^\phi(A, \rho)$. Therefore it follows from Remark \ref{remark: bdry tubular neigh} that
\begin{equation*}
    \mathcal{H}^n(S^\phi(A, \rho) \cap U_{\epsilon, \delta}(x,u) \setminus \partial^v_+B^\phi(A, \rho)) =0.
\end{equation*}
Since $ x $ is arbitrarily chosen in $ S^\phi(A, \rho) \cap \Unp^\phi(A)$, we obtain the assertions in \eqref{lem: unp and level sets eq1}.

\medskip

The assertions in \eqref{lem: unp and level sets eq2} follow immediately from what we have already shown.
\end{proof}

For a closed set  $ A \subseteq \mathbf{R}^{n+1} $, a bounded Borel set $D\subset\R^{n+1}\times B^\circ$ and $\rho>0$, we define
$$
P^\phi_\rho(A,D)=\left\{x\in\Unp(A)\setminus A:\bm{\psi}^\phi_A(x)\in D,\bm{\delta}_A^\phi(x)\le \rho\right\}
$$
and
$$
V(A,D;\rho):=\mathcal{L}^{n+1}(P^\phi_\rho(A,D)).
$$
Recall that $\mathcal{W}^\phi=B^\circ$.  If   $ \xi : \mathbf{R}^{n+1} \setminus A \rightarrow \partial A $ and $ \nu : \mathbf{R}^{n+1} \setminus A \rightarrow \partial \mathcal{W}^\phi $ are two Borel functions such that
\begin{equation}\label{eq:Borelf}
\xi(x)\in \bm{\xi}^\phi_A(x)\qquad\text{and}\qquad \nu(x) \in \bm{\nu}^\phi_A(x),
\end{equation}
then we also define
$$
\overline{P}^\phi_\rho(A,D)=\left\{x\in \R^{n+1}\setminus A : (\xi(x),\nu(x))\in D,\bm{\delta}_A^\phi(x)\le \rho\right\}
$$
and
$$
\overline{V}(A,D;\rho):=\mathcal{L}^{n+1}(\overline{P}^\phi_\rho(A,D)).
$$
Clearly, we have $\overline{V}(A,D;\rho)={V}(A,D;\rho)$, and if
$D=A\times B^\circ$, then $\overline{V}(A,D;\rho)={V}(A,D;\rho)=\mathcal{L}^{n+1}(B^\phi(A,\rho)
\setminus A)$. Furthermore, note that if $x\in S^\phi(A,\rho)\cap  \Unp^\phi(A)\setminus A$, then $\xi(x)=x-\rho\cdot \nu(x)$.

\begin{Theorem}\label{theo: derivative of volume}
Let  $A \subseteq \mathbf{R}^{n+1} $ be a closed set, and let $ \phi $ be a uniformly convex $ \mathcal{C}^2 $-norm on $ \mathbf{R}^{n+1} $. Let  $D\subset\R^{n+1}\times \mathcal{W}^\phi$ be a bounded Borel set, and let $ \xi : \mathbf{R}^{n+1} \setminus A \rightarrow \partial A $ and $ \nu : \mathbf{R}^{n+1} \setminus A \rightarrow \partial \mathcal{W}^\phi $  be Borel functions as in \eqref{eq:Borelf}. Define a Borel function $ u: \mathbf{R}^{n+1} \setminus A \rightarrow \mathbf{S}^n $ by $u(x) = \bm{n}^\phi(\nu(x))$ for $x\in \R^{n+1}\setminus A$.
Then, for every $ \rho > 0 $  the right derivative $ V'_+(A,D;\rho) $ and the left derivative $ V'_-(A,D;\rho) $ of $ V $ at $ \rho $ exist and are given by
\begin{flalign}\label{theo: derivative of volume:1}
V'_+(A,D;\rho) & = \sum_{d=0}^n \int_{\widetilde{N}^\phi_d(A) \cap D\cap \{ \bm{r}^\phi_A> \rho\}}\phi(\bm{n}^\phi(\eta))J^\phi_A(a,\eta)\,\rho^{n-d}\prod_{i=1}^d(1 + \rho \kappa^\phi_{A,i}(a, \eta))\, d\mathcal{H}^n(a,\eta) \notag \\
& = \sum_{i=0}^n \rho^i \int_{N^\phi(A)\cap D \cap \{ \bm{r}^\phi_A > \rho\}}\phi(\bm{n}^\phi(\eta))J^\phi_A(a,\eta)\bm{H}^\phi_{A,i}(a,\eta)\, d\mathcal{H}^n(a,\eta) \notag \\
& = \int_{\partial^v_+ B^\phi(A,\rho)}\ind_D(\xi(x),\nu(x))  \phi(u(x))\, d\mathcal{H}^n(x)
\end{flalign}
and
\begin{flalign}\label{theo: derivative of volume:2}
V'_-(A,D;\rho)  & = \sum_{d=0}^n\int_{\widetilde{N}^\phi_d(A) \cap D \cap \{\bm{r}^\phi_A \geq \rho\}}\phi(\bm{n}^\phi(\eta)) J^\phi_A(a,\eta) \,\rho^{n-d}\prod_{i=1}^d(1 + \rho\kappa^\phi_{A,i}(a, \eta)) \, d\mathcal{H}^n(a,\eta) \notag \\
& = \sum_{i=0}^n \rho^i \int_{N^\phi(A)\cap D \cap \{ \bm{r}^\phi_A \geq  \rho\}}\phi(\bm{n}^\phi(\eta))J^\phi_A(a,\eta)\bm{H}^\phi_{A,i}(a,\eta)\, d\mathcal{H}^n(a,\eta) \notag \\
& =  V'_+(A,D; \rho)  +  \int_{S^\phi(A,\rho) \setminus \Unp^\phi(A)} \big(\bm{1}_D(x-\rho\nu(x),\nu(x))\\
&\qquad\qquad\qquad\qquad \qquad\qquad \quad+ \bm{1}_D(x+\rho\nu(x),-\nu(x)\big)  \phi(u(x))\, d\mathcal{H}^n(x).\nonumber
\end{flalign}
Consequently, $ V(A,D;\cdot) $ is differentiable at $ \rho>0 $ if
\begin{equation*}
	\int_{N^\phi(A)\cap D \cap \{ \bm{r}^\phi_A = \rho\}}\phi(\bm{n}^\phi(\eta))J^\phi_A(a,\eta)\bm{H}^\phi_{A,i}(a,\eta)\, d\mathcal{H}^n(a,\eta) = 0 \qquad \textrm{for $ i =0, \ldots , n $},
\end{equation*}
and this happens for all but countably many $ \rho \in (0,\infty)$.

Furthermore, $V(A,D;\cdot)$ is differentiable at $\rho>0$ if  and only if
$$ \int_{S^\phi(A,\rho) \setminus \Unp^\phi(A)} \big(\bm{1}_D(x-\rho\nu(x),\nu(x))+ \bm{1}_D(x+\rho\nu(x),-\nu(x)\big)  \phi(u(x))\, d\mathcal{H}^n(x).
$$
In particular, $V(A,(\R^{n+1})^2;\cdot)$ is differentiable at $\rho>0$ if and only if $\mathcal{H}^n(S^\phi(A,\rho) \setminus \Unp^\phi(A))=0$.
\end{Theorem}

\begin{proof}
Note that if $(a,\eta)\in N^\phi(A)$ and $0<t<\bm{r}^\phi_A(a,\eta)$, then $a+t\eta\in\Unp^\phi(A)$ and therefore $\xi(a+t\eta)=\bm{\xi}^\phi_A(a+t\eta)=a$ and $\nu(a+t\eta)=\bm{\nu}^\phi_A(a+t\eta)=\eta$. Therefore, choosing $ h(x) = \bm{1}_D(\xi(x), \nu(x))$ in Corollary \ref{theo: change of variable}, we get that
\begin{flalign*}
&\overline{V}(A,D;\rho) \notag \\
& \qquad =\sum_{d=0}^n\int_{\widetilde{N}^\phi_d(A) \cap D}\phi(\bm{n}^\phi(\eta)) J^\phi_A(a,\eta)\bigg(\int_{0}^\rho \bm{1}\{t <  \bm{r}^\phi_A(a, \eta)\}\,t^{n-d}\prod_{i=1}^d(1 + t \kappa^\phi_{A,i}(a, \eta)) \, dt\bigg)\, d\mathcal{H}^n(a,\eta)
\end{flalign*}
for  $ \rho > 0 $. Since the integrand is non-negative, we can exchange the order of integration by Fubini theorem \cite[Theorem2.6.2]{MR0257325} and obtain
\begin{flalign*}
&\overline{V}(A,D;\rho) \notag \\
& \qquad = \int_{0}^\rho \left(\sum_{d=0}^n\int_{\widetilde{N}^\phi_d(A) \cap D \cap \{\bm{r}^\phi_A > t\}}\phi(\bm{n}^\phi(\eta)) J^\phi_A(a,\eta) \,t^{n-d}\prod_{i=1}^d(1 + t \kappa^\phi_{A,i}(a, \eta)) \, d\mathcal{H}^n(a,\eta)\right)\, dt
\end{flalign*}
for $ \rho > 0 $. For $ \rho > 0 $, we define
$$
g(\rho)=\sum_{d=0}^n\int_{\widetilde{N}^\phi_d(A) \cap D \cap \{\bm{r}^\phi_A > \rho\}}\phi(\bm{n}^\phi(\eta)) J^\phi_A(a,\eta) \,\rho^{n-d}\prod_{i=1}^d(1 + \rho\kappa^\phi_{A,i}(a, \eta)) \, d\mathcal{H}^n(a,\eta)
$$
and $ f_\rho : N_\rho \rightarrow S^\phi(A, \rho) $ by $ f_\rho(a, \eta) = a + \rho \eta $ (where $ N_\rho $ is defined as in Theorem \ref{theo: Steiner closed}). Since
\begin{equation*}
    \bm{1}_{\{\bm{r}^\phi_A > s\}}(a, \eta) \nearrow \bm{1}_{\{\bm{r}^\phi_A > t\}}(a, \eta) \qquad \textrm{as $ s \searrow t$ for $(a, \eta) \in N^\phi(A)$,}
\end{equation*}
an application of the dominated convergence theorem gives that $ \lim_{s \downarrow t} g(s) = g(t)$ for $ t > 0 $. Using the formula \eqref{theo: Steiner closed: formula 0} in the proof of Theorem \ref{theo: Steiner closed}, we get
\begin{equation}\label{eq:rightder}
g(t) =  \int_{N^\phi(A) \cap D \cap \{\bm{r}^\phi_A > t\}}  \phi(\bm{n}^\phi(\eta))  \,\ap J_n^{N_t} f_t(a, \eta)\, d\mathcal{H}^n(a,\eta) < \infty
\end{equation}
for  $ t > 0 $. Now, noting that $  \overline{V}(A,D;\rho)=\int_0^\rho g(s)\, ds $, we readily obtain that
\begin{equation}
\overline{V}'_+(A,D;\rho)= g(\rho) \qquad \textrm{for $ \rho > 0$.}
\end{equation}
 By Lemma \ref{lem: unp and level sets},  we have $ f_\rho(N^\phi(A) \cap \{ \bm{r}^\phi_A > \rho   \}) = \partial^v_+ B^\phi(A, \rho) \subseteq \Unp^\phi(A) \cap S^\phi(A, \rho) $; moreover,
\begin{equation*}
	\bm{q}(f_\rho^{-1}(\{x\})) = \{\nu(x)   \}=\{\bm{\nu}^\phi_A(x)\} \qquad \textrm{for $ x \in \partial^v_+ B^\phi(A, \rho) $}.
\end{equation*}
Noting \eqref{eq:rightder}, we can apply the coarea formula to conclude that
\begin{equation*}
g(\rho)= \int_{\partial^v_+ B^\phi(A,\rho)}\bm{1}_D(\xi(x),\nu(x)) \phi(u(x))\, d\mathcal{H}^n(x) \qquad \textrm{for  $ \rho > 0$}.
\end{equation*}

Now we deal with the left derivative. Since
$$
\bm{1}_{\{\bm{r}^\phi_A > s\}}\searrow \bm{1}_{\{\bm{r}^\phi_A \geq t\}}\quad \text{as }s\nearrow t,
$$
it follows again from the dominated convergence theorem that
$$
\lim_{s\uparrow t}g(s)=\sum_{d=0}^n\int_{\widetilde{N}^\phi_d(A) \cap D \cap \{\bm{r}^\phi_A \geq t\}}\phi(\bm{n}^\phi(\eta)) J^\phi_A(a,\eta) \,t^{n-d}\prod_{i=1}^d(1 + t\kappa^\phi_{A,i}(a, \eta)) \, d\mathcal{H}^n(a,\eta)
$$
for  $ t > 0 $. Hence we deduce again using the formula \eqref{theo: Steiner closed: formula 0} in the proof of Theorem \ref{theo: Steiner closed} that
\begin{align*}
\overline{V}'_-(A,D;\rho)&=\sum_{d=0}^n\int_{\widetilde{N}^\phi_d(A) \cap D \cap \{\bm{r}^\phi_A \geq \rho\}}\phi(\bm{n}^\phi(\eta)) J^\phi_A(a,\eta) \,\rho^{n-d}\prod_{i=1}^d(1 + \rho\kappa^\phi_{A,i}(a, \eta)) \, d\mathcal{H}^n(a,\eta)\\
&=\int_{N^\phi(A)\cap D \cap \{\bm{r}^\phi_A\ge\rho\}}\phi(\bm{n}^\phi(\eta))\ap J^{N_\rho}_nf_\rho(a,u)\, d\mathcal{H}^n(a,\eta)\\
&=\int_{S^\phi(A, \rho)} \int_{f_\rho^{-1}(\{x\})}\bm{1}_D(a,\eta)\phi(\bm{n}^\phi(\eta))\, d\mathcal{H}^0(a,\eta)\, d\mathcal{H}^n(x) ,
\end{align*}
where \eqref{eq:frho1} and the coarea formula have been used in the last step.

If $x\in S^\phi(A,\rho)\cap\Unp^\phi(A)$, then there is a unique $(a,\eta)\in N^\phi(A)\cap \{\bm{r}^\phi_A\ge\rho\}$ with $f_\rho(a,\eta)=x$, and we have $\xi(x)=\bm{\xi}^\phi_A(x)=a$, $\nu(x)=\bm{\nu}^\phi_A(x)=\eta$ and $u(x)=\bm{n}^\phi(\eta)$. On the other hand, if $ x \in S^\phi(A, \rho) \setminus \Unp^\phi(A) $, then $	\mathcal{H}^0\big(\bm{q}(f_\rho^{-1}(\{x\}))\big) > 1 $, and in addition we have
\begin{equation}\label{eq:late1}
\bm{q}(f_\rho^{-1}(\{x\})) \subseteq N^\phi(S^\phi(A, \rho), x) \qquad \textrm{for   $ x \in S^\phi(A, \rho) $.}
\end{equation}
Consequently, we deduce from Lemma \ref{lem: exterior normal basic properties closed} that
\begin{equation}\label{eq:late2}
\bm{q}(f_\rho^{-1}(\{x\})) = \{ \nu(x), -\nu(x)  \} \qquad \textrm{for $ \mathcal{H}^n $ a.e.\  $ x \in S^\phi(A, \rho) \setminus \Unp^\phi(A) $}.
\end{equation}
Thus we conclude that
\begin{align*}
\overline{V}'_-(A,D;\rho)&=\int_{S^\phi(A,\rho)\cap\Unp^\phi(A)}\bm{1}_D(\xi(x),\nu(x))\phi(u(x))\, d\mathcal{H}^n(x)\\
&\qquad +
\int_{S^\phi(A,\rho)\setminus\Unp^\phi(A)} \big(\bm{1}_D(x-\rho\nu(x),\nu(x))\phi(\bm{n}^\phi(\nu(x))\\
&\qquad\qquad\qquad\qquad\qquad\qquad +
\bm{1}_D(x+\rho\nu(x),-\nu(x))\phi(\bm{n}^\phi(-\nu(x))\big)\, d\mathcal{H}^n(x).
\end{align*}
It follows from \eqref{eq:frho2} and \eqref{lem: unp and level sets eq1} that the first summand on the right side of the preceding equation equals $\overline{V}'_+(A,D;\rho)$. We now obtain the asserted representation of $\overline{V}'_-(A,D;\rho)$ by observing that $\bm{n}^\phi(\pm\nu(x))=\pm u(x)$.

To check the second equality both in \eqref{theo: derivative of volume:1} and in \eqref{theo: derivative of volume:2} we  notice that
\begin{equation*}
    \sum_{d=0}^n \rho^{n-d}\prod_{i=1}^d(1 + \rho\kappa^\phi_{A,i}(a, \eta))\, \bm{1}_{\widetilde{N}^\phi_d(A)}(a, \eta) = \sum_{i=0}^n \rho^i \bm{H}^\phi_{A,i}(a, \eta) \qquad \textrm{for $ (a, \eta) \in \widetilde{N}^\phi(A)$}
\end{equation*}
and we use the integrability condition in \eqref{theo: Steiner closed 3.1} proved in Theorem \ref{theo: Steiner closed} to interchange summation and integral.

Finally,  for   $ \epsilon >0 $ and $ s > 0 $ the set
\begin{equation*}
    I_{\epsilon, s} = \bigg\{  t \in[ s,\infty) :   \int_{N^\phi(A) \cap \{ \bm{r}^\phi_A = t\}}\phi(\bm{n}^\phi(\eta))J^\phi_A(a,\eta)|\bm{H}^\phi_{A,i}(a,\eta)|\, d\mathcal{H}^n(a,u)  \geq \epsilon \bigg\}
\end{equation*}
is finite by the integrability property in \eqref{theo: Steiner closed 3.1}. This readily implies that
\begin{equation*}
    \int_{N^\phi(A) \cap \{ \bm{r}^\phi_A = \rho\}}\phi(\bm{n}^\phi(\eta))J^\phi_A(a,\eta)\bm{H}^\phi_{A,i}(a,\eta)\, d\mathcal{H}^n(a,\eta) \neq 0
\end{equation*}
for at most countably many different values of $ \rho$.
\end{proof}

\begin{Remark}\label{rem:finsec3}
The derivative of the volume function has been the subject of several investigations; see \cite{MR442202}, \cite{MR2031455}, \cite{MR2218870}, \cite{MR2865426} and \cite{MR4329249}.


In order to compare our characterization of the differentiability of the parallel volume function of a compact set (in the non-localized setting) with the one in \cite{MR4329249}, we introduce the following notation. If $F\subseteq\R^{n+1} $ is a measurable set, then we write $F^1$ for the set of all $x\in\R^{n+1}$ for which the $(n+1)$-dimensional density of $F$ at $x$ equals $1$.  Let  $A\subset\R^{n+1}$ be compact and $\rho>0$. Then it follows from \eqref{eq:late1} and \eqref{eq:late2} that
$$
\mathcal{H}^n\left([S^\phi(A,\rho)\setminus\Unp^\phi(A)]\setminus B^\phi(A,\rho)^1\right)=0.
$$
On the other hand, we also have
\begin{align*}
\mathcal{H}^n\left(S^\phi(A,\rho)\cap B^\phi(A,\rho)^1\cap   \Unp^\phi(A)
\right)
 =\mathcal{H}^n\left(S^\phi(A,\rho)\cap  \Unp^\phi(A)\setminus  \partial^v_+B^\phi(A,\rho)\right)=0,
\end{align*}
where we use \eqref{lem: unp and level sets eq1}  and the fact that
$B^\phi(A,\rho)^1\cap  \partial^v_+B^\phi(A,\rho)=\varnothing$. Thus we see that if $\Delta$ denotes the symmetric difference operator for subsets of $\R^{n+1}$, then
$$
\mathcal{H}^n\left([S^\phi(A,\rho)\cap B^\phi(A,\rho)^1]\Delta
[S^\phi(A,\rho)\setminus\Unp^\phi(A)]
\right)=0.
$$

More generally, our method allows us to  study  the differentiation of {\em local}  parallel volumes in an anisotropic setting. Such a localization was suggested in an isotropic framework by results in \cite{MR3939268} (note however that \cite[Lemma 2.9]{MR3939268} is not correct, which  affects the proof of \cite[Proposition 2.10]{MR3939268} for instance).
\end{Remark}

\begin{Remark}
By the arguments in the proof of Theorem \ref{theo: derivative of volume} it also follows that $V_+'(A,D;\cdot)$ is continuous from the right and $V_-'(A,D;\cdot)$ is continuous from the left.
\end{Remark}

\section{Alexandrov points of sets of positive reach}\label{sec: Alexandrov points}

Throughout this section, $ \phi $ is a uniformly convex $\mathcal{C}^2 $-norm on $ \mathbf{R}^{n+1} $ with corresponding gauge body $B$. For $\rho>0$ we set $B^\phi(\rho)=\rho B^\circ=\rho \mathcal{W}^\phi$ and $B^\phi= B^\circ=\mathcal{W}^\phi$. Recall that if $\phi$ is the Euclidean norm, then the upper index $\phi$ is omitted.

\paragraph{Notation.} For $a\in \mathbf{R}^{n+1}$, $u\in\mathbf{S}^n$ and $\varepsilon,\delta>0$ we define
$$
u^\perp=\{x\in \mathbf{R}^{n+1}:x\bullet u=0\},
$$
$$
U_\varepsilon(a,u)=\{x\in a+u^\perp:|x-a|<\varepsilon\},\quad U_{\varepsilon,\delta}(a,u) = \{ x + \lambda u : x \in U_\epsilon(a,u),\; \lambda \in  (-\delta, \delta)\}.
$$
If $f:a + u^\perp\to\R$ is a function (and $u$ is a given orienting normal vector), we write
\begin{equation}\label{eq: epigraph}
\graph(f)=\{x-f(x)u:x\in a + u^\perp\}\qquad\text{and}\qquad  \epi(f)=\{x-su:s\ge f(x),\, x \in a + u^\perp \}
\end{equation}
for the graph of $f$ and the epigraph of $f$ (with respect to $u$), respectively.

\subsection{Positive $\phi$-reach}

\begin{Definition}\label{def: reach}
Let $ A \subseteq \mathbf{R}^{n+1} $ be a closed set. The \emph{$ \phi $-reach}   of $ A $ is defined as the non-negative number
\begin{equation*}
\reach^\phi(A)=\sup\{\rho\ge 0:B^\phi(A, \rho) \subseteq \Unp^\phi(A)\}.
\end{equation*}
The set $A$ is said to have \emph{positive $ \phi $-reach} if $\reach^\phi(A)>0$.
\end{Definition}

In the following, we say that a convex body $L\subset \mathbf{R}^{n+1}$ slides freely inside a convex body
$K\subset \mathbf{R}^{n+1}$ if for each $x\in\partial K$ there is some $t\in \mathbf{R}^{n+1}$ such that $x\in L+t\subseteq K$ (see \cite[Section 3.2]{MR3155183}). In particular, $L$ slides freely inside $K$ if and only if $L$ is a summand of $K$. It follows from \cite[Theorem 3.2.12]{MR3155183} that if $L,K\subset \mathbf{R}^{n+1}$ are convex bodies of class $C^2_+$, then there is some $\rho>0$ (depending on $L,K$) such that $\rho L$ slides freely inside $K$.

\begin{Lemma}\label{lem:comparereach}
Let $A\subset \mathbf{R}^{n+1}$ be a closed set. Let $\phi,\bar\phi$ be uniformly convex $C^2$ norms on $\mathbf{R}^{n+1}$ with corresponding gauge bodies $B,\bar B$. Let $\rho>0$ be such that $B^\phi(\rho)=\rho B^\circ$ slides freely inside $B^{\bar\phi}=\bar B^\circ$. If $\reach^{\bar\phi}(A)>r$, then $\reach^{\phi}(A)>\rho r$.
\end{Lemma}

\begin{proof}
Assume that $\reach^{\bar\phi}(A)>r$ and $B^\phi(\rho)=\rho B^\circ$ slides freely inside $B^{\bar\phi}=\bar B^\circ$. Then $s B^\circ$ slides freely inside $r\bar B^\circ$ if $0<s\le \rho r$. Aiming at a contradiction, we assume that there is some $z\in \mathbf{R}^{n+1}\setminus A$ and there are $x_1,x_2\in\partial A$ with $x_1\neq x_2$ and such that
$\{x_1,x_2\}\subseteq (z+s_0B^\circ)\cap A$ and $\Int(z+s_0B^\circ)\cap A=\emptyset$ for some $s_0\in (0,\rho r]$. Then $x_i-z\in\partial (s_0 B^\circ)$ and   $N(s_0B^\circ,x_i-z)=\{v_i\}$ for some unit vector $v_i$, for $i=1,2$ (since $B^\circ$ is smooth). In particular, we have $x_i-z=s_0\nabla h_{B^\circ}(v_i)$ for $i=1,2$. Since $B^\circ$ is of class $C^2_+$ (and hence a Euclidean ball slides freely inside $B^\circ$), it follows that $-v_i\in N
(A,x_i)$ for $i=1,2$. Since $\reach^{\bar\phi}(A)>r$, we conclude that
\begin{equation}\label{eq:pr1}
(x_i-r\nabla h_{\bar B^\circ}(v_i)+r\bar B^\circ)\cap A=\{x_i\}.
\end{equation}
Using first that  $x_i-s_0\nabla h_{ B^\circ}(v_i)=z$ for $i=1,2$ and then that $s_0B^\circ$ slides freely inside $r\bar B^\circ$, we get
$$
\{x_1,x_2\}\subset z+s_0B^\circ\subseteq x_i-s_0 \nabla h_{B^\circ}(v_i)+s_0B^\circ\subseteq x_i-r\nabla h_{\bar B^\circ}(v_i)+r\bar B^\circ.
$$
But then \eqref{eq:pr1} yields
$$
x_2\in (x_1-r\nabla h_{\bar B^\circ}(v_1)+r\bar B^\circ)\cap A=\{x_1\},
$$
a contradiction.
\end{proof}

\begin{Remark}
Assume that $ \phi,\psi $ are any two uniformly convex $ \mathcal{C}^2 $-norms. Then $ \reach^\phi(A) > 0 $ if and only if $ \reach^\psi(A) > 0 $. We say that a closed set $ A \subseteq \mathbf{R}^{n+1} $ is a \emph{set of positive reach} if  $ \reach(A) > 0 $, that is, if $A$ has positive reach with respect to the Euclidean norm. Hence, a set has positive reach if and only if it has positive $\phi$-reach for some (and then for any) uniformly convex $\mathcal{C}^2 $-norm $\phi$ on $ \mathbf{R}^{n+1} $.
\end{Remark}

\begin{Remark}\label{rem: sets of positive reach}
The class of sets of positive reach as defined here is precisely the class of sets of positive reach introduced in \cite[Definition 4.1]{MR0110078}. We recall from \cite[Theorem 4.8 (12)]{MR0110078} that if $A\subseteq \R^{n+1}$ is a set with positive reach, then
\begin{equation*}
	N(A) = \{(a,u) : a \in A, u \in \Nor(A,a), |u|=1   \}.
\end{equation*}
Moreover it follows from Lemma \ref{theo: projectionLipschitz} that $ \bm{\xi}^\phi_A| \{ x \in \mathbf{R}^{n+1} : 0 < \bm{\delta}^\phi_A(x) < \reach^\phi(A)  \} $ is a locally Lipschitz map and $ \bm{\psi}|S^\phi(A,r) $ is a locally bilipschitz homeomorphism onto $ N^\phi(A) $ for $  r \in (0, \reach^\phi(A) )$. In particular,  $ S^\phi(A,r) $ is a $ \mathcal{C}^{1,1} $ closed hypersurface for $  r \in (0, \reach^\phi(A) )$ and $ N^\phi(A) $ is a closed Lipschitz $ n $-dimensional submanifold of $ \mathbf{R}^{n+1}\times \mathbf{R}^{n+1} $.
\end{Remark}

In the next section we exploit the fundamental connection between sets of positive reach and semiconvex functions. We will refer to the fundamental work \cite{MR0816398} (see also the cited references therein). Here we recall the notion of a semiconvex function.

\begin{Definition}
Suppose $ U \subseteq \mathbf{R}^k $ is open and convex. Then a function $ f : U \rightarrow \mathbf{R} $ is called \emph{semiconvex} if there exists $ 0 \leq \kappa < \infty$ such that the function  $ U \ni y \mapsto \kappa |y|^2 + f(y)$ is convex. We denote the \emph{generalized Clarke gradient} of $f $ at $ a \in U$ by $ \partial f(a)$ (see \cite{MR0816398} or \cite{MR1058436}).
\end{Definition}

\subsection{Alexandrov points and pointwise curvatures}

The first main result of this section (see Theorem \ref{lem: graphical representability}) states that for a set $ A $ of positive reach  the set $\bm{p}(\widetilde{N}^\phi(A)) $ (which is the set of curvature points of $ A $) can be partitioned as follows
\begin{equation}\label{eq: partition}
    \bm{p}(\widetilde{N}^\phi(A))  = \bm{p}(\widetilde{N}^\phi(A) \setminus \widetilde{N}^\phi_n(A)) \cup \bm{p}(\widetilde{N}^\phi_n(A))
\end{equation}
For a convex body this partition is well known, as  $ \bm{p}(\widetilde{N}^\phi_n(A))$ is the set of normal boundary points (also known as Alexandrov points); see \cite[Notes for Sections 1.5 and 2.6]{MR3155183},  \cite[Lemma 3.1]{MR1654685} and the literature cited there. The second goal of this section is to extend the notion of an Alexandrov point to sets of positive reach. This notion will play a central role in most of the subsequent rigidity statements;  see for instance Theorem \ref{theo: heintze karcher positive reach} and Corollary \ref{theo: lower bound} in the next section.

\begin{Theorem}\label{lem: graphical representability}
Let $ A \subseteq \mathbf{R}^{n+1}$ be a closed set with $ \reach(A) >0$.

If $(a,\eta) \in \widetilde{N}^\phi_n(A)$, $ u \in N(A,a)$ with $ \nabla \phi(u) = \eta $ and $ \nu \in \mathbf{S}^n \setminus u^\perp $, then the following statements hold.
\begin{enumerate}[{\rm (a)}]
    \item\label{lem: graphical representability 1} There exist $ \epsilon(\nu) > 0$ and a function $ f_\nu : U_{\epsilon(\nu)}(a,\nu) \rightarrow \mathbf{R}$ which is differentiable at $ a $ and such that $  \graph(f_\nu) \subseteq \partial A $.
 \item\label{lem: graphical representability 2} There exists a map $ \widetilde{\eta} : U_{\epsilon(u)}(a,u) \rightarrow \partial \mathcal{W}^\phi$ such that $\widetilde{\eta}(b) \in N^\phi(A, b + f_u(b)u)$ for every $ b \in U_{\epsilon(u)}(a,u) $, $ \widetilde{\eta}(a) = \eta $, $\widetilde{\eta}$ is differentiable at $ a $ and the eigenvalues $ \lambda_1 \leq \ldots \leq \lambda_n $ of $ \Der \widetilde{\eta}(a)$ satisfy
\begin{equation*}
    \lambda_i = \kappa^\phi_{A,i}(a,\eta) \qquad \textrm{for $ i = 1, \ldots , n$.}
\end{equation*}
\end{enumerate}

Moreover,  $\bm{p}(\widetilde{N}^\phi_n(A)) \cap \bm{p}(\widetilde{N}^\phi(A) \setminus \widetilde{N}^\phi_n(A)) = \varnothing $.
\end{Theorem}

\begin{proof} \ref{lem: graphical representability 1}
Suppose $ a =0 $ and $ 0 < r < \reach^\phi(A) $. Fix $ \nu \in \mathbf{S}^n \setminus u^\perp $ and let $ \pi_\nu $ be the orthogonal projection onto $ \nu^\perp$. We consider the lipschitz function $ \psi_\nu  : S^\phi(A,r) \rightarrow  \nu^\perp $, $\pi_\nu \circ \bm{\xi}^\phi_A $. If $ H $ is the halfspace orthogonal to $ u $ which does not contain $u$, we notice the inclusions
\begin{equation*}
    \Der \bm{\xi}^\phi_A(r\eta) [\Tan(S^\phi(A,r), r\eta)] \subseteq \Tan(\bm{\xi}^\phi_A(S^\phi(A,r)), 0) \subseteq \Tan(A,0) \subseteq H.
\end{equation*}
Since $ \Der \bm{\xi}^\phi_A(r\eta)|\Tan(S^\phi(A,r), r\eta)  $ is injective, we infer that
\begin{equation*}
    \Tan(S^\phi(A,r), r\eta) =  \Der \bm{\xi}^\phi_A(r\eta) [\Tan(S^\phi(A,r), r\eta)] = u^\perp
\end{equation*}
and $ \Der \psi_\nu(r\eta) $ is an isomorphism onto $ \nu^\perp $.  Consequently, we can apply the (right-) inverse function theorem in \cite{MR1769420} to infer the existence of a constant $ \epsilon(\nu) > 0 $ and a map $  \varphi_\nu : U_{\epsilon(\nu)}(0,\nu) \rightarrow S^\phi(A,r) $ such that $ \varphi_\nu(0) = r\eta $, $ \varphi_\nu $ is differentiable at $ 0 $ with $ \Der \varphi_\nu(0) = \Der \psi_\nu(r\eta)^{-1} = \Der \bm{\xi}^\phi_A(r\eta)^{-1} \circ (\pi_\nu|u^\perp)^{-1}$ and $\psi_\nu(\varphi_\nu(b)) = b $ for every $ b \in U_{\epsilon(\nu)}(a,\nu)$. We define $ f_\nu: U_{\epsilon(\nu)}(a,\nu) \rightarrow \mathbf{R}$ by
\begin{equation*}
    f_\nu(b) = \bm{\xi}^\phi_A(\varphi_\nu(b)) \bullet \nu \qquad \textrm{for $ b \in U_{\epsilon(\nu)}(a,\nu)$}
\end{equation*}
and we readily check that $ f_\nu(0) =0$, $ \Der f_\nu(0) = (\pi_\nu|u^\perp)^{-1} \bullet \nu $ and $ b + f_\nu(b) \nu = \bm{\xi}^\phi_A(\varphi_\nu(b)) \in \partial A $ for every $ b \in U_{\epsilon(\nu)}(a,\nu)$.

\medskip

\ref{lem: graphical representability 2}
We define $ \widetilde{\eta}: U_{\epsilon(u)}(a,u) \rightarrow \partial \mathcal{W}^\phi$ by
\begin{equation*}
    \widetilde{\eta}(b) =\frac{1}{r}(\varphi_u(b) - \bm{\xi}^\phi_A(\varphi_u(b))) \qquad \textrm{for $ b \in U_{\epsilon(u)}(a,u)$}
\end{equation*}
and notice that  $ \widetilde{\eta}(0) = \eta $ and $ \widetilde{\eta}(b) \in N^\phi(A, b + f_u(b) u)$ for every $ b \in U_{\epsilon(u)}(a,u)$ and
\begin{equation*}
    \Der \widetilde{\eta}(0) = \frac{1}{r}(\Der \bm{\xi}^\phi_A(r\eta)^{-1}|u^\perp -\bm{1}_{u^\perp}).
\end{equation*}
Noting that $ \{  (1 - r \rchi^\phi_{A,i}(r\eta))^{-1}: i = 1, \ldots , n \}$ are the eigenvalues of $\Der \bm{\xi}^\phi_A(r\eta)^{-1}|u^\perp   $, we conclude that the eigenvalues of $ \Der \widetilde{\eta}(0)$ satisfy  the equation $\lambda_i = \kappa^\phi_{A,i}(a,\eta) $ for  $ i \in \{1, \ldots, n \}$.

\medskip

Finally, since $ \nabla \phi(u) \bullet u = \phi(u) \neq 0$ for $ u \in \mathbf{S}^n$ (see \ref{eq: normal wulff shape and gradient phi}), the remaining assertion follows from part \ref{lem: graphical representability 1} and Lemma \ref{lem suff condition for finite curvature}.
\end{proof}

A more refined description of a set of positive reach around the points of the viscosity boundary is given by the following result.

\begin{Theorem}\label{prop:p1}
Let $A\subset\R^{n+1}$ be a closed set with $\reach(A)>0$ and $a\in \partial^v A$. Assume that $N(A,a)=\{u\}$ for some $u\in\bS^n$. Then the following statements hold.
\begin{enumerate}[{\rm (a)}]
\item There are $\varepsilon,\delta>0$ and  a semiconvex lipschitz function $f:a+u^\perp\to\R$ such that $ f(a) =0$, $ f $ is differentiable at $ a $ with $ \Der f(a) =0$ and (with respect to the orienting normal vector $u$)
\begin{equation}\label{prop:p1:eq1}
\graph(f)\cap U_{\varepsilon,\delta}(a,u)=\partial A \cap U_{\varepsilon,\delta}(a,u),\quad \epi(f)\cap U_{\varepsilon,\delta}(a,u)=  A \cap U_{\varepsilon,\delta}(a,u).
\end{equation}
\item  $(a, \nabla \phi(u)) \in \widetilde{N}^\phi_n(A)$ if and only if $ f $ is pointwise twice differentiable at $ a $. In this case, every map $ \nu : a + u^\perp \rightarrow \mathbf{S}^n$ such that $ \nu(b) \in N(\epi(f), b-f(b) u)$ for every $ b \in a + u^\perp$ is differentiable at $ a $ and satisfies
\begin{equation*}
    \Der^2f(a)(\tau_1, \Der(\nabla \phi)(u)(\tau_2)) = \Der (\nabla \phi \circ \nu)(a)(\tau_1) \bullet \tau_2 \qquad \textrm{for $ \tau_1 , \tau_2 \in u^\perp $.}
\end{equation*}
\end{enumerate}

\end{Theorem}

\begin{proof}
Let $0<r<\reach(A)$. We start with (a).
By \cite[Theorem 4.8 (12)]{MR0257325} we have $\Tan (A,a)= \{v\in \R^{n+1}:v\bullet u\le 0\}$. Then $\mathcal{C}=\{v\in\R^{n+1}:v\bullet u\le -\frac{1}{2} |v|\}$ defines a closed convex cone such that $\mathcal{C}\subseteq\Int(\Tan(A,a))\cup\{0\}$. Lemma 3.5 in \cite{MR3683461} shows that there is some $s\in (0,r/2)$ such that
\begin{equation}\label{eq:Nr1}
(a+\mathcal{C})\cap B(a,s)\subseteq A.
\end{equation}
 We consider the set
 $$
 M_u=\{z\in\partial A:|u-n_z|<1/4 \text{ for some }n_z\in\Nor(A,z)\cap\bS^{n}\}.
 $$
It follows from \cite[Proposition 3.1 (iii)]{MR3683461} that there is some $s'\in (0,s)$ such that if $z\in\partial A\cap B(a,s')$, then $|u-n_z|<1/4$ whenever $n_z\in\Nor(A,z)\cap \bS^n$. The proof of \cite[Theorem 5.9]{MR3683461} then shows that
$\partial A\cap B(a,s')\subset M_u\cap B(a,s')$ is contained in the graph of a lipschitz semiconvex function $f:a+u^\perp\to\R$ with $f(a)=0$, and hence there are $\varepsilon,\delta>0$ such that (with respect to the orienting normal vector $u$)
\begin{equation}\label{eq:Nr2}
 \partial A \cap U_{\varepsilon,\delta}(a,u) \subseteq \graph(f)\cap U_{\varepsilon,\delta}(a,u).
\end{equation}
In order to verify the remaining assertions, we use \cite[Proposition 3.3]{MR3683461} (see \cite[Theorem 4.18]{MR0110078}). Thus we get
\begin{align}
A\cap B(a,r)&\subseteq \left\{a+x+tu: t \leq \tfrac{1}{2r}(|x|^2+t^2),x\in u^\perp,t\in\R, |x+tu|\le r\right\}\nonumber\\
&\subseteq \left\{a+x+tu:-r\le t\le \tfrac{1}{r}|x|^2,x\in u^\perp,t\in\R, |x|\le r\right\}.\label{eq:Nr3}
\end{align}
At this point the equalities in \eqref{prop:p1:eq1} follows from elementary topology by combining \eqref{eq:Nr1}--\eqref{eq:Nr3}. Moreover, since $ \Nor(\epi(f), a) =\Nor(A,a) = \{tu : t \geq 0\}$, it follows from \cite[Remarks 1.4, Lemma 2.9]{MR0816398} that the generalized Clarke gradient of $ f $ at $ a $ contains only $ 0 $ and $ f $ is differentiable at $ a $ with $ \Der f(a) =0$.

\medskip

(b) Let $\mathcal{N}$  be the family of all maps $ \nu : a + u^\perp \rightarrow \mathbf{S}^n $ such that $ \nu(b) \in N(\epi(f), b - f(b)u)$ for every $ b \in a + u^\perp $, where $f$ is chosen according to \eqref{prop:p1:eq1} with respect to the orienting normal vector $u$. For every $ \nu \in  \mathcal{N}  $  we define the function $ g_\nu : U_\epsilon(a,u) \rightarrow u^\perp$ by
\begin{equation*}
    g_\nu(b) = \frac{\nu(b) - (\nu(b) \bullet u)u}{\nu(b) \bullet u} \quad \textrm{for $ b \in U_\epsilon(a,u)$.}
\end{equation*}
(Notice $ \nu(b) \bullet u > 0 $ for every $ b \in U_\epsilon(a,u) $ since $ f $ is lipschitz). Employing \cite[Lemma 2.9]{MR0816398} we conclude that $ g_\nu(b) \in \partial f(b) $  for every $ b \in U_\epsilon(a,u)$. In particular, if we assume that $(a, \nabla \phi(u)) \in \widetilde{N}^\phi_n(A)$, then it follows from Theorem \ref{lem: graphical representability} that there exists at least one map $ \nu_0 \in \mathcal{N}$ that is differentiable at $ a $. Therefore $ g_{\nu_0} $ is differentiable at $ a $ and the classical theory of subgradients for (semi)convex functions (see \cite{MR0534228} or \cite[Lemma 2.39]{kolasinski2021regularity})  implies that $ f $ is twice differentiable at $ a $.

Suppose $ f $ is twice differentiable at $ a $, $ a =0 $ and $ \nu \in \mathcal{N}$. Then the aforementioned theory of subgradients for (semi)convex functions implies that $ \nu $ is differentiable at $  0$ with $ \Der \nu(0) \bullet \nu(0)=0$ and
\begin{equation}\label{prop:p1 eq2}
 \Der \nu(0)(\tau_1) \bullet \tau_2 =  \Der g_\nu(0)(\tau_1) \bullet \tau_2 = \Der^2 f(0)(\tau_1, \tau_2)
\end{equation}
for $ \tau_1, \tau_2 \in u^\perp$ and for every $ \nu \in  \mathcal{N}$. It follows that
\begin{equation*}
    \Der(\nabla \phi \circ \nu)(0)(\tau_1) \bullet \tau_2 = \Der \nu(0)(\tau_1) \bullet \Der(\nabla \phi)(u)(\tau_2) =  \Der^2f(0)(\tau_1, \Der(\nabla \phi)(u)(\tau_2))
\end{equation*}
for $ \tau_1, \tau_2 \in u^\perp$. Now we choose $ r , \epsilon' > 0 $ so that $ 0 < \epsilon' < r < r + \epsilon' < \reach(A)$ and we define $ J = ( r-\epsilon' , r+ \epsilon') $, $ \eta = \nabla \phi \circ \nu $ and the function $ F : U_\epsilon(0,u) \times J \rightarrow \mathbf{R}^{n+1}$ by \begin{equation*}
    F(b,t) = b - f(b) u + t \eta(b) \qquad \textrm{for $ b \in U_\epsilon(0,u) \times J$.}
\end{equation*}
Then $ F$ is differentiable at $ (0,r) $ and, noting that $ U(\reach^\phi(A)\eta(0), \reach^\phi(A)) \cap \epi(f) \cap U_{\epsilon, \delta}(0,u) = \varnothing$, we conclude from  \eqref{prop:p1 eq2}, employing the same comparison-of-curvatures argument as the one used in the proof of Lemma \ref{lem suff condition for finite curvature}, that all eigenvalues of $\Der \eta(0) $ are smaller or equal than $ \reach^\phi(A)^{-1}$. Consequently $  \Der F(a,r) $ is invertible. Let $ \pi :  \mathbf{R}^{n+1} \rightarrow u^\perp $ be the orthogonal projection onto $  u^\perp $.  Since the function $ G : (\bm{\delta}_A^{\phi})^{-1}(J) \rightarrow u^\perp \times \mathbf{R}$, defined by $ G(x) = (\pi(\bm{\xi}_A^\phi(x)), \bm{\delta}_A^\phi(x))$ for $ x \in (\bm{\delta}_A^{\phi})^{-1}(J) $, is Lipschitzian and satisfies
\begin{equation*}
    F(U_\epsilon(0,u) \times J) \subseteq (\bm{\delta}_A^{\phi})^{-1}(J) \quad \textrm{and} \quad G \circ F = \bm{1}_{U_\epsilon(0,u) \times J},
\end{equation*}
it follows from Lemma \ref{lem: elementary diff} that $ G $ and consequently $ \bm{\xi}_A $ is differentiable at $ r u $. Since $ \arl{A}{\phi}(0, \nabla \phi(u)) \geq \reach^\phi(A) > r $ by \cite[Lemma 4.16]{kolasinski2021regularity}, we infer that $ \bm{\xi}^\phi_A$ is differentiable at $ s \nabla \phi(u) $ for every $ 0 < s < \arl{A}{\phi}(0, \nabla \phi(u))$. Therefore $(0, \nabla \phi(u)) \in \widetilde{N}^\phi(A)$. Since $ f $ is twice differentiable at $ 0$, it follows that $ 0 \in \partial^v_+ A $ and $(0, \nabla \phi(u)) \in \widetilde{N}^\phi_n(A)$ by Lemma \ref{lem suff condition for finite curvature}\ref{lem suff condition to finite curvature 1}.
\end{proof}

Subsequently, we prefer to write $C$ for a set of positive reach. Theorem \ref{prop:p1} motivates the following definition.

\begin{Definition}\label{def: pointwise mean curvature}
Suppose $ C $ is a set of positive reach, $ a \in \partial^v C $, $ N(C,a) = \{u\} $  and $ f : a + u^\perp \rightarrow \mathbf{R}$ is a semiconvex function locally representing $ C $ as in Theorem \ref{prop:p1}. Then $ a $ is said to be an \emph{Alexandrov point of $ C $} if  $ f $ is twice differentiable at $ a $. Moreover, if $ \phi $ is a uniformly convex $ \mathcal{C}^2$-norm and $ k \in\{ 1, \ldots , n\} $, then the \emph{pointwise $ k $-th $ \phi $-mean curvature} of $C$ at $a$ is defined by \index{O5@$\bm{h}^\phi_{C,k}$}
\begin{equation*}
    \bm{h}^\phi_{C,k}(a) = S_k( \Der(\nabla \phi \circ \nu)(a)),
\end{equation*}
where $ \nu : a + u^\perp \rightarrow \mathbf{S}^n  $ is a map differentiable at $ a $ such that $ \nu(b) \in N(\epi(f), b - f(b)u) $ for $ b \in a + u^\perp $ and $S_k( \Der(\nabla \phi \circ \nu)(a))$ is the $k$-th elementary symmetric function of the eigenvalues (counted with multiplicities) of the endomorphism $\Der(\nabla \phi \circ \nu)(a)$ on $u^\perp $.

We denote the set of Alexandrov points \index{B6@$\mathcal{A}(C)$} of $ C $ by $ \mathcal{A}(C)$.
\end{Definition}

\begin{Remark}
If $ C $ is a convex body, then this notion of an Alexandrov point coincides with the classical notion of a normal boundary point of $ C $; see  \cite[Notes for Sections 1.5 and  2.6]{MR3155183} for further background  information.
\end{Remark}

\begin{Corollary}\label{cor: postive reach and viscosity boundary}
Suppose $ C \subseteq \mathbf{R}^{n+1}$ is a set of positive reach and $ \phi $ is a uniformly convex $ \mathcal{C}^2$-norm. Then the following statements hold.
\begin{enumerate}[{\rm (a)}]
    \item\label{cor: postive reach and viscosity boundary 1} $ \mathcal{A}(C) = \bm{p}(\widetilde{N}^\phi_n(C) ) \cap \partial^v C = \bm{p}(\widetilde{N}^\phi(C)) \cap \partial^v_+ C $ and
    \begin{equation*}
        \bm{H}^\phi_{C, k}(a, \eta) = \bm{h}^\phi_{C,k}(a) \qquad \textrm{for  $ a \in \mathcal{A}(C)$ and $ N^\phi(C,a)= \{\eta \}$.}
    \end{equation*}
    \item\label{cor: postive reach and viscosity boundary 2} $ \mathcal{H}^n(\partial^v C \setminus \mathcal{A}(C))=0$ and
    \begin{equation}\label{cor: postive reach and viscosity boundary 2 eq}
        \mathcal{P}^\phi(C) =\int_{\partial ^v_+ C}\phi(\bm{n}(C,a))\, d\mathcal{H}^n(a) = \int_{N^\phi(C)| \partial^v_+C}J^\phi_C(a,\eta) \phi(\bm{n}^\phi(\eta))\, d\mathcal{H}^{n}(a,\eta).
    \end{equation}
    \item\label{cor: postive reach and viscosity boundary 3} If $\mathcal{L}^{n+1}(C) < \infty $ then
    \begin{equation*}
        \Int(C) \neq \varnothing \quad \iff \quad \mathcal{L}^{n+1}(C) > 0 \quad \iff \mathcal{H}^n(\partial^v_+C) > 0.
    \end{equation*}
    \item\label{cor: postive reach and viscosity boundary 4} $ \mathcal{H}^n\big[ \bm{p}(\widetilde{N}^\phi_n(C)) \setminus \mathcal{A}(C) \big]=0 $ if and only if $ \mathcal{H}^n(\partial C \setminus \partial^v C) =0$.
\end{enumerate}
\end{Corollary}

\begin{proof}
The assertions in \ref{cor: postive reach and viscosity boundary 1} follow from Theorem \ref{lem: graphical representability}, Theorem \ref{prop:p1} and Lemma \ref{lem suff condition for finite curvature}. Using \ref{cor: postive reach and viscosity boundary 1} we infer
\begin{equation*}
    \partial^v C \setminus \mathcal{A}(C) = \partial^v C \setminus \bm{p}(\widetilde{N}^\phi_n(C))\subseteq     \bm{p}(N^\phi(C))\setminus
    \bm{p}(\widetilde{N}^\phi_n(C))\subseteq \bm{p}\left(N^\phi(C)\setminus\widetilde{N}^\phi_n(C)\right),
\end{equation*}
whence we obtain $ \mathcal{H}^n(\partial^v C \setminus \mathcal{A}(C))=0$ from Lemma \ref{lem: exterior normal basic properties closed} \ref{lem: exterior normal basic properties closed: c}. Since $ \partial C = \bm{p}(N^\phi(C)) $  we obtain the first equality in \eqref{cor: postive reach and viscosity boundary 2 eq} from Remark \ref{rmk: viscosity boundary} and the second equality by combining Remark \ref{rmk: curvature positive reach}, Lemma \ref{lem: tangent of normal bundle} and the coarea formula.

We prove \ref{cor: postive reach and viscosity boundary 3}. Clearly, $\Int(C) \neq \varnothing$ implies $ \mathcal{L}^{n+1}(C) > 0 $. Let us assume $\mathcal{L}^{n+1}(C) > 0$. Then it follows from Lemma
\ref{lem:finite perimeter and positive reach}  that $\mathcal{H}^n(\partial^m C)>0$. Since $\partial^v_+ C\subseteq\partial^m C\subseteq \partial^v C$ by Remark \ref{rmk: viscosity boundary} and $\mathcal{H}^n(\partial^v C\setminus\partial^v_+ C)=0$ by \ref{cor: postive reach and viscosity boundary 2}, it follows that $\mathcal{H}^n(\partial^v_+ C)>0$. It is again clear that $\mathcal{H}^n(\partial^v_+ C)>0$ implies $\Int(C) \neq \varnothing$.

Finally, it follows by \ref{cor: postive reach and viscosity boundary 1} that
\begin{equation*}
    \partial C \setminus \partial^v C \subseteq \big[\partial C \setminus \bm{p}(\widetilde{N}^\phi_n(C))\big] \cup \big[ \bm{p}(\widetilde{N}^\phi_n(C)) \setminus \partial^v C \big] = \big[\partial C \setminus \bm{p}(\widetilde{N}^\phi_n(C))\big] \cup \big[  \bm{p}(\widetilde{N}^\phi_n(C)) \setminus \mathcal{A}(C) \big]
\end{equation*}
and, since $ \mathcal{H}^n\big[\partial C \setminus \bm{p}(\widetilde{N}^\phi_n(C))\big] =0$  by Lemma
\ref{lem: exterior normal basic properties closed} (c), we obtain that $ \mathcal{H}^n\big[ \bm{p}(\widetilde{N}^\phi_n(C)) \setminus \mathcal{A}(C) \big]=0 $ implies that $ \mathcal{H}^n(\partial C \setminus \partial^v C) =0$.
On the other hand, if $ \mathcal{H}^n(\partial C \setminus \partial^v C) =0$, then Lemma \ref{lem: exterior normal basic properties closed} (c) implies again that
 $ \mathcal{H}^n\big[ \bm{p}(\widetilde{N}^\phi_n(C)) \setminus \partial^v C \big]=0 $, and hence $ \mathcal{H}^n\big[ \bm{p}(\widetilde{N}^\phi_n(C)) \setminus \mathcal{A}(C) \big]=0 $ follows from part (a).
\end{proof}

\begin{Remark}
Suppose $ C $ is a closed convex set with $\Int(C)\neq\varnothing$. Then $C^{(n)}\setminus \partial^vC=\varnothing$ and, recalling that $ \bm{p}(\widetilde{N}^\phi_n(C)) \subseteq C^{(n)}$  by Lemma \ref{lem: exterior normal basic properties closed}\ref{lem: exterior normal basic properties closed: a}, we infer from Corollary \ref{cor: postive reach and viscosity boundary} that
\begin{equation*}
   \mathcal{A}(C) = \bm{p}(\widetilde{N}_n^\phi(C)).
\end{equation*}
\end{Remark}

\subsection{Lower-bounded pointwise mean curvature and bubbling}\label{sec:5.3}

After some preparations, we will deduce the Heintze-Karcher inequality for sets of positive reach from the more general version for arbitrary closed sets.

\begin{Lemma}\label{lem: convex body and its complementary}
Suppose $ C \subset\mathbf{R}^{n+1} $ is a set of positive reach with $ \Int(C) \neq \varnothing $, $ K = \mathbf{R}^{n+1} \setminus \Int(C) $ and
\begin{equation*}
\iota : \mathbf{R}^{n+1} \times \mathbf{R}^{n+1} \rightarrow \mathbf{R}^{n+1} \times \mathbf{R}^{n+1}
\end{equation*}
is the linear map defined by $ \iota(a,\eta) = (a, -\eta) $ for   $(a,\eta)\in \mathbf{R}^{n+1} \times \mathbf{R}^{n+1}$.

Then the following statements hold.
\begin{enumerate}[{\rm (a)}]
\item\label{lem: convex body and its complementary 1} $ \bm{p}(N^\phi(K)) = \partial^v_+C = \partial^v_+ K $ and $ N^\phi(K,a) = \{-\nabla \phi(\bm{n}(C,a))\} = - N^\phi(C,a) $  for  $ a \in \bm{p}(N^\phi(K)) $.
 \item\label{lem: convex body and its complementary 2} $ \widetilde{N}^{\phi}_n(K) = \widetilde{N}^\phi(K)$ and  $ \mathcal{H}^n(N^\phi(K)|S)=0$ for  every $ S \subseteq \mathbf{R}^{n+1}$ with $ \mathcal{H}^n(S)=0$.
\item\label{lem: convex body and its complementary 3} $\kappa^\phi_{K,i}(a, \eta) = - \kappa^\phi_{C,n+1-i}(a, -\eta) $ for $ \mathcal{H}^n $ a.e.\ $ (a,\eta) \in N^\phi(K) $ and $ i = 1,\ldots , n $.
\item\label{lem: convex body and its complementary 5} $ \bm{H}^\phi_{K,1}(a,\eta) = -\bm{H}^\phi_{C,1}(a,-\eta) $ for $ \mathcal{H}^n $ a.e.\ $(a,\eta)\in N^\phi(K) $.
\item\label{lem: convex body and its complementary 4} $J^\phi_K(a,\eta) = \ap J_n^{N^\phi(K)}\iota(a,\eta)\, J^\phi_C(a, -\eta)$ for $ \mathcal{H}^n $ a.e.\ $(a,\eta)\in N^\phi(K) $.
\end{enumerate}
	\end{Lemma}

\begin{proof}
The statement in \ref{lem: convex body and its complementary 1} readily follows from Remark \ref{rmk: viscosity boundary}.

For the statement in \ref{lem: convex body and its complementary 2}  we combine \ref{lem: convex body and its complementary 1} and Lemma \ref{lem suff condition for finite curvature}\ref{lem suff condition to finite curvature 1} to infer that $ \widetilde{N}^{\phi}_n(K) = \widetilde{N}^\phi(K)$. Moreover if $ \mathcal{H}^n(S) =0$, then we can combine Lemma \ref{coarea} (first applied on $N_s$ for some $s>0$) with Lemma \ref{lem: tangent of normal bundle} to see that
\begin{equation*}
    \int_{\widetilde{N}^{\phi}_n(K)|S} \ap J_n^{N^\phi(K)}\bm{p}(a,u) \, d\mathcal{H}^n(a,u) = 0
\end{equation*}
and $ \ap J_n^{N^\phi(K)}\bm{p}(a,u) >0 $ for $ \mathcal{H}^n$ a.e.\ $(a,u) \in \widetilde{N}^{\phi}_n(K) $. We conclude that $ \mathcal{H}^n(\widetilde{N}^\phi(K)|S) =0 $ provided that $ \mathcal{H}^n(S) =0$. Since $ \mathcal{H}^n(N^\phi(K) \setminus \widetilde{N}^\phi(K)) =0$, we obtain \ref{lem: convex body and its complementary 2}.

Next we prove \ref{lem: convex body and its complementary 3}. We define the open set $ U = \{ y \in \mathbf{R}^{n+1}: 0 < \bm{\delta}^\phi_K(y) < \reach^\phi(C)   \} $. For  $ 0 < \lambda \leq 1 $, $ y \in U $ and $ a \in \bm{\xi}^\phi_K(y) $ we notice that
\begin{equation*}
	a + \lambda\bm{\delta}^\phi_K(y)\frac{a-y}{\bm{\delta}^\phi_K(y)} = (1+\lambda)a-\lambda y,
\end{equation*}
\begin{equation*}
	N^\phi(K,a)= \bigg\{ \frac{y-a}{\bm{\delta}^\phi_K(y)}\bigg\}, \qquad N^\phi(C,a)= \bigg\{ \frac{a-y}{\bm{\delta}^\phi_K(y)}\bigg\},
\end{equation*}
\begin{equation*}
	\bm{\nu}^\phi_C((1+\lambda)a-\lambda y) = \frac{a-y}{\bm{\delta}^\phi_K(y)}, \qquad -\bm{\nu}^\phi_C((1+\lambda)a-\lambda y) \in \bm{\nu}^\phi_K(y).
\end{equation*}
We infer that
\begin{equation}\label{lem: convex body and its complementary: 5}
	\bm{\nu}^\phi_K(y) = \{ - \bm{\nu}^\phi_C((1+\lambda)a - \lambda y)  : a \in \bm{\xi}^\phi_K(y) \} \qquad \textrm{for $ y \in U $ and $ 0 < \lambda \leq 1 $.}
\end{equation}
Define  $ S = \bm{p}\big( \iota(N^\phi(K)) \setminus \widetilde{N}^\phi(C)\big) $ and notice that $ \mathcal{H}^n(S) =0 $ by Remark \ref{rem: principla curvatures basic rem}. It follows from \ref{lem: convex body and its complementary 2} that $ \mathcal{H}^n(N^\phi(K)|S) =0 $. Fix now $(a,\eta)\in \widetilde{N}^\phi_n(K) $ with $ a \notin S $, $ 0 < r < \inf\{\bm{r}^\phi_K(a,\eta),\reach^\phi(C)\} $ and, noting that $ 1 - r \chi^\phi_{K,i}(a+r\eta) > 0 $ for $ i = 1, \ldots, n $,  we select  $ 0 < \lambda \leq 1 $ so that
\begin{equation*}
\chi^\phi_{K,i}(a+ r\eta) < \frac{1}{(1+\lambda)r} \qquad \textrm{for $ i = 1, \ldots , n $.}
\end{equation*}
Since $ a \notin S $, then $(a, -\eta)\in \widetilde{N}^\phi(C) $ and $ \bm{\nu}^\phi_C $ is differentiable at $ a - t\eta $ for every   $ 0 < t < \reach^\phi(C) $. Since $a + r\eta \in U $ and $ \bm{\nu}^\phi_C $ is differentiable at $ (1+\lambda) a - \lambda(a+r\eta) = a - \lambda r\eta $, we differentiate at $ a + r\eta $ the equality in \eqref{lem: convex body and its complementary: 5}, and thus we get
\begin{equation*}
\Der \bm{\nu}^\phi_K(a+ r\eta) = - \Der \bm{\nu}^\phi_C(a- \lambda r\eta) \circ ((1+\lambda)\Der \bm{\xi}^\phi_K(a+r\eta) - \textrm{Id}_{\mathbf{R}^{n+1}}).
\end{equation*}
If $ \tau_1, \ldots , \tau_n $ form a basis of $ \Tan(\partial \mathcal{W}^\phi, \eta)  $ such that $\Der \bm{\nu}^\phi_K(a+ r\eta)(\tau_i) = \chi^\phi_{K,i}(a+r\eta)\tau_i $ for  $ i = 1, \ldots , n $, then we infer
\begin{equation*}
\Der \bm{\nu}^\phi_C(a- \lambda r\eta)(\tau_i) = \frac{\chi^\phi_{K,i}(a+r\eta)}{(1+\lambda)r\chi^\phi_{K,i}(a+r\eta) -1} \tau_i \qquad \textrm{for $ i = 1, \ldots , n $}.
\end{equation*}
Note that $ \Tan(\partial \mathcal{W}^\phi_1, -\eta) = \Tan(\partial \mathcal{W}^\phi_1, \eta) $ and recall that $(1+\lambda)r\chi^\phi_{K,i}(a+r\eta) -1 < 0 $ for  $  i = 1, \ldots  n $. Hence, we conclude that
\begin{equation*}
\chi^\phi_{C, n+1-i}(a- \lambda r\eta) = \frac{\chi^\phi_{K,i}(a+r\eta)}{(1+\lambda)r\chi^\phi_{K,i}(a+r\eta) -1}
\end{equation*}
for $ i = 1, \ldots , n $. Therefore,
\begin{equation*}
	\kappa^\phi_{C, n+1-i}(a,-\eta) = \frac{\chi^\phi_{C, n+1-i}(a- \lambda r\eta)}{1 - \lambda r \chi^\phi_{C, n+1-i}(a- \lambda r\eta)} = \frac{\chi^\phi_{K,i}(a+ r\eta)}{r \chi^\phi_{K,i}(a+ r\eta) -1} = - \kappa^\phi_{K,i}(a,\eta)
\end{equation*}
for $ i = 1, \ldots, n $.

To prove \ref{lem: convex body and its complementary 5} we use \ref{lem: convex body and its complementary 2} and   \ref{lem: convex body and its complementary 3}, which yields that
\begin{equation*}
\bm{H}^\phi_{K,1}(a,\eta) = \sum_{i=1}^n \kappa^\phi_{K,i}(a,\eta) = - \sum_{i=1}^n \kappa^\phi_{C,i}(a,-\eta) = - \bm{H}^\phi_{C,1}(a,-\eta)
\end{equation*}
for $ \mathcal{H}^n $ a.e.\ $(a,\eta)\in \widetilde{N}^\phi_n(K) $.

Finally, we prove \ref{lem: convex body and its complementary 4}. Let $ \tau_1, \ldots , \tau_n, \zeta_1, \ldots , \zeta_n $ be $ \mathcal{H}^n \restrict N^\phi(K) $-measurable functions satisfying the hypothesis of Lemma \ref{lem: tangent of normal bundle}. Noting \ref{lem: convex body and its complementary 2} and  Lemma \ref{lem: existence of curvatures}, we observe that the proof of \ref{lem: convex body and its complementary 3}  shows that
\begin{equation*}
\Der \bm{\nu}^\phi_C(a - t\eta)(\tau_i(a,\eta)) = \rchi^\phi_{C, n+1-i}(a-t\eta) \tau_i(a,\eta)
\end{equation*}
for $ \mathcal{H}^n $ a.e.\ $(a,u)\in N^\phi(K) $ and   $ 0 < t < \reach^\phi(C) $. Since $ \kappa^\phi_{C,i}(a, -\eta) < \infty $ for $ \mathcal{H}^n $ a.e.\ $ (a,\eta)\in N^\phi(K) $ by \ref{lem: convex body and its complementary 2} and \ref{lem: convex body and its complementary 3}, we infer that
\begin{equation*}
J^\phi_C(a,-\eta) = \frac{|\tau_1(a,\eta)\wedge \ldots \wedge \tau_n(a,\eta)|}{|\iota(\zeta_1(a,\eta))\wedge \ldots \wedge \iota(\zeta_n(a,\eta))|}
\end{equation*}
for $\mathcal{H}^n$ a.e.\ $(a,\eta)\in N^\phi(K) $. Since
\begin{equation*}
\ap J_n^{N^\phi(K)}\iota(a,\eta) = \frac{|\iota(\zeta_1(a,\eta))\wedge \ldots \wedge \iota(\zeta_n(a,\eta))|}{|\zeta_1(a,\eta)\wedge \ldots \wedge \zeta_n(a,\eta)|},
\end{equation*}
the equation in \ref{lem: convex body and its complementary 4} follows.
\end{proof}

\begin{Remark}\label{rem:5.13}
The second statement in Lemma \ref{lem: convex body and its complementary}  \ref{lem: convex body and its complementary 2} can also be obtained as follows.  Let $0<s<\reach(C)$. From   \ref{lem: convex body and its complementary 1} we get
$$
N^\phi(K)|S=\bigcup_{\ell\in\N}\left\{(x,-\nabla\phi(\bm{n}(C,x))):x\in S\cap X_{\frac{s}{4\ell},s}(C)\right\}.
$$
The assertion now follows from Remark \ref{rem:lipnorm}.
\end{Remark}

\begin{Remark} We also outline an alternative argument for Lemma \ref{lem: convex body and its complementary}  \ref{lem: convex body and its complementary 3}. First, we obtain that for $\mathcal{H}^n$ a.e. $(a,\eta)\in \widetilde{N}^\phi_n(K)$ also $(a,-\eta)\in \widetilde{N}^\phi (C)$ and
$$
T=\Tan^n(\mathcal{H}^n\restrict \widetilde{N}^\phi_n(K),(a,\eta))=\text{lin}\{\zeta_1^K(a,\eta),\ldots,\zeta_n^K(a,\eta)\}=
\text{lin}\{\zeta_1^C(a,\eta),\ldots,\zeta_n^C(a,\eta)\},
$$
where
$$
\zeta_i^K(a,\eta)=\left(\tau_i^K(a,\eta),\kappa_{K,i}^\phi(a,\eta)\tau_i^K(a,\eta)\right)\qquad \text{for $i=1,\ldots,n$},
$$
the linearly independent vectors $\tau_1^K(a,\eta),\ldots,\tau_n^K(a,\eta)$ span an $n$-dimensional linear subspace $V$ of $\R^{n+1}$ and $\kappa_{K,i}^\phi(a,\eta)\in\R$, and where
$$
\zeta_i^C(a,\eta)=\begin{cases}
    \left(\tau_i^C(a,-\eta),-\kappa_{C,i}^\phi(a,-\eta)\tau_i^C(a,-\eta)\right),&\text{ if }\kappa_{C,i}^\phi(a,-\eta)<\infty,\\[1.5ex]
    \left(0,-\kappa_{C,i}^\phi(a,-\eta)\tau_i^C(a,-\eta)\right),&\text{ if }\kappa_{C,i}^\phi(a,-\eta)=\infty,
\end{cases}
$$
with linearly independent vectors $\tau_1^C(a,-\eta),\ldots,\tau_n^C(a,-\eta)$ which span an $n$-dimensional linear subspace $V'$ of $\R^{n+1}$.

Since the number of curvatures which are infinite equals the dimension of the kernel of the image of the linear map $\bm{p}|T$, a comparison of the two representations of $T$ shows that $\kappa_{C,i}^\phi(a,-\eta)<\infty$ for $i=1,\ldots,n$. Hence  $V=V'$ and $\bm{p}|T$ is an injective linear map. Therefore the linear map $L:V\to V$ with $L=\bm{q}\circ\bm{p}^{-1}$ is well defined and its eigenvalues are $\kappa_{K,i}^\phi(a,\eta)$  with corresponding eigenvectors $\tau_i^K(a,\eta)$, but also $-\kappa_{C,i}^\phi(a,-\eta)$ with corresponding eigenvectors $\tau_i^C(a,-\eta)$ for $i=1,\ldots,n$. This implies the assertion.
\end{Remark}

 We can now state the Heintze--Karcher inequality for sets of positive reach in the following form. Recall that $ \mathcal{H}^n(\partial^v C \setminus \mathcal{A}(C)) =0$ by Corollary \ref{cor: postive reach and viscosity boundary}.

 \begin{Theorem}\label{theo: heintze karcher positive reach}
Suppose $ \varnothing\neq C \subseteq \mathbf{R}^{n+1} $ is a set of positive reach with finite volume and assume that
\begin{equation*}
\bm{h}^\phi_{C,1}(a) \geq 0  \quad \textrm{  for $ \mathcal{H}^n\  \text{a.e.}\ a \in \mathcal{A}(C) $}.
\end{equation*}
Then
\begin{equation}\label{theo: heintze karcher positive reach: inequality}
	(n+1)\mathcal{L}^{n+1}(C) \leq n	\int_{\partial^v C} \frac{\phi(\bm{n}(C,a))}{\bm{h}^\phi_{C,1}(a)}\, d\mathcal{H}^n(a).
\end{equation}
	If $ \Int(C) \neq \varnothing $, equality holds in \eqref{theo: heintze karcher positive reach: inequality} and there exists $ q < \infty $ so that $\bm{h}^\phi_{C,1}(a)\leq q $ for $ \mathcal{H}^n $ a.e.\ $a\in \mathcal{A}(C) $, then there are  $ N \in\mathbb{N} $, $ c_1, \ldots , c_N \in \mathbf{R}^{n+1} $ and $ \rho_1, \ldots , \rho_N \geq \frac{n}{q} $ such that
\begin{equation*}
\Int(C) = \bigcup_{i=1}^N  \Int(c_i + \rho_i \mathcal{W}^\phi), \qquad \dist^\phi\big(c_i + \rho_i \mathcal{W}^\phi, c_j + \rho_j \mathcal{W}^\phi\big) \geq 2\reach^\phi(C) \quad \textrm{for  $ i \neq j $.}
\end{equation*}
\end{Theorem}
\begin{proof}
We assume $ \Int(C) \neq \varnothing $ (otherwise there is nothing to prove) and we define $ K = \mathbf{R}^{n+1} \setminus \Int(C)$. Note that $\mathcal{L}^{n+1}(\partial C)=0$ and $\iota(N^\phi(K))=N^\phi(C)|\partial ^v_+C$. By Lemma \ref{lem: convex body and its complementary} and the assumption, we infer that
\begin{equation*}
  \bm{H}^\phi_{K,1}(a,\eta) =   -\bm{H}^\phi_{C,1}(a,-\eta) \leq 0 \quad \textrm{for $ \mathcal{H}^n $ a.e.\ $(a,\eta) \in N^\phi(K)$.}
\end{equation*}
Therefore, applying Theorem \ref{theo: heintze karcher}, Lemma \ref{lem: convex body and its complementary} and the coarea formula, we obtain
\begin{flalign*}
(n+1)\mathcal{L}^{n+1}(C) & \leq n \int_{N^\phi(K)} J^\phi_K(a,\eta) \, \frac{\phi(\bm{n}^\phi(\eta))}{|\bm{H}^\phi_{K,1}(a,\eta)|}\, d\mathcal{H}^n(a,\eta)\\
& = n \int_{N^\phi(K)} \ap J^{N^\phi(K)}_n \iota (a,\eta)\, J^\phi_C(a,-\eta) \, \frac{\phi(\bm{n}^\phi(\eta))}{\bm{H}^\phi_{C,1}(a,-\eta)}\, d\mathcal{H}^n(a,\eta)\\
& = n \int_{N^\phi(C)| \partial^v_+C}  J^\phi_C(a,\eta) \, \frac{\phi(\bm{n}^\phi(\eta))}{\bm{H}^\phi_{C,1}(a,\eta)}\, d\mathcal{H}^n(a,\eta).
\end{flalign*}
Since $\phi(\bm{n}^\phi(\eta))=\phi(\bm{n}^\phi(\nabla \phi(\bm{n}(C,a))))=\phi(\bm{n}(C,a))$ for $(a,\eta)\in N^\phi(C)|\partial^v_+C$, recalling  Remark \ref{rmk: curvature positive reach} and Corollary \ref{cor: postive reach and viscosity boundary}, we apply coarea formula in combination with Lemma \ref{lem: tangent of normal bundle} to obtain
$$
\int_{N^\phi(C)| \partial^v_+C}  J^\phi_C(a,\eta) \, \frac{\phi(\bm{n}^\phi(\eta))}{\bm{H}^\phi_{C,1}(a,\eta)}\, d\mathcal{H}^n(a,\eta)=\int_{ \partial^v C}  \frac{\phi(\bm{n}(C,a))}{\bm{h}^\phi_{C,1}(a)}\, d\mathcal{H}^n(a),
$$
which yields the first part of the assertion of the theorem.

Assume now that $ \bm{h}^\phi_{C,1}(a) \leq q $ for $ \mathcal{H}^n $ a.e.\ $ a \in \mathcal{A}(C) $ and $ \Int(C) \neq \varnothing $. Combining Corollary  \ref{cor: postive reach and viscosity boundary} with Lemma \ref{lem: convex body and its complementary} we get that
\begin{equation*}
 -\bm{H}^\phi_{K,1}(a, \eta) =  \bm{H}^\phi_{C,1}(a, -\eta) = \bm{h}^\phi_{C,1}(a) \leq q \qquad \textrm{for $ \mathcal{H}^n $ a.e.\ $ (a, \eta) \in N^\phi(K) $.}
\end{equation*}
Therefore if the equality holds in \eqref{theo: heintze karcher positive reach: inequality} then  the conclusion follows from the characterization provided by Theorem \ref{theo: heintze karcher}.
\end{proof}

\medskip

From Theorem \ref{theo: heintze karcher positive reach} we obtain a geometric rigidity result for a set $C$ of positive reach with positive and finite volume under the assumption of a sharp lower bound on the pointwise $\phi$ mean-curvature at almost all points in $\partial^v C $. For the set $ C $ in the next theorem we notice that  $ \mathcal{P}^\phi(C) > 0$ and $ \Int(C) \neq \varnothing $  by Corollary \ref{cor: postive reach and viscosity boundary}.

\begin{Corollary}\label{theo: lower bound}
Suppose $ C \subset \mathbf{R}^{n+1} $ is a set of positive reach with  finite and positive volume and define $ \rho =\frac{(n+1)\mathcal{L}^{n+1}(C)}{\mathcal{P}^\phi(C)} $. Assume that
\begin{equation}\label{eq:zeps}
	\bm{h}^\phi_{C,1}(a)\geq
	\frac{n}{\rho}\qquad \textrm{for $ \mathcal{H}^n $ a.e.\ $a\in \mathcal{A}(C) $.}
\end{equation}
Then there exist $ N \in\mathbb{N} $ and $ c_1, \ldots , c_N \in \mathbf{R}^{n+1} $ such that
\begin{equation*}
\Int(C)= \bigcup_{i=1}^N \Int(c_i + \rho \mathcal{W}^\phi), \quad \dist^\phi\big(c_i + \rho \mathcal{W}^\phi, c_j + \rho  \mathcal{W}^\phi\big) \geq 2\reach^\phi(C) \quad \textrm{for $ i \neq j $.}
\end{equation*}
\end{Corollary}

\begin{proof}
Notice that there is at least one point $a\in\partial^v_+C$ with  $\bm{h}^\phi_{C,1}(a)<\infty$, hence we obtain $\rho >0$. Therefore  $0 < \mathcal{P}^\phi(C)<\infty$.

For  $ \epsilon > 0 $ we set
\begin{equation*}
	Z_\epsilon = \left\{ a\in  \mathcal{A}(C)  :\bm{h}^\phi_{C,1}(a)\geq (1+ \epsilon) \frac{n}{\rho} \right\}.
\end{equation*}
We claim that $ \mathcal{H}^n(Z_\epsilon) =0 $ for $ \epsilon > 0 $. Suppose that $ \mathcal{H}^n(Z_\epsilon) > 0 $ for some $\epsilon>0$. Then we deduce
\begin{align*}
	&	n	\int_{ \partial^v C}  \frac{\phi(\bm{n}(C,a))}{\bm{h}^\phi_{C,1}(a)}\, d\mathcal{H}^n(a)
	  = 	n	\int_{ \partial^v C\setminus Z_\epsilon}  \frac{\phi(\bm{n}(C,a))}{\bm{h}^\phi_{C,1}(a)}\, d\mathcal{H}^n(a)
	+	n	\int_{ Z_\epsilon}  \frac{\phi(\bm{n}(C,a))}{\bm{h}^\phi_{C,1}(a)}\, d\mathcal{H}^n(a)\\
	&\qquad\le \rho \int_{ \partial^v C\setminus Z_\epsilon} \phi(\bm{n}(C,a))\, d\mathcal{H}^n(a)+
	(1+\epsilon)^{-1}\rho  \int_{Z_\epsilon} \phi(\bm{n}(C,a))\, d\mathcal{H}^n(a)\\
	&\qquad<\rho\mathcal{P}^\phi(C)=(n+1)\mathcal{L}^{n+1}(C),
\end{align*}
where we used \eqref{eq:zeps} on $\partial^v C\setminus Z_\epsilon$     and the lower bound for $\bm{h}^\phi_{C,1}(a)$ on $Z_\epsilon$. This contradicts  the inequality in Theorem \ref{theo: heintze karcher positive reach} and thus proves the claim.

Since $ \mathcal{H}^n(Z_\epsilon) =0 $ for  $ \epsilon > 0 $, we infer that
\begin{equation*}
	\bm{h}^\phi_{C,1}(a) = \frac{n\mathcal{P}^\phi(C)}{(n+1)\mathcal{L}^{n+1}(C)}\qquad \textrm{for $ \mathcal{H}^n $ a.e.\ $a\in \partial^v C $},
\end{equation*}
whence we infer  that
$$
n	\int_{\partial^v C} \frac{\phi(\bm{n}(C,a))}{\bm{h}^\phi_{C,1}(a)}\, d\mathcal{H}^n(a)
 =(n+1)\mathcal{L}^{n+1}(C) ,
$$
thus \eqref{theo: heintze karcher positive reach: inequality} holds with equality.
We obtain now the conclusion of the theorem by employing the second part of Theorem \ref{theo: heintze karcher positive reach}.
\end{proof}

\begin{Remark}\label{spherical caps}
	Corollary \ref{theo: lower bound} is sharp already in the special isotropic case and for  convex bodies. In fact, if we consider the union of two congruent proper antipodal spherical caps of the unit sphere, we obtain a convex body $ K $ whose $ k $-th mean curvature on the smooth part of its boundary is constant and smaller than $\frac{\mathcal{H}^n(\partial K)}{(n+1)\mathcal{L}^{n+1}(K)}\binom{n }{ k}$. We provide the details for completeness. Let $ P \in \mathbf{G}(n+1,n) $ and $ \eta \in P^\perp $ with $ | \eta | = 1 $. For every $ 0 < \epsilon < 1 $ we define
	\begin{equation*}
		\Sigma^+_\epsilon = \mathbf{S}^{n} \cap \{ x : x \bullet \eta \geq \epsilon   \}, \quad 	\Sigma^-_\epsilon = \mathbf{S}^{n} \cap \{ x : x \bullet \eta \leq- \epsilon   \},
	\end{equation*}
	\begin{equation*}
		P^+_\epsilon = \mathbf{B}(0,1) \cap \{ x : x \bullet \eta = \epsilon   \}, \quad 	P^-_\epsilon = \mathbf{B}(0,1) \cap \{ x : x \bullet \eta =- \epsilon   \},
	\end{equation*}
	and we denote by $ K^+_\epsilon $ and $ K^-_\epsilon $ the convex bodies enclosed by $ \Sigma^+_\epsilon \cup P^+_\epsilon $ and $ \Sigma^-_\epsilon \cup P^-_\epsilon $ respectively. Then we define
	\begin{equation*}
		K_\epsilon = \{ x - \epsilon \eta : x \in K^+_{\epsilon}    \} \cup \{x + \epsilon \eta: x \in K^-_\epsilon \}.
	\end{equation*}
	Let $ X(x) = x $ for every $ x \in \mathbf{R}^{n+1} $. Since $ K^+_\epsilon $ is a set of finite perimeter, we denote by $ \eta_\epsilon $ the exterior unit normal and we compute by means of the divergence theorem \cite[Gauss--Green Theorem 4.5.6]{MR0257325}  (alternatively by noting that $K_\epsilon=\conv(\{o\}\cup \Sigma_\epsilon^+)\setminus \conv(\{o\}\cup P_\epsilon) $, where $\conv$ denotes the convex hull operator)
	\begin{flalign*}
		(n+1)\mathcal{L}^{n+1}(K^+_\epsilon)  &=  \int_{K^+_\epsilon} \mathrm{div} X \, d\mathcal{L}^{n+1} \\
		& = \int_{\Sigma^+_\epsilon \cup P^+_\epsilon} \eta_\epsilon(x) \bullet X(x)\, d\mathcal{H}^n(x) = \mathcal{H}^n(\Sigma^+_\epsilon) - \epsilon\mathcal{H}^n(P^+_\epsilon).
	\end{flalign*}
	We conclude that
	\begin{equation*}
		\frac{\mathcal{H}^n(\partial K_\epsilon)}{(n+1)\mathcal{L}^{n+1}(K_\epsilon)} = \frac{\mathcal{H}^n(\Sigma^+_\epsilon)}{(n+1)\mathcal{L}^{n+1}(K^+_\epsilon)} = 1 + \epsilon  \frac{\mathcal{H}^n(P_\epsilon^+)}{(n+1)\mathcal{L}^{n+1}(K^+_\epsilon)} > 1
	\end{equation*}
	for  $ 0 < \epsilon < 1 $. Finally, we notice that the $ k $-th mean curvature of $ K_\epsilon $ equals $ \binom{n }{ k} $ on the smooth part of  $ \partial K_\epsilon $.
\end{Remark}

\section{Curvature measures and soap bubbles}\label{Section: positive reach}

\subsection{Curvature measures and Minkowski formulae}
In the following, we write $C$ for a non-empty set with positive reach in $\R^{n+1}$.
For sets with positive reach, the Steiner formula simplifies in the following way (also in the anisotropic setting).

\begin{Corollary}[Anisotropic Steiner formula for sets of positive reach]\label{theo: Steiner positive reach}
Let $ \varnothing\neq C \subset \mathbf{R}^{n+1} $ be a set of positive reach. Let $ \varphi : N^\phi(C) \rightarrow \mathbf{R} $ be a bounded Borel function with compact support.  Then
$$
	 \int_{\{x \in\R^{n+1}: 0 < \bm{\delta}^\phi_C(x) \leq \rho\}} (\varphi \circ \bm{\psi}^\phi_C)\, d\mathcal{L}^{n+1} \notag \\
	   = \sum_{i=0}^n \frac{\rho^{i+1}}{i+1} \int_{N^\phi(C)}\phi(\bm{n}^\phi(\eta))\, J^\phi_C(a,\eta)\, \bm{H}^\phi_{C,i}(a,\eta)\;\varphi(a,u)\,  d\mathcal{H}^n(a,\eta)
$$
for $ 0 < \rho < \reach^\phi(C) $.
\end{Corollary}
\begin{proof}
	The assertion is a straightforward consequence of Theorem \ref{theo: Steiner closed}.
\end{proof}

We can now introduce the generalized curvature measures of a set of positive reach with respect to $ \phi $. These are real-valued Radon measures (see \cite{MR0849863,MR3932153,Hug99,MR1782274} and the references cited there).
\begin{Definition}
Let $ \varnothing\neq C \subset \mathbf{R}^{n+1} $ be a set of positive reach and $ m \in\{0, \ldots , n\} $. The \emph{$m$-th generalized curvature measure of $ C $ with respect to $ \phi $} is the real-valued Radon measure $\Theta^\phi_{m}(C,\cdot) $ on $ \mathbf{R}^{n+1}\times \mathbf{R}^{n+1} $ such that \index{Q2@$\Theta^\phi_{m}(C,\cdot)$}
\begin{equation*}
	\Theta^\phi_{m}(C,B) = \frac{1}{n-m+1}\int_{ N^\phi(C)\cap B} \phi(\bm{n}^\phi(\eta))\,J^\phi_C(a,\eta)\, \bm{H}^\phi_{C,n-m}(a,\eta)\; d\mathcal{H}^n(a,\eta)
\end{equation*}
for any bounded Borel subset $ B \subset \mathbf{R}^{n+1} \times \mathbf{R}^{n+1}$.   Moreover, we set \index{Q2@$\mathcal{V}^\phi_m(C)$}
\begin{equation*}
\mathcal{V}^\phi_m(C) = \Theta^\phi_m(C, N^\phi(C))  \qquad \textrm{for $ m \in \{ 0, \ldots , n\} $.}
\end{equation*}
\end{Definition}

\begin{Remark}
Let $ C \subseteq \mathbf{R}^{n+1} $ be a set of positive reach and $ m\in\{0, \ldots , n\} $. The \emph{$m$-th curvature measure of $ C $ with respect to $ \phi $} is the real-valued Radon measure \index{Q1@$\mathcal{C}^\phi_{m}(C,\cdot)$} $\mathcal{C}^\phi_{m}(C,\cdot) $ on $ \mathbf{R}^{n+1}$ such that
\begin{equation*}
	\mathcal{C}^\phi_{m}(C,B) = \Theta^\phi_m(C, B \times \mathbf{R}^{n+1})
\end{equation*}
for any bounded Borel subset $ B \subset \mathbf{R}^{n+1} $.
\end{Remark}

\medskip

\begin{Lemma}\label{lem: abs continuity c.m.}
Suppose $ C $ is a set of positive reach, $ m \in \{0, \ldots , n\}$ and let $ \nu : C^{(n)} \rightarrow \mathbf{S}^n $ be a Borel map such that $ \nu(a) \in N(C,a) $ for every $ a \in C^{(n)}$. Let $\eta(a)=\nabla\phi(\nu(a))$. Then
\begin{flalign*}
&(n-m+1)\Theta^\phi_{m}(C, B \cap \widetilde{N}^\phi_n(C)) \\
&\qquad= \int_{\mathcal{A}(C)}\mathbf{1}_B(a,\eta(a))\phi(\nu(a))\, \bm{h}^\phi_{C,n-m}(a)\, d\mathcal{H}^n (a) \\
&\qquad \quad + \int_{ C^{(n)}\setminus \partial^v C} \phi(\nu(a)) \big[ \mathbf{1}_B(a,\eta(a))\bm{H}^\phi_{C,n-m}(a, \eta(a)) +  \mathbf{1}_B(a,-\eta(a))\bm{H}^\phi_{C,n-m}(a, -\eta(a)) \big]\, d\mathcal{H}^n (a)
\end{flalign*}
for every Borel set $ B \subseteq N^\phi(C)$.
\end{Lemma}

\begin{proof}
Let $ B \subseteq N^\phi(C) $ be a Borel set. Combining Lemma \ref{lem: tangent of normal bundle} and the coarea formula, we get
\begin{equation*}
    (n-m+1)\Theta^\phi_{m}(C, B \cap \widetilde{N}^\phi_n(C)) = \int_{\bm{p}(\widetilde{N}^\phi_n(C))}\int_{N^\phi(C,a)}\mathbf{1}_B(a,\eta)\phi(\bm{n}^\phi(\eta))\, \bm{H}^\phi_{C,n-m}(a, \eta)\, d\mathcal{H}^0(\eta)\, d\mathcal{H}^n (a).
\end{equation*}
Since $ \bm{n}^\phi(\nabla \phi(\nu(a))) = \nu(a)$ for $ a \in C^{(n)}  $, the argument is completed by applying Lemma \ref{lem: exterior normal basic properties closed} \ref{lem: exterior normal basic properties closed: c} and Corollary \ref{cor: postive reach and viscosity boundary}, since
$ N^\phi(C,a) = \{\pm \nabla \phi(\nu(a))\}$ for  $ a \in C^{(n)} \setminus \partial^v C $.
\end{proof}

\begin{Remark}
The Lebesgue decomposition of the curvature measure  $\mathcal{C}^\phi_{m}(C,\cdot)$ with respect to $\mathcal{H}^n\llcorner \partial C$ is given by $\Theta^\phi_{m}(C, (\cdot\times\R^{n+1}) \cap \widetilde{N}^\phi_n(C))$ (the absolutely continuous part) and $\Theta^\phi_{m}(C, (\cdot\times\R^{n+1})  \setminus\widetilde{N}^\phi_n(C))$ (the singular part). It follows from Lemma \ref{lem: exterior normal basic properties closed} \ref{lem: exterior normal basic properties closed: e} that these parts are indeed singular with respect to each other. For the absolutely continuous part, Lemma \ref{lem: abs continuity c.m.} yields an explicit description. In the case of convex bodies  where $C^{(n)}\setminus \partial^vC=\emptyset$, a corresponding analysis can be found in \cite{MR1654685} in the isotropic framework.
\end{Remark}

\medskip


\medskip

The following lemma extends \cite[Lemma 2.1]{MR140907} from Euclidean curvature measures of convex bodies to generalized curvature measures with respect to a $C^2$-norm $\phi$ and sets with positive reach (compare also \cite[Section 3]{MR1742247}).  The non-negative Radon measure $ |\Theta^\phi_m(C,\cdot)\llcorner(A\times \partial \mathcal{W}^\phi)|  $ in the next lemma is the total variation of the real-valued Radon measure $ \Theta^\phi_m(C,\cdot)\llcorner(A\times \partial \mathcal{W}^\phi)  $; see \cite[Definition 1.4]{MR1857292}.

\begin{Lemma}\label{Lem:abs}
Let $\varnothing\neq C\subset\R^{n+1}$ be a set of positive reach. Let $A\subset\R^{n+1}$ be a $\mathcal{H}^m$ measurable set and $m\in\{0,\ldots,n\}$. Then there is a non-negative constant $c$, depending only on $n,\phi$, such that
$$
|\Theta^\phi_m(C,\cdot)\llcorner(A\times \partial \mathcal{W}^\phi)|\le c\cdot \mathcal{H}^m(A).
$$
\end{Lemma}

\begin{proof} For the proof, one can assume that $\mathcal{H}^m(A)<\infty$. An application of Theorem \ref{thm:disint} to the positive and the negative part of
$\Theta^\phi_m(C,\cdot)\llcorner(A\times \partial \mathcal{W}^\phi)$ then yields the assertion.
\end{proof}

The following lemma is now an immediate consequence of Lemma \ref{Lem:abs}. We do not include the case $m=n$ in the statement of the lemma, since in this case the hypothesis is always satisfied by Lemma \ref{lem: exterior normal basic properties closed} \ref{lem: exterior normal basic properties closed: c} and the conclusion holds essentially by definition; see Remark \ref{remcurvdef}.

\begin{Lemma}\label{lem:thetared}
Let $\varnothing\neq C \subset\R^{n+1}$ be a set of positive reach. Let $ m \in \{0, \ldots, n-1\}$ and assume that
\begin{equation*}
    \mathcal{H}^m[\bm{p}(\widetilde{N}^\phi(C) \setminus \widetilde{N}_n^\phi(C))]=0.
\end{equation*}
Then
\begin{equation*}
(n-m+1)\cdot    \Theta^\phi_{m}(C, B) = \int_{\widetilde{N}_n^\phi(C) \cap B}\phi(\bm{n}^\phi(\eta))J^\phi_C(a,\eta) \bm{H}^\phi_{C,n-m}(a,\eta)\, d\mathcal{H}^n(a,\eta)
\end{equation*}
for every Borel set $ B \subseteq N^\phi(C)$.
\end{Lemma}

\begin{proof}
An application of  Lemma \ref{Lem:abs} with $ A = \bm{p}(\widetilde{N}^\phi(C) \setminus \widetilde{N}_n^\phi(C))$ yields that
$$ \Theta^\phi_m(C,\cdot)\llcorner (\bm{p}( \widetilde{N}^\phi(C) \setminus  \widetilde{N}_n^\phi(C))\times\partial \mathcal{W}^\phi) =0,
$$
and hence $\Theta^\phi_m(C,B\cap  \widetilde{N}^\phi(C) \setminus \widetilde{N}_n^\phi(C) )=0$,  which is the desired conclusion.
\end{proof}

We now prove the anisotropic Minkowski formulae for sets of positive reach. The case of convex bodies has been treated in a different way in \cite{Hug99}.

\begin{Theorem}\label{theo: minkowski formula}
If $\varnothing\neq C \subset\R^{n+1}$ is a set of positive reach
  with finite volume and $ r \in\{ 1, \ldots , n \}$, then
\begin{flalign*}
& (n-r+1)\int_{N^\phi(C)} \phi(\bm{n}^\phi(\eta))\, J^\phi_C(a,\eta)\, \bm{H}^\phi_{C, r-1}(a,\eta)\, d\mathcal{H}^n(a,\eta) \\
& \qquad  = r\int_{N^\phi(C)} [a \bullet \bm{n}^\phi(\eta)]J^\phi_C(a,\eta)\, \bm{H}^\phi_{C,r}(a,\eta)\, d\mathcal{H}^n(a,\eta)
\end{flalign*}
and
\begin{equation*}
\int_{N^\phi(C)}a \bullet \bm{n}^\phi(\eta)\,J^\phi_C(a,\eta)\, \bm{H}^\phi_{C,0}(a,\eta) \, d\mathcal{H}^n(a,\eta) = (n+1)\mathcal{L}^{n+1}(C).
\end{equation*}
\end{Theorem}

\begin{proof}
We set $ u(x) = \frac{\nabla \bm{\delta}^\phi_C(x)}{|\nabla \bm{\delta}^\phi_C(x)|} $ for $ x \in \Unp^\phi(C) $ and  notice that for $ 0 < \rho < \reach^\phi(C)$ the set $ B^\phi(C, \rho) $ is a domain with $ \mathcal{C}^{1,1}$-boundary  $ \partial B^\phi(C, \rho) = S^\phi(C, \rho) $  whose exterior unit normal is given by $ u|\partial B^\phi(C, \rho) $. Moreover, for  $ 0 < \rho < \reach^\phi(C) $ the map $ f_\rho : N^\phi(C) \rightarrow S^\phi(C,\rho) $ defined by
\begin{equation*}
f_\rho(a,\eta) = a + \rho \eta \qquad \textrm{for $(a,\eta)\in N^\phi(C) $}
\end{equation*}
is a bi-lipschitz homeomorphism by Remark \ref{rem: sets of positive reach}. We observe (see proof of Theorem \ref{theo: Steiner closed}) that
\begin{equation*}
J_n^{N^\phi(C)}f_\rho(a,\eta) =J^\phi_C(a,\eta) \sum_{m=0}^n \rho^{n-m}\, \bm{H}^\phi_{C,n-m}(a,\eta)
\end{equation*}
for $ \mathcal{H}^n $ a.e.\ $(a,\eta)\in N^\phi(C) $. We set
\begin{equation*}
	I_m(C) = \int_{N^\phi(C)}a \bullet \bm{n}^\phi(\eta)\,J^\phi_C(a,\eta)\, \bm{H}^\phi_{C,n-m}(a,\eta) \, d\mathcal{H}^n(a,\eta) \qquad \textrm{for $ m =0, \ldots , n $.}
\end{equation*}
The divergence theorem and Remark \ref{rem: tangent level sets and Wulff shapes} yield
\begin{flalign*}
(n+1)\mathcal{L}^{n+1}(B^\phi(C, \rho)) & = \int_{S^\phi(C,\rho)} x \bullet u(x)\, d\mathcal{H}^n(x)\\
& = \int_{N^\phi(C)}[(a+ \rho \eta)\bullet u(a+ \rho \eta)] \, J_n^{N^\phi(C)}f_\rho(a,\eta)\, d\mathcal{H}^n(a,\eta)\\
& = \sum_{m=0}^n \rho^{n-m}I_m(C) +\sum_{m=0}^n (n-m+1)\rho^{n-m +1}\mathcal{V}^\phi_m(C)
\end{flalign*}
for $ 0 < \rho < \reach^\phi(C) $. Employing the Steiner formula \ref{theo: Steiner positive reach}, we get
\begin{equation*}
(n+1)\mathcal{L}^{n+1}(B^\phi(C, \rho)) = (n+1)\mathcal{L}^{n+1}(C) + (n+1)\sum_{m=0}^n \rho^{n-m+1}\mathcal{V}^\phi_m(C)
\end{equation*}
for $ 0 < \rho < \reach^\phi(C) $. Hence, we infer
\begin{equation*}
	\sum_{m=0}^{n-1}[I_m(C) - (m+1)\mathcal{V}^\phi_{m+1}(C)]\rho^{n-m} + I_n(C) - (n+1)\mathcal{L}^{n+1}(C) = 0
\end{equation*}
for  $ 0 < \rho < \reach^\phi(C) $. It follows that $ I_m(C) = (m+1)\mathcal{V}^\phi_{m+1}(C) $ for $ m =0, \ldots , n-1 $ and in addition we have $ I_n(C) = (n+1)\mathcal{L}^{n+1}(C) $.
\end{proof}

\subsection{The soap bubble theorem for sets of positive reach}
The following notion of $ k $-convexity generalizes the classical analogous notion used in the Euclidean setting to study isoperimetric-type inequalities for Querrmassintegrals (see \cite{MR1287239} or the more recent \cite{MR3107515}). Analogous concepts also arise in the context of elliptic differential operators (see \cite{MR1726702} and \cite{Salani} and the references given there to earlier work for instance by Caffarelli, Nirenberg, Spruck  ('85), Garding ('59), Ivochkina ('83, '85), Li ('90)).

\begin{Definition}\label{def:mc}
	Let $\varnothing\neq C \subset \mathbf{R}^{n+1} $ be a set of positive reach with $ \mathcal{P}^\phi(C) > 0$, and let $ r \in\{ 0, \ldots , n\} $. We say that $ C $ is \emph{ $ (r, \phi) $-mean convex } if
	\begin{equation}\label{eq mean convex 1}
		\bm{h}^\phi_{C,i}(a)\geq 0 \qquad \textrm{for $ \mathcal{H}^n $ a.e.\ $a\in \mathcal{A}(C) $ and $ i =1, \ldots , r-1 $}
	\end{equation}
	and
	\begin{equation}\label{eq mean convex 2}
		\bm{H}^\phi_{C,r}(a,u)\geq 0 \qquad \textrm{for $ \mathcal{H}^n $ a.e.\ $(a,u)\in {N}^\phi(C) $.}
	\end{equation}
\end{Definition}

\begin{Remark}
Suppose that $\varnothing\neq C \subset \mathbf{R}^{n+1} $ is a set of positive reach with finite volume. Then $ \mathcal{P}^\phi(C) > 0$ if and only if $ \mathcal{H}^n(\partial^v_+ C) > 0 $ by Corollary \ref{cor: postive reach and viscosity boundary}. This in turn is equivalent to $\Int(C)\neq\varnothing$.
\end{Remark}

\begin{Remark}
The $ (0, \phi)$-mean convex sets are all sets of positive reach  with non-empty interior, since by definition we have $ \bm{H}^\phi_{C,0} = \bm{1}_{\widetilde{N}^\phi_n(C)}\ge 0 $.  Moreover, if $ C $ is a set of  positive reach and positive perimeter such that $ \Theta^\phi_{n-k}(C. \cdot)$ are non-negative measures for all $ k \in \{1, \ldots , r\}$, then $ C $ is $ (r, \phi)$-mean convex.
\end{Remark}

Before we can add another remark, we need some preparations. Let $d_H^\phi$ denote the Hausdorff distance on the space of closed subsets of the metric space $(\R^{n+1},\phi^*)$. We say that a sequence of closed sets $C_i\subseteq\R^{n+1}$, $i\in\N$, converges to a closed set $C\subseteq\R^{n+1}$ as $i\to\infty$ if $d_H^\phi(C_i,C)\to 0$ as $i\to\infty$, which is equivalent to the uniform convergence of $\delta^\phi_{C_i}$ to $\delta^\phi_C$ as $i\to\infty$ on $\R^{n+1}$ (see \cite{MR0801600}). The following lemma is well known in the Euclidean setting (see  Theorem 4.13,  Remark 4.14 and Theorem 5.9 in \cite{MR0110078} and \cite[Section 3.1, pp.~7--9]{MR1846894}).

\begin{Lemma}\label{lem:reachvague}
Let $\varnothing\neq C_i\subset\R^{n+1}$ for $i\in\N$ be a sequence of closed sets converging to a closed set $C\subset\R^{n+1}$. Suppose there is a constant $\rho>0$ such that $\reach^\phi(C_i)\ge \rho$ for all $i\in\N$. Then $\reach^\phi(C)\ge \rho$ and for each $k\in\{0,\ldots,n\}$ the Radon measures $\Theta^\phi_k(C_i,\cdot)$ converge vaguely to the Radon measure $\Theta^\phi_k(C,\cdot)$ as $i\to\infty$.
\end{Lemma}

\begin{proof} Let $\rho_1\in (0,\rho)$. Under the assumptions of the lemma, we show that $\reach(C)\ge \rho_1$. Let $x\in\R^{n+1}\setminus C$ with $\delta^\phi_C(x)\le \rho_1$. There is some $i_1\in \N$ such that if $i\ge i_1$, then $0<\delta^\phi_{C_i}(x)\le \rho_2<\rho$, where $\rho_2:=\frac{1}{2}(\rho+\rho_1)$. Define $x_i:=\xi_{C_i}^\phi(x)+\rho_2\cdot\nu^\phi_{C_i}(x)$, hence $\delta^\phi_{C_i}(x_i)=\rho_2$, $\delta^\phi_{C_i}(x_i)=\delta^\phi_{C_i}(x)$ and  $U^\phi(x_i,\rho_2)\cap C_i=\varnothing$, since $\bm{r}^\phi_{C_i}(\xi_{C_i}^\phi(x),\nu_{C_i}^\phi(x))\ge \reach(C_i)$. By compactness, we can find an infinite subset $ I\subseteq\N$ such $\xi^\phi_{C_i}(x_i)\to\xi\in\partial C$, $x_i\to z$, $\delta^\phi_C(z)=\rho_2$,
$U^\phi(z,\rho_2)\cap C=\varnothing$ and $x\in (\xi,z)$, where $I\ni i\to\infty$. But then clearly $U^\phi(x,\rho_2)\cap C=\varnothing$, $x\in \Unp^\phi(C)$ and $\xi=\xi^\phi_C(x)$. This proves the first assertion.

In view of Corollary \ref{theo: Steiner positive reach}, it is sufficient to show that if $\varphi:\R^{n+1}\times\R^{n+1}\to\R$ is a continuous function with compact support and $t\in (0,\rho)$, then
$$
\int\mathbf{1}\{0<\delta^\phi_{C_i}(x)\le t\}\varphi(\xi^{\phi}_{C_i}(x),\nu^{\phi}_{C_i}(x))\, d\mathcal{L}^{n+1}(x)\to
\int\mathbf{1}\{0<\delta^\phi_{C}(x)\le t\}\varphi(\xi^{\phi}_{C}(x),\nu^{\phi}_{C}(x))\, d\mathcal{L}^{n+1}(x)
$$
as $i\to\infty$.

We consider an arbitrary point $x\in \R^{n+1}$.
If $0<\delta_C^\phi(x)<t$, then  also $0<\delta_{C_i}^\phi(x)<t$ if $i\in\N$ is large enough. Moreover, $\xi_{C_i}^\phi(x)\to \xi_C^\phi(x)$ and $\nu_{C_i}^\phi(x)\to \nu_C^\phi(x)$ as $i\to\infty$, which can be obtained by minor adjustments of the proof for \cite[Lemma 2.2]{MR1782274}. Therefore,
\begin{equation}\label{eq:pc}
\mathbf{1}\{0<\delta^\phi_{C_i}(x)\le t\}\varphi(\xi^{\phi}_{C_i}(x),\nu^{\phi}_{C_i}(x))\to
\mathbf{1}\{0<\delta^\phi_{C}(x)\le t\}\varphi(\xi^{\phi}_{C}(x),\nu^{\phi}_{C}(x))
\end{equation}
as $i\to\infty$.

If $\delta_C^\phi(x)>t$, then also $\delta_{C_i}^\phi(x)>t$ if $i\in\N$ is large enough, hence \eqref{eq:pc} holds trivially. The same is true if  $x\in \Int(C)$ which implies that $x\in C_i$ for all sufficiently large $i\in\N$.

Since $\mathcal{L}^{n+1}(\{x\in\R^{n+1}:\delta_C^\phi(x)=t\}\cup \partial C)=0$, the assertion follows from the dominated convergence theorem.
\end{proof}

The preceding lemma implies that a non-negativity condition closely related to Definition \ref{def:mc} is stable with respect to converging sequences of sets of positive reach, as described in the following corollary.

\begin{Corollary}\label{cor:reachvague}
Let $\rho>0$ be a fixed constant. Let $\varnothing\neq C_i\subset\R^{n+1}$ for $i\in\N$ be a sequence of closed sets with $\reach(C_i)\ge \rho>0$  converging to a closed set $C\subset\R^{n+1}$. Then the following statements hold.
\begin{enumerate}
    \item[{\rm (a)}] If $\Theta^\phi_k(C_j,\cdot)\ge 0$ for all $j\in\N$, then also $\Theta^\phi_k(C,\cdot)\ge 0$.
    \item[{\rm (b)}] If $C_j$ is $(r,\phi)$-mean convex and smooth for all $j\in\N$ (so that  $\Theta^\phi_{n-k}(C_j,\cdot)\ge 0$ holds for   $k=1,\ldots,r$), then $\Theta^\phi_{n-k}(C,\cdot)\ge 0$ for $k=1,\ldots,r$; hence, $C$ is $(r,\phi)$-mean convex.
\end{enumerate}
\end{Corollary}

\medskip

A major ingredient for the proof of the Alexandrov theorem for sets with positive reach is the next lemma.

\begin{Lemma}\label{lem: constant mean curvature properties}
Let $ \varnothing\neq C \subset\mathbf{R}^{n+1} $ be a set of positive reach with finite and positive volume. Let $ \lambda \in \mathbf{R} $ and   $ r\in\{ 1, \ldots , n\} $ be such that
\begin{equation}\label{eq:hypo}
\Theta^\phi_{n-r}(C, \cdot) = \lambda\, \Theta^\phi_n(C, \cdot).
\end{equation}
Then $\bm{H}^\phi_{C,r}(a,\eta) = 0$ for $ \mathcal{H}^n $ a.e.\ $(a,\eta)\in N^\phi(C) \setminus \widetilde{N}^\phi_n(C) $ and
\begin{equation*}
\bm{H}^\phi_{C,r}(a,\eta) = (r+1) \lambda =\frac{n-r+1}{(n+1)\mathcal{L}^{n+1}(C)}\mathcal{V}^\phi_{n-r+1}(C)  \qquad \textrm{for $ \mathcal{H}^n $ a.e.\ $(a,\eta)\in \widetilde{N}^\phi_n(C) $}.
\end{equation*}
Furthermore, the following two statements hold.
\begin{enumerate}
    \item[{\rm (a)}] \label{lem: constant mean curvature properties:p1}If $ \lambda \neq 0 $, $ r \geq 2$ and $ C $ is $ (r-1, \phi) $-mean convex, then
    \begin{equation*}
\infty> \frac{\bm{h}^\phi_{C,1}(a)}{\binom{n}{ 1}} \geq  \Bigg(\frac{\bm{h}^\phi_{C,2}(a)}{\binom{n}{ 2}}\Bigg)^{\frac{1}{2}} \geq \ldots \geq \Bigg(\frac{\bm{h}^\phi_{C,r}(a)}{{\binom{n}{ r}}}\Bigg)^{\frac{1}{r}}	\geq \frac{\mathcal{P}^\phi(C)}{(n+1)\mathcal{L}^{n+1}(C)}
\end{equation*}
for $ \mathcal{H}^n $ a.e.\ $ a \in \partial^v C $.
\item[{\rm (b)}] \label{lem: constant mean curvature properties:p2} If $ r = 1 $, then $\bm{H}^\phi_{C,1}(a,\eta) \geq\frac{n\mathcal{P}^\phi(C)}{(n+1)\mathcal{L}^{n+1}(C)}$ for $ \mathcal{H}^n $ a.e.\ $(a,\eta)\in \widetilde{N}^\phi_n(C) $.
\end{enumerate}
\end{Lemma}

\begin{proof}
From the equality $ \Theta^\phi_{n-r}(C, \cdot) = \lambda\,
\Theta^\phi_n(C, \cdot) $ we get
\begin{flalign*}
&\int_{B \cap \widetilde{N}^\phi_n(C)} \phi(\bm{n}^\phi(\eta))\,J^\phi_C(a,\eta)\, \Big(\frac{1}{r+1}\bm{H}^\phi_{C,r}(a,\eta) - \lambda\Big)\; d\mathcal{H}^n(a,\eta) \\
& \qquad + \frac{1}{r+1} \int_{(N^\phi(C) \setminus \widetilde{N}_n^\phi(C))\cap B} \phi(\bm{n}^\phi(\eta))\,J^\phi_C(a,\eta)\, \bm{H}^\phi_{C, r}(a,\eta)\; d\mathcal{H}^n(a,\eta) =0
\end{flalign*}
for every bounded Borel set $ B \subset N^\phi(C) $. Since $ J^\phi_C(a,\eta)\phi(\bm{n}^\phi(\eta)) > 0 $ for $ \mathcal{H}^n $ a.e.\ $ (a,\eta)\in N^\phi(C) $, we conclude that $\bm{H}^\phi_{C,r}(a,\eta) = 0$ for $ \mathcal{H}^n $ a.e.\ $(a,\eta)\in N^\phi(C) \setminus \widetilde{N}^\phi_n(C) $ and $\bm{H}^\phi_{C,r}(a,\eta) = (r+1) \lambda$ for $ \mathcal{H}^n $ a.e.\ $(a,\eta)\in \widetilde{N}^\phi_n(C) $. Noting that $ \bm{H}^\phi_{C,0} = \bm{1}_{\widetilde{N}^\phi_n(C)} $ and employing Theorem \ref{theo: minkowski formula}, we derive
\begin{flalign}\label{lem: constant mean curvature properties: eq1}
	& (n-r+1)\int_{N^\phi(C)} \phi(\bm{n}^\phi(\eta))\, J^\phi_C(a,\eta)\, \bm{H}^\phi_{C, r-1}(a,\eta)\, d\mathcal{H}^n(a,\eta) \notag \\
&\qquad = r \int_{N^\phi(C)} [a \bullet \bm{n}^\phi(\eta)]J^\phi_C(a,\eta)\, \bm{H}^\phi_{C,r}(a,\eta)\, d\mathcal{H}^n(a,\eta) \notag \\
& \qquad  = r(r+1)\lambda \int_{\widetilde{N}_n^\phi(C)} [a \bullet \bm{n}^\phi(\eta)]J^\phi_C(a,\eta)\, d\mathcal{H}^n(a,\eta) \notag \\
& \qquad  = r(r+1) \lambda (n+1)\mathcal{L}^{n+1}(C),
\end{flalign}
from which we infer that
\begin{equation}\label{eq:representation}
	\bm{H}^\phi_{C,r}(a,\eta) = \frac{n-r+1}{(n+1)\mathcal{L}^{n+1}(C)}\mathcal{V}^\phi_{n-r+1}(C) \qquad \textrm{for $ \mathcal{H}^n $ a.e.\ $(a,\eta)\in \widetilde{N}^\phi_n(C) $,}
\end{equation}
since $\mathcal{L}^{n+1}(C)\in (0,\infty)$. In particular, using Corollary  \ref{cor: postive reach and viscosity boundary}, we obtain
\begin{equation*}
    	\bm{h}^\phi_{C,r}(a) = \frac{n-r+1}{(n+1)\mathcal{L}^{n+1}(C)}\mathcal{V}^\phi_{n-r+1}(C) \qquad \textrm{for $ \mathcal{H}^n $ a.e.\ $a \in \partial^v C $.}
\end{equation*}
Notice that the asserted conclusion for $ r = 1 $ already follows from \eqref{eq:representation}, since $ \mathcal{V}^\phi_{n}(C) \geq \mathcal{P}^\phi(C)$.

We assume now $ r \geq 2 $, $ \lambda \neq 0 $ and that $ C $ is $ (r-1, \phi) $-mean convex.  The  non-negativity property of $\bm{H}^\phi_{C, r-1} $ in combination with \eqref{lem: constant mean curvature properties: eq1} implies that $ \lambda \geq 0$, that means $ \lambda > 0 $. Therefore $ \bm{h}^\phi_{C,r}(a)> 0 $ is satisfied for $ \mathcal{H}^n $ a.e. $ a \in \partial^v C $. Hence we can apply Lemma \ref{lem:Maclaurin} to conclude that
\begin{equation}\label{lem: constant mean curvature properties: Mclaurin inequality}
\frac{\bm{h}^\phi_{C,1}(a)}{\binom{n}{1}} \geq \ldots \geq
\left(\frac{\bm{h}^\phi_{C,r-1}(a)}{\binom{n}{r-1}}\right)^{\frac{1}{r-1}} \geq  \left(\frac{\bm{h}^\phi_{C,r}(a)}{\binom{n}{r}}\right)^{\frac{1}{r}}
=\left(\frac{(r+1)\lambda}{\binom{n}{r}}\right)^{\frac{1}{r}}
\end{equation}
for $ \mathcal{H}^n $ a.e.\ $ a \in \partial^v C $. Using again that $\bm{H}^\phi_{C,r-1}(a,\eta)\geq 0$ for $\mathcal{H}^n$ a.e. $(a,\eta)\in N^\phi(C)$ and the lower bound for $\bm{H}^\phi_{C,r-1}(a,\eta)$ from   \eqref{lem: constant mean curvature properties: Mclaurin inequality}, we get
\begin{flalign}\label{lem: constant mean curvature properties: eq2}
	&(n-r+1)\int_{N^\phi(C)} \phi(\bm{n}^\phi(\eta))\, J^\phi_C(a,\eta)\, \bm{H}^\phi_{C, r-1}(a,\eta)\, d\mathcal{H}^n(a,\eta) \notag \\
	& \qquad \geq (n-r+1) \int_{N^\phi(C)|\partial^v_+C} \phi(\bm{n}^\phi(\eta))\, J^\phi_C(a,\eta)\, \bm{H}^\phi_{C, r-1}(a,\eta)\, d\mathcal{H}^n(a,\eta)\notag  \\
	& \qquad \geq (n-r+1) [(r+1)\lambda]^{\frac{r-1}{r}}{\binom{n}{r}}^{\frac{1-r}{r}}{\binom{n}{r-1}}\int_{N^\phi(C)|\partial^v_+C} \phi(\bm{n}^\phi(\eta))\, J^\phi_C(a,\eta)\, d\mathcal{H}^n(a,\eta) \notag \\
	& \qquad =(n-r+1)[(r+1)\lambda]^{\frac{r-1}{r}} {\binom{n}{r}}^{\frac{1-r}{r}}{\binom{n}{r-1}}\mathcal{P}^\phi(C).
\end{flalign}
Combining \eqref{lem: constant mean curvature properties: eq1} and \eqref{lem: constant mean curvature properties: eq2}, noting that $ (n-r+1){\binom{n}{r}}^{-1}{\binom{n}{r-1}}\frac{1}{r} = 1 $ and recalling that $\lambda>0$, we conclude
\begin{equation*}
\infty>[\lambda (r+1)]^{\frac{1}{r}} \geq \frac{\mathcal{P}^\phi(C)}{(n+1)\mathcal{L}^{n+1}(C)} {\binom{n}{r}}^{\frac{1}{r}}.
\end{equation*}
Now the remaining assertion follows from \eqref{lem: constant mean curvature properties: Mclaurin inequality}.
\end{proof}

The Alexandrov theorem for sets of positive reach is now a corollary of our previous results.

\begin{Theorem}\label{theo: Alexandrov}
Let $ r\in\{ 1, \ldots , n \}$,  and let $ C \subset\mathbf{R}^{n+1} $ be an $(r-1, \phi)$-mean convex set with positive and  finite volume such that
\begin{equation}\label{eqmc}
	\Theta^\phi_{n-r}(C, \cdot) = \lambda\, \Theta^\phi_n(C, \cdot)\quad \textrm{for some $\lambda\in \R\setminus\{0\}$}.
\end{equation}

Then $\lambda>0$ and there exist a finite natural number $ N \geq 1 $ and $ c_1, \ldots , c_N \in \mathbf{R}^{n+1} $ such that
\begin{equation*}
\Int(C) = \bigcup_{i=1}^N \Int(c_i + \rho \mathcal{W}^\phi), \quad \dist^\phi\big(c_i + \rho \mathcal{W}^\phi, c_j + \rho  \mathcal{W}^\phi\big) \geq 2\reach^\phi(C) \quad \textrm{for  $ i \neq j $,}
\end{equation*}
where $ \rho$ satisfies the relations $ \binom{n}{r}^{\frac{1}{r}}\left(\lambda(r+1)\right)^{-\frac{1}{r}}= \rho  =\frac{(n+1)\mathcal{L}^{n+1}(C)}{\mathcal{P}^\phi(C)} $ and $\lambda>0$.

If $r=1$, then the same conclusion is obtained for any $\lambda\in\R$ and any set $C$ of positive reach with positive and finite volume.
\end{Theorem}

\begin{proof}
The assertion is implied by a combination of Lemma \ref{lem: constant mean curvature properties} and Corollary \ref{theo: lower bound}, in particular it follows that $\lambda>0$.

To check the equation $ \rho = \binom{n}{r}^{\frac{1}{r}}\left(\lambda(r+1)\right)^{-\frac{1}{r}} $, we first observes that $ \kappa^\phi_{\rho \mathcal{W}^\phi}(a, \eta) = \frac{1}{\rho} $ for $(a, \eta)\in N^\phi(\rho \mathcal{W}^\phi) $ (see for instance \cite[Corollary 2.33]{MR4160798}); since $ C $ is the disjoint union of $ N $ translated copies of $ \rho \mathcal{W}^\phi $, one infers that $ \bm{H}^\phi_{C,r}(a,\eta)= \binom{n}{r}\rho^{-r} $ for $(a,\eta) \in N^\phi(C) $. As $\bm{H}^\phi_{C,r}(a,\eta) = \lambda(r+1) $ for $(a, \eta) \in N^\phi(C) $ by Lemma \ref{lem: constant mean curvature properties}, we obtain the aforementioned equation.
\end{proof}

\medskip

\medskip

In Theorem \ref{theo: Alexandrov} we deal with sets which are non-convex, hence it is natural to state the proportionality assumption \eqref{eqmc} on the  $r$-th  $\phi$-mean curvatures in terms of generalized curvature measures. However, under a slightly more restrictive mean convexity assumption we also get the following variant of  Theorem \ref{theo: Alexandrov} in which a corresponding assumption on the proportionality of a curvature measure is imposed in \eqref{eqmcb}. The proof is based on a  modified version of Lemma \ref{lem: constant mean curvature properties}, whose proof only requires minor adjustments (and uses \eqref{eq: partition} in first part of the argument).

\begin{Theorem}\label{theo: Alexandrovmodifiedb}
Let $ r\in\{ 1, \ldots , n \}$, and let $ C \subset\mathbf{R}^{n+1} $ be an $(r-1, \phi)$-mean convex set with positive and  finite volume and such that $ \bm{H}^\phi_{C,r}(a,\eta)\ge 0 $ for $\mathcal{H}^n$ a.e. $(a,\eta)\in N^\phi(C)$. Assume that
\begin{equation}\label{eqmcb}
	\mathcal{C}^\phi_{n-r}(C, \cdot) = \lambda\, \mathcal{C}^\phi_n(C, \cdot)\quad\textrm{for some $\lambda>0$.}
\end{equation}
Then there exist a finite natural number $ N \geq 1 $ and $ c_1, \ldots , c_N \in \mathbf{R}^{n+1} $ such that
\begin{equation*}
\Int(C) = \bigcup_{i=1}^N \Int(c_i + \rho \mathcal{W}^\phi), \quad \dist^\phi\big(c_i + \rho \mathcal{W}^\phi, c_j + \rho  \mathcal{W}^\phi\big) \geq 2\reach^\phi(C) \quad \textrm{for  $ i \neq j $,}
\end{equation*}
where $ \rho $ is given as in Theorem \ref{theo: Alexandrov}.
\end{Theorem}

We conclude this section discussing the validity of the hypothesis of Theorem \ref{theo: Alexandrov} in terms of Alexandrov points and pointwise curvatures for a large subclass of sets of positive reach, namely those sets $ C $ for which $ \mathcal{H}^{n}(\partial C \setminus \partial^v C) =0  $. This class includes all convex bodies, and more generally  all closed sets that can be locally represented as the epigraph of a semiconvex function; see Lemma \ref{lem: semiconvex sets}. But it includes much more; indeed,  it is easy to construct sets of positive reach $ C $ for which $ \mathcal{H}^{n}(\partial C \setminus \partial^v C) =0  $, but the boundary is not a topological manifold (see \cite[Example 7.12]{MR3683461} or \cite[Example 1]{MR2443761}).

The hypotheses in the next statement should be seen in connection with the disjoint union displayed in \eqref{eq: partition}.

\begin{Lemma}\label{thm: alexandorv points and mean convexity}
If $ k \in \{1, \ldots , n\} $, $ \lambda \in \mathbf{R}$ and $\varnothing\neq C \subset   \mathbf{R}^{n+1}$ is a set of positive reach such that $ \mathcal{H}^{n}(\partial C \setminus \partial^v C) =0  $  and $ \mathcal{P}^\phi(C) > 0$, then the following two statements hold.
\begin{enumerate}[{\rm (a)}]
    \item  If $ \bm{h}^\phi_{C,k}(a) = \lambda $ for $ \mathcal{H}^n $ a.e.\ $ a \in\mathcal{A}(C)$ and $  \mathcal{H}^{n-k}\big[ \bm{p}(\widetilde{N}^\phi(C) \setminus \widetilde{N}^\phi_n(C))\big] =0  $, then
    \begin{equation*}
        \Theta^\phi_{n-k}(C, \cdot) = \lambda\, \Theta^\phi_{n}(C, \cdot).
    \end{equation*}
    \item If $ \bm{h}^\phi_{C,1}(a) \geq 0, \ldots , \bm{h}^\phi_{C, k-1}(a) \geq 0  $ for $ \mathcal{H}^n $ a.e.\ $ a \in \mathcal{A}(C) $ and $  \mathcal{H}^{n-k+1}\big[ \bm{p}(\widetilde{N}^\phi(C) \setminus \widetilde{N}^\phi_n(C))\big] =0  $, then $ C $ is $ (k-1,\phi) $-mean convex.
\end{enumerate}
\end{Lemma}

\begin{proof}
(a) Let $ B \subseteq N^\phi(C)$ be a Borel set. Noting that $  \Theta^\phi_{n}(C,B)  = \Theta^\phi_{n}\big(C,B \cap \widetilde{N}^\phi_n(C) \big) $ by definition of $ \bm{H}^\phi_{C,0}$, we can use  Lemma   \ref{lem:thetared}, Lemma \ref{lem: abs continuity c.m.} and Corollary \ref{cor: postive reach and viscosity boundary} (a), (b) together with $\eta(a)=\nabla\phi(\bm{n}(C,a))$ for $a\in \partial^vC$ to compute
\begin{flalign*}
    \Theta^\phi_{n-k}(C,B)  & = \Theta^\phi_{n-k}\big(C,B \cap \widetilde{N}^\phi_n(C) \big) \\
    & = \int_{ \partial^v C }\mathbf{1}_B(a,\eta(a))\phi(\bm{n}(C,a)) \bm{h}^\phi_{C, k}(a) \, d\mathcal{H}^n(a) \\
    & = \lambda  \int_{\partial^v C } \mathbf{1}_B(a,\eta(a))\phi(\bm{n}(C,a))   \, d\mathcal{H}^n(a)\\
    & = \lambda\, \Theta^\phi_n(C, B).
\end{flalign*}

(b) Arguing as in (a) we can compute
\begin{flalign*}
    \Theta^\phi_{n-k+1}(C,B)  & = \Theta^\phi_{n-k+1}\big(C,B \cap \widetilde{N}^\phi_n(C) \big) \\
     & = \int_{ \partial^v C }\mathbf{1}_B(a,\eta(a))\phi(\bm{n}(C,a)) \bm{h}^\phi_{C, k-1}(a) \, d\mathcal{H}^n(a)  \geq 0
\end{flalign*}
for every Borel set $ B \subseteq N^\phi(C)$. This means that $ \bm{H}^\phi_{C, k-1}(a, \eta) \geq 0 $ for $ \mathcal{H}^n $ a.e.\ $(a, \eta) \in N^\phi(C) $ and consequently $ C $ is $ (\phi, k-1) $-mean convex.
\end{proof}

Now with the help of Lemma \ref{thm: alexandorv points and mean convexity}, the following result can be easily deduced as a special case of Theorem \ref{theo: Alexandrov}.
\begin{Corollary}\label{Cor: Alexandrov}
Suppose $ k \in \{1, \ldots , n\} $, $ \lambda \in \mathbf{R} \setminus \{0\}$ and $\varnothing\neq C \subset   \mathbf{R}^{n+1}$ is a set of positive reach  with finite and positive volume such that
\begin{enumerate}
    \item[{\rm (1)}] $ \mathcal{H}^{n}(\partial C \setminus \partial^v C) =0  $ and  $ \mathcal{H}^{n-k}\big[ \bm{p}(\widetilde{N}^\phi(C) \setminus \widetilde{N}^\phi_n(C))\big] =0  $,
    \item[{\rm (2)}] $ \bm{h}^\phi_{C,k}(a) = \lambda $ for $ \mathcal{H}^n $ a.e.\ $ a \in\mathcal{A}(C)$,
    \item[{\rm (3)}] $ \bm{h}^\phi_{C,1}(a) \geq 0, \ldots , \bm{h}^\phi_{C, k-1}(a) \geq 0  $ for $ \mathcal{H}^n $ a.e.\ $ a \in \mathcal{A}(C) $.
\end{enumerate}
Then the conclusion of Theorem \ref{theo: Alexandrov} holds. If $ k = 1 $, then the same conclusion is  true for every $ \lambda \in \mathbf{R}$.
\end{Corollary}

\begin{Remark}
Corollary \ref{Cor: Alexandrov} includes as very special cases the soap bubble theorems of Alexandrov (\cite{MR0102114}), Korevaar--Ros (\cite{MR996826} and \cite{MR925120}) and He--Li--Ma--Ge (\cite{MR2514391}). In fact for connected and compact domains with $ \mathcal{C}^2$-boundary the hypothesis (c) of Corollary \ref{Cor: Alexandrov} can be easily deduced from the existence of an elliptic point, the continuity of the principal curvatures and the Garding theory on hyperbolic polynomials (see \cite[page 450]{MR925120} for further details). The continuity of the principal curvatures and the assumption of connectedness play a key role in this argument.
\end{Remark}

Recall the definition of epigraph from \eqref{eq: epigraph}.
\begin{Lemma}\label{lem: semiconvex sets}
Suppose $ C \subseteq \mathbf{R}^{n+1}$ is a compact set such that for every $ a \in \partial C $ there exists $ u \in \mathbf{S}^n$, $ \epsilon, \delta > 0 $ and a semiconvex function $ f : a + u^\perp \rightarrow \mathbf{R}$ such that
\begin{equation}\label{eq:locepi}
    \epi(f)\cap U_{\varepsilon,\delta}(a,u)=  C \cap U_{\varepsilon,\delta}(a,u).
\end{equation}
Then $ \reach(C) > 0$ and $ \mathcal{H}^n(\partial C \setminus \partial^v C) =0$.
\end{Lemma}

\begin{proof} For  $a\in\partial C$, let $\varepsilon(a),\delta(a)>0$ and the local representation in terms of a semiconvex function $f_a$ be as in \eqref{eq:locepi}. By \cite[Theorem 2.3]{MR0816398} we know that $\reach(\epi(f_a))\ge r(a)>0$. Define  $\rho(a)=\frac{1}{4}\min\{\varepsilon(a),\delta(a),r(a)\}$. Then  we have $U(a,\rho(a))\subseteq \Unp(C)$ for every $ a \in \partial C$. If $a\in\Int(C)$, then there is also a positive number $\rho(a)$ such that  $U(a,\rho(a))\subset\Int(C)\subseteq \Unp(C)$. Since the sets $U(a,\rho(a))$, for $a\in C$, are an open cover of the compact set $C$, we get a finite number of points $ a_1, \ldots , a_N \in C $ such that $ C \subseteq \bigcup_{i=1}^N U(a_i, \frac{\rho(a_i)}{2})$. Then for $0 < \tau < \inf\{\frac{\rho(a_1)}{2}, \ldots , \frac{\rho(a_N)}{2} \}$, it holds that for every $c\in  C$ there is some $a_i $ such that $U(c,\tau)\subseteq  U(a_i,\rho(a_i))\subseteq\Unp(C)$. This shows that $\reach(C)\ge \tau>0$.

Note that $\partial C=\bm{p}(N(C))$ (see e.g.~\cite[Corollary  4.12(a)]{MR3932153}). Since $\mathcal{H}^0(N(C,a))\neq 2$ for $a\in\partial C$ due to \eqref{eq:locepi}, it follows from Lemma \ref{lem: exterior normal basic properties closed}  (c) that  $\mathcal{H}^n(\bm{p}(N(C))\setminus \partial^v C)=0$, which gives the remaining assertion.
\end{proof}

\bigskip

\noindent
\textbf{Acknowledgements.} 
D. Hug was supported by DFG research grant HU 1874/5-1 (SPP 2265) and gratefully acknowledges support by
ICERM (Brown University).

\bigskip

\noindent
Daniel Hug, Karlsruhe Institute of Technology (KIT), Institute of Stochastics, D-76128 Karlsruhe, Germany, daniel.hug@kit.edu

\medskip

\noindent Mario Santilli,  Department of Information Engineering, Computer Science and Mathematics, Università degli Studi dell'Aquila,
  67100 L’Aquila, Italy, mario.santilli@univaq.it

\printindex

\end{document}